\documentclass[microtype]{gtpart}

\usepackage{pinlabel}
\usepackage[utf8]{inputenc}

\title[The Higher Morita Category of $\mathbb{E}_{n}$-Algebras]{The Higher Morita Category of $\mathbb{E}_{n}$-Algebras}

\author[R Haugseng]{Rune Haugseng}
\givenname{Rune}
\surname{Haugseng}
\address{Department of Mathematical Sciences\\
University of Copenhagen\\
Universitetsparken 5\\
2100 København Ø\\
Denmark}
\email{haugseng@math.ku.dk}
\urladdr{http://sites.google.com/site/runehaugseng}


\arxivreference{1412.8459}
\arxivpassword{sebng}

\volumenumber{}
\issuenumber{}
\publicationyear{}
\papernumber{}
\startpage{}
\endpage{}
\doi{}
\MR{}
\Zbl{}
\received{}
\revised{}
\accepted{}
\published{}
\publishedonline{}
\proposed{}
\seconded{}
\corresponding{}
\editor{}
\version{
}

\usepackage{amsmath,amssymb,amsthm}
\usepackage{url}

\usepackage[shortlabels]{enumitem}
\theoremstyle{theorem}
\newtheorem{thm}{Theorem}[section]
\newtheorem{lemma}[thm]{Lemma}
\newtheorem{propn}[thm]{Proposition}
\newtheorem{cor}[thm]{Corollary}
\newtheorem*{thm*}{Theorem}

\theoremstyle{definition}
\newtheorem{defn}[thm]{Definition}

\newtheorem{notation}[thm]{Notation}

\newtheorem{warning}[thm]{Warning}

\newtheorem{remark}[thm]{Remark}

\newtheorem{conjecture}[thm]{Conjecture}

\usepackage{eucal}

\DeclareSymbolFontAlphabet{\amsbb}{AMSb}
\makeatother
\usepackage{mbboard}
\renewcommand{\mathbb}[1]{\amsbb{#1}}

\newcommand{\blank}{\text{--}}

\newcommand{\defterm}[1]{\emph{#1}}
\newcommand{\isoto}{\xrightarrow{\sim}}
\newcommand{\isofrom}{\xleftarrow{\sim}}
\newcommand{\IFF}{if and only if}
\newcommand{\catname}[1]{\ensuremath{\text{\textup{#1}}}}
\newcommand{\txt}[1]{\ensuremath{\text{\textup{#1}}}}
\newcommand{\Set}{\catname{Set}}
\newcommand{\sSet}{\Set_{\Delta}}
\newcommand{\Cat}{\catname{Cat}}
\newcommand{\CatI}{\catname{Cat}_\infty}
\newcommand{\LCatI}{\widehat{\catname{Cat}}_\infty}

\newcommand{\Mod}{\catname{Mod}}

\newcommand{\Fun}{\txt{Fun}}

\newcommand{\Map}{\txt{Map}}
\newcommand{\Hom}{\txt{Hom}}

\newcommand{\op}{\txt{op}}
\newcommand{\icat}{$\infty$-category}
\newcommand{\icats}{$\infty$-categories}
\newcommand{\icatl}{$\infty$-categorical}

\newcommand{\xto}[1]{\xrightarrow{#1}}

\usepackage{tikz}
\usetikzlibrary{matrix,arrows}
\usepackage{tikz-cd}

\newcommand{\ltikzcd}[1]{{\setlength\mathsurround{0pt} \begin{tikzcd}#1\end{tikzcd}}}

\newcommand{\csquare}[8]{ %
\[ %
\begin{tikzpicture} %
\matrix (m) [matrix of math nodes,row sep=3em,column sep=2.5em,text height=1.5ex,text depth=0.25ex] %
{ #1 \pgfmatrixnextcell #2 \\ %
  #3 \pgfmatrixnextcell #4 \\ }; %
\path[->,font=\footnotesize] %
(m-1-1) edge node[auto] {$#5$} (m-1-2)%
(m-1-1) edge node[left] {$#6$} (m-2-1)%
(m-1-2) edge node[auto] {$#7$} (m-2-2)%
(m-2-1) edge node[below] {$#8$} (m-2-2);%
\end{tikzpicture}%
\]%
}

\newcommand{\nolabelcsquare}[4]{\csquare{#1}{#2}{#3}{#4}{}{}{}{}}

\newcommand{\liftcsquare}[9]{ %
\[ %
\begin{tikzpicture} %
\matrix (m) [matrix of math nodes,row sep=3em,column sep=2.5em,text height=1.5ex,text depth=0.25ex] %
{ #1 \pgfmatrixnextcell #2 \\ %
  #3 \pgfmatrixnextcell #4 \\ }; %
\path[->,font=\footnotesize] %
(m-1-1) edge node[auto] {$#5$} (m-1-2)%
(m-1-1) edge node[left] {$#6$} (m-2-1)%
(m-1-2) edge node[auto] {$#7$} (m-2-2)%
(m-2-1) edge node[below] {$#8$} (m-2-2);%
\path[->,dashed,font=\footnotesize](m-2-1) edge node[above] {$#9$} (m-1-2);
\end{tikzpicture}%
\]%
}

\newcommand{\opctriangle}[6]{ %
\[ %
\begin{tikzpicture} %
\matrix (m) [matrix of math nodes,row sep=3em,column sep=1.2em,text height=1.5ex,text depth=0.25ex] %
{  #1 \pgfmatrixnextcell \pgfmatrixnextcell #2 \\ %
  \pgfmatrixnextcell #3 \pgfmatrixnextcell \\ %
}; %
\path[->,font=\footnotesize] %
(m-1-1) edge node[above] {$#4$} (m-1-3)%
(m-1-1) edge node[below left] {$#5$} (m-2-2)%
(m-1-3) edge node[below right] {$#6$} (m-2-2);%
\end{tikzpicture}%
\]%
}

\newcommand{\id}{\txt{id}}
\DeclareMathOperator{\colimP}{colim}
\newcommand{\colim}{\mathop{\colimP}}
\newcommand{\simp}{\bbDelta}

\newcommand{\Alg}{\catname{Alg}}
\newcommand{\Opd}{\catname{Opd}}
\newcommand{\OpdI}{\Opd_{\infty}}
\newcommand{\iopd}{$\infty$-operad}
\newcommand{\iopds}{$\infty$-operads}

\newcommand{\nsiopd}{non-symmetric $\infty$-operad}
\newcommand{\nsiopds}{non-symmetric $\infty$-operads}
\newcommand{\gnsiopd}{generalized non-symmetric $\infty$-operad}
\newcommand{\gnsiopds}{generalized non-symmetric $\infty$-operads}

\newcommand{\RR}{\mathbb{R}}
\newcommand{\Seg}{\txt{Seg}}
\newcommand{\MonI}{\Mon_{\infty}}
\newcommand{\OpdIS}{\OpdI^{\Sigma}}
\newcommand{\AlgS}{\Alg^{\Sigma}}
\newcommand{\Mon}{\txt{Mon}}
\newcommand{\MonS}{\txt{Mon}^{\Sigma}}
\newcommand{\FUN}{\txt{FUN}}

\newcommand{\Dn}{\simp^{n}}
\newcommand{\Dnop}{\simp^{n,\op}}
\newcommand{\DnI}{\Dn_{/I}}
\newcommand{\DnIop}{\Dnop_{/I}}
\newcommand{\DnC}{\Dn_{/C}}
\newcommand{\DnCop}{(\DnC)^{\op}}

\newcommand{\LOpdI}{\widehat{\txt{Opd}}_{\infty}}
\newcommand{\OpdIDn}{\OpdI^{\Dn}}
\newcommand{\OpdIDng}{\OpdI^{\Dn,\txt{gen}}}

\newcommand{\LOpdIDng}{\LOpdI^{\Dn,\txt{gen}}}

\newcommand{\Lbrn}{\bbLambda_{/[n]}}
\newcommand{\Lbrnop}{\Lbrn^{\op}}
\newcommand{\Lbri}{\bbLambda_{/[i]}}
\newcommand{\Lbriop}{\Lbri^{\op}}
\newcommand{\Ln}{\bbLambda^{n}}
\newcommand{\Lnop}{\bbLambda^{n,\op}}
\newcommand{\LnI}{\Ln_{/I}}
\newcommand{\LnIop}{\Lnop_{/I}}

\newcommand{\AlgDn}{\catname{Alg}^{n}}
\newcommand{\LnIP}{\LnI[\Phi]}

\newcommand{\LIP}{\bbLambda_{/[n]}[\phi]}

\newcommand{\LnP}{\bbLambda_{/[n]}[\phi]}

\newcommand{\DnIAop}{(\DnI)^{\amalg,\op}}

\newcommand{\Cell}{\txt{Cell}}
\newcommand{\Celln}{\Cell^{n}}
\newcommand{\CellnI}{\Celln_{/I}}
\newcommand{\CellnIop}{\Cell^{n,\op}_{/I}}

\newcommand{\ALG}{\txt{ALG}}

\newcommand{\Dniopd}{$\Dn$-$\infty$-operad}
\newcommand{\Dniopds}{$\Dn$-$\infty$-operads}
\newcommand{\gDniopd}{generalized $\Dn$-$\infty$-operad}
\newcommand{\gDniopds}{generalized $\Dn$-$\infty$-operads}
\newcommand{\dniopd}{$\Dn$-$\infty$-operad}
\newcommand{\dniopds}{$\Dn$-$\infty$-operads}
\newcommand{\gdniopd}{generalized $\Dn$-$\infty$-operad}
\newcommand{\gdniopds}{generalized $\Dn$-$\infty$-operads}

\newcommand{\siopd}{symmetric \iopd{}}
\newcommand{\siopds}{symmetric \iopds{}}
\newcommand{\gsiopd}{generalized \siopd{}}
\newcommand{\gsiopds}{generalized \siopds{}}

\newcommand{\angled}[1]{\langle #1 \rangle}

\newcommand{\En}{\mathbb{E}_{n}}

\newcommand{\LMonI}{\widehat{\txt{Mon}}_{\infty}}
\newcommand{\LMonIDn}{\LMonI^{\Dn}}
\newcommand{\LMonIDnGR}{\LMonI^{\Dn,\txt{GRTP}}}
\newcommand{\LMonIGR}{\LMonI^{\txt{GRTP}}}
\newcommand{\LMonIDnGRtens}{\LMonI^{\Dn,\txt{GRTP},\times}}

\newcommand{\Gop}{\bbGamma^{\op}}

\newcommand{\Algn}{\Alg^{n}}

\newcommand{\falg}{\mathfrak{alg}}

\newcommand{\fAlg}{\mathfrak{Alg}}
\newcommand{\fAlgn}{\fAlg_{n}}
\newcommand{\fALG}{\mathfrak{ALG}}
\newcommand{\fALGn}{\fALG_{n}}
\newcommand{\ofALG}{\overline{\fALG}}
\newcommand{\ofALGn}{\ofALG_{n}}

\newcommand{\Dop}{\simp^{\op}}

\newcommand{\Algns}{\Alg^{1}}

\newcommand{\OpdInsg}{\OpdI^{\simp,\txt{gen}}}
\newcommand{\Bimod}{\txt{Bimod}}
\newcommand{\Ass}{\txt{Ass}}

\newcommand{\act}{\txt{act}}

\newcommand{\rcite}[2]{\cite[#2]{#1}}

\begin{document}

\begin{abstract}
  We introduce simple models for associative algebras and bimodules in
  the context of non-symmetric \iopds{}, and use these to construct an
  $(\infty,2)$-category of associative algebras, bimodules, and
  bimodule homomorphisms in a monoidal \icat{}. By working with
  \iopds{} over $\Dnop$ we iterate these definitions and generalize
  our construction to get an $(\infty,n+1)$-category of
  $\mathbb{E}_{n}$-algebras and iterated bimodules in an
  $\mathbb{E}_{n}$-monoidal \icat{}. Moreover, we show that if
  $\mathcal{C}$ is an $\mathbb{E}_{n+k}$-monoidal \icat{} then the
  $(\infty,n+1)$-category of $\mathbb{E}_{n}$-algebras in
  $\mathcal{C}$ has a natural $\mathbb{E}_{k}$-monoidal structure. We
  also identify the mapping $(\infty,n)$-categories between two
  $\mathbb{E}_{n}$-algebras, which allows us to define interesting
  non-connective deloopings of the Brauer space of a commutative ring
  spectrum.
\end{abstract}
\begin{asciiabstract}
  We introduce simple models for associative algebras and bimodules in
  the context of non-symmetric infinity-operads, and use these to
  construct an (infinity,2)-category of associative algebras,
  bimodules, and bimodule homomorphisms in a monoidal
  infinity-category. By working with infinity-operads over
  Delta^{n,op} we iterate these definitions and generalize our
  construction to get an (infinity,n+1)-category of E_n-algebras and
  iterated bimodules in an E_n-monoidal infinity-category. Moreover,
  we show that if C is an E_{n+k}-monoidal infinity-category then the
  (infinity,n+1)-category of E_n-algebras in C has a natural
  E_k-monoidal structure. We also identify the mapping
  (infinity,n)-categories between two E_n-algebras, which allows us to
  define interesting non-connective deloopings of the Brauer space of
  a commutative ring spectrum.
\end{asciiabstract}

\maketitle

\tableofcontents

\section{Introduction}
The goal of this paper is to construct higher categories of
$\mathbb{E}_{n}$-algebras and their iterated bimodules, using a
completely algebraic or combinatorial approach to these objects, and
establish some of their basic properties. Our construction is
motivated by the interesting connections of these higher categories to
topological quantum field theories, and a notion of ``higher Brauer
groups'' that can be extracted from them. We will discuss these
potential applications, both of which we intend to explore further in future
work, after summarizing the main results of the present paper.

\subsection{Summary of Results}
If $\mathbf{C}$ is a monoidal category, then the associative algebra
objects\footnote{Also commonly called associative monoids, but we will
  reserve the term monoid for the case when the tensor product in
  $\mathbf{C}$ is the Cartesian product.} in $\mathbf{C}$ and their
bimodules form a bicategory $\fAlg_{1}(\mathbf{C})$. More precisely,
this bicategory has
\begin{itemize}
\item associative algebras in $\mathbf{C}$ as objects,
\item $A$-$B$-bimodules in $\mathbf{C}$ as 1-morphisms from $A$ to $B$,
\item bimodule homomorphisms as 2-morphisms,
\end{itemize}
with composition of 1-morphisms given by taking tensor products: if
$M$ is an $A$-$B$-bimodule and $N$ is a $B$-$C$-bimodule then their
composite is $M \otimes_{B}N$ with its natural $A$-$C$-bimodule
structure. Moreover, if $\mathbf{C}$ is a symmetric monoidal category,
such as $\Mod_{R}$ for $R$ a commutative ring, then
$\fAlg_{1}(\mathbf{C})$ inherits a symmetric monoidal
structure.\footnote{Although it is intuitively clear that the tensor
  product on $\mathbf{C}$ induces such a symmetric monoidal structure,
  this seems to have been completely defined only quite recently by
  Shulman in \cite{ShulmanSymMonBicat}, following a construction of a
  braided monoidal structure by Garner and Gurski in
  \cite{GarnerGurskiTricat}. Considering the difficulty of even
  defining symmetric monoidal bicategories in full generality, this is
  perhaps not entirely unsurprising ---
  cf. \cite[\S 2.1]{SchommerPries2TQFT} for a discussion of the
  history of such definitions.}  When $R$ is a commutative ring, this
symmetric monoidal bicategory $\fAlg_{1}(\Mod_{R})$
organizes a wealth of interesting algebraic information --- for
example, two $R$-algebras are equivalent in $\fAlg_{1}(\Mod_{R})$ precisely
when they are Morita equivalent, i.e. have equivalent categories of
modules.

Since all the concepts involved have derived analogues, it is
reasonable to expect that there is a derived or higher-categorical
version of the bicategory $\fAlg_{1}(\Mod_{R})$, based on chain complexes of
$R$-modules up to quasi-isomorphism. More generally, it should be
possible to allow
$R$ to be a differential graded algebra --- or even a ring
spectrum, with chain complexes replaced by $R$-modules in spectra up
to stable weak equivalence.

In this paper we will indeed construct such generalizations of the
bicategory of algebras and bimodules. However, the coherence issues
that must be solved to define these seem intractable from the point of
view of classical (enriched) category theory. To avoid this problem,
we instead work in the setting of \icats{}. 

Roughly speaking, an \emph{\icat{}} (or \emph{$(\infty,1)$-category})
is a structure that has objects and morphisms like a category, but
also ``homotopies'' (or invertible 2-morphisms) between morphisms,
``homotopies between homotopies'' (or invertible 3-morphisms), and so
on. The morphisms can be composed, but the composition is not strictly
associative, only associative up to a coherent choice of (higher)
homotopies. Using homotopy theory there are a number of ways of making
this idea precise in such a way that one can actually work with the
resulting structures; we will make use of the theory of
\emph{quasicategories} as developed by Joyal and Lurie~\cite{HTT},
which is by far the best-developed variant.

Similarly, one can consider \emph{$(\infty,n)$-categories} for $n >
1$; these have $i$-morphisms for all $i$ that are required to be
invertible when $i > n$, and are thus the ``$\infty$-version'' of
$n$-categories. We will encounter them in the guise of Barwick's
\emph{$n$-fold Segal spaces} \cite{BarwickThesis}, which we will
review below in \S\ref{subsec:inftyn}.

In this higher-categorical setting there is a natural notion of a
\emph{monoidal} \icat{}, i.e. an \icat{} equipped with a tensor
product that is associative up to coherent homotopy. Our first main
result, which we will prove in \S\ref{sec:algbimod}, is a construction
of an $(\infty,2)$-category $\fAlg_{1}(\mathcal{C})$ of algebras,
bimodules, and bimodule homomorphisms in any monoidal \icat{}
$\mathcal{C}$ that satisfies some mild technical assumptions.

In the \icatl{} setting it is also natural to ask how this structure
extends to \emph{$\mathbb{E}_{n}$-algebras}. In the context of
ordinary categories, an object equipped with two compatible
associative multiplications is a commutative algebra. When we pass to
higher categories, however, this is no longer true. The most familiar
example of this phenomenon is iterated algebras in the 2-category of
categories --- if we consider associative algebras in the appropriate
2-categorical sense, these are monoidal categories; categories with
two compatible monoidal structures are then braided monoidal
categories, and ones with three or more monoidal structures are
symmetric monoidal categories. In general, objects with $k$ compatible
associative algebra structures in an $n$-category are commutative
algebras for $k > n$ --- this is a form of the Baez-Dolan
stabilization hypothesis\footnote{See \cite[Corollary 5.1.1.7]{HA}
  for a proof of this statement.}; in other words, in an $n$-category
compatible associative algebra structures give $n+1$ different
algebraic structures.  For an \icat{}, then, objects equipped with
multiple compatible multiplications give an infinite sequence of
algebraic structures lying between associative and commutative
algebras, namely the $\En$-algebras for $n = 1,2,\ldots$.\footnote{The
  Dunn-Lurie Additivity Theorem \cite[Theorem 5.1.2.2]{HA} says that
  this iterative definition agrees with the classical definition in
  terms of configuration spaces of little discs in $\mathbb{R}^{n}$.}
In particular, we can consider $\En$-algebras in the \icat{} $\CatI$
of \icats{}, which gives the notion of \emph{$\En$-monoidal \icats{}},
i.e. \icats{} equipped with $n$ compatible tensor products.

The general version of our first main result, which we will prove in
\S\ref{subsec:np1foldcat}, is then a construction of
$(\infty,n+1)$-categories of $\En$-algebras in any nice $\En$-monoidal
\icat{}:
\begin{thm}
  Let $\mathcal{C}$ be a nice $\En$-monoidal \icat{}. Then there
  exists an $(\infty,n+1)$-category $\fAlg_{n}(\mathcal{C})$ whose
  objects are $\En$-algebras in $\mathcal{C}$, with 1-morphisms given
  by $\mathbb{E}_{n-1}$-algebras in bimodules in $\mathcal{C}$,
  2-morphisms by $\mathbb{E}_{n-2}$-algebras in bimodules in bimodules
  in $\mathcal{C}$, and so forth.
\end{thm}

Here the precise meaning of ``nice'' amounts to the existence of
well-behaved relative tensor products over algebras in $\mathcal{C}$,
which is needed to have well-defined compositions in these higher
categories. For example, we can take $\mathcal{C}$ to be the
(symmetric monoidal) \icat{} $\Mod_{R}$ of modules over a commutative
ring spectrum $R$ or the ``derived \icat{}'' $\mathcal{D}_{\infty}(R)$
of modules over an associative ring $R$, obtained by inverting the
quasi-isomorphisms in the category of chain complexes of $R$-modules
(more generally, we can consider the analogous localization of the
category of dg-modules over a dg-algebra $R$).

If $\mathcal{C}$ is a symmetric monoidal \icat{}, we will also show
that $\fAlg_{n}(\mathcal{C})$ inherits a symmetric monoidal
structure. More precisely, our second main result (proved in
\S\ref{subsec:ALGnmon}) is as follows:
\begin{thm}
  If $\mathcal{C}$ is a nice $\mathbb{E}_{m+n}$-monoidal \icat{}, then
  the $(\infty,n+1)$-category $\fAlg_{n}(\mathcal{C})$ inherits a
  natural $\mathbb{E}_{m}$-monoidal structure.
\end{thm}

Finally, our third main result, which we prove in
\S\ref{subsec:Algnmaps}, explains how the $(\infty,n+1)$-categories
$\fAlg_{n}(\mathcal{C})$ are related for different $n$:
\begin{thm}\label{thm:maps} 
  Suppose $\mathcal{C}$ is a nice $\mathbb{E}_{n}$-monoidal
  \icat{}. Then for any $\mathbb{E}_{n}$-algebras $A$ and $B$ in
  $\mathcal{C}$, the $(\infty,n)$-category
  $\fAlg_{n}(\mathcal{C})(A,B)$ of maps from $A$ to $B$ is equivalent
  to $\fAlg_{n-1}(\Bimod_{A,B}(\mathcal{C}))$, where
  $\Bimod_{A,B}(\mathcal{C})$ is the \icat{} of $A$-$B$-bimodules in
  $\mathcal{C}$ equipped with a natural $\mathbb{E}_{n-1}$-monoidal
  structure. In particular, if $I$ is the unit of the monoidal
  structure then $\fAlg_{n}(\mathcal{C})(I,I) \simeq
  \fAlg_{n-1}(\mathcal{C})$.
\end{thm}

\subsection{Higher Brauer Groups}
If $\mathbf{C}$ is a symmetric monoidal category, we say that an
object $X \in \mathbf{C}$ is \emph{invertible} if there exists another
object $X^{-1}$ such that $X \otimes X^{-1}$ is isomorphic to the
identity; by considering the homotopy 1-category this gives a notion
of invertible objects in any symmetric monoidal $(\infty,n)$-category.

In particular, if $R$ is a commutative ring then the invertible
objects of $\fAlg_{1}(\Mod_{R})$ are those associative $R$-algebras $A$ that
have an inverse $A^{-1}$ in the sense that $A \otimes_{R} A^{-1}$ is
Morita equivalent to $R$ --- these are precisely the \emph{Azumaya
  algebras} over $R$. By considering these invertible objects and the
invertible 1- and 2-morphisms between them we obtain a symmetric
monoidal 2-groupoid $\mathfrak{Br}_{1}(R)$ with very interesting homotopy
groups:
\begin{itemize}
\item $\pi_{0}\mathfrak{Br}_{1}(R)$, i.e. the set of isomorphism classes
  of objects in $\mathfrak{Br}_{1}(R)$, is the classical Brauer group of
  Azumaya $R$-algebras,
\item $\pi_{1}\mathfrak{Br}_{1}(R)$ is the Picard group of invertible
  $R$-modules,
\item $\pi_{2}\mathfrak{Br}_{1}(R)$ is the group $R^{\times}$ of
  multiplicative units in $R$.
\end{itemize}
Moreover, the ``loop space'' $\Omega \mathfrak{Br}_{1}(R) =
\mathfrak{Br}_{1}(R)(R, R)$ is the \emph{Picard groupoid} of
invertible $R$-modules and isomorphisms.

Using the results of this paper, we can also consider the invertible
objects in $\fAlg_{n}(\mathcal{C})$ for any suitable symmetric
monoidal \icat{} $\mathcal{C}$. Restricting to the invertible $i$-morphisms
between these for all $i$, we get a symmetric monoidal $\infty$-groupoid
$\mathfrak{Br}_{n}(\mathcal{C})$, or equivalently an
$\mathbb{E}_{\infty}$-space; we will call this the \emph{$n$-Brauer
  space} of $\mathcal{C}$. It is evident from the definition of the invertible
objects that this $\mathbb{E}_{\infty}$-space is \emph{grouplike},
i.e. the induced multiplication on
$\pi_{0}\mathfrak{Br}_{n}(\mathcal{C})$ makes this monoid a group, and
so it corresponds to a connective spectrum.

It follows immediately from Theorem~\ref{thm:maps} that the loop space
$\Omega \mathfrak{Br}_{n}(\mathcal{C})$ is equivalent to
$\mathfrak{Br}_{n-1}(\mathcal{C})$. Thus the $n$-Brauer spaces
$\mathfrak{Br}_{n}(\mathcal{C})$ are a sequence of deloopings, and so
we can combine these spaces into a non-connective ``Brauer spectrum''
$\mathfrak{BR}(\mathcal{C})$ with \[\pi_{-k}\mathfrak{BR}(\mathcal{C})
= \pi_{n-k}\mathfrak{Br}_{n}(\mathcal{C})\] for $n \geq k$.

If $R$ is a commutative ring spectrum, the ``$n$-Brauer groups''
\[\txt{Br}_{n}(R) := \pi_{-n} \mathfrak{BR}(\Mod_{R}) =
\pi_{0}\mathfrak{Br}_{n}(\Mod_{R})\] can be thought of as consisting of
the $\mathbb{E}_{n}$-analogues of (derived) Azumaya algebras,
considered up to an $\mathbb{E}_{n}$-variant of Morita equivalence. In
particular,
\begin{itemize}
\item for $n = 1$ we recover the Brauer groups of commutative ring
  spectra and the derived Brauer groups of commutative rings, as
  studied by To\"{e}n~\cite{ToenAzumaya}, Szymik~\cite{SzymikBrauer},
  Baker-Richter-Szymik~\cite{BakerRichterSzymik},
  Antieau-Gepner~\cite{AntieauGepnerBrauer}, and others;
\item for $n = 0$ we recover the Picard group of invertible
  $R$-modules, as studied by
  Hopkins-Mahowald-Sadofsky~\cite{HopkinsMahowaldSadofsky},
  May~\cite{MayPic}, Mathew-Stojanoska~\cite{MathewStojanoska}, and
  others.
\end{itemize}
The ``negative Brauer groups'' (i.e. the positive homotopy groups of
$\mathfrak{BR}(\Mod_{R})$) are also easy to describe: for $* < 0$ we
get the homotopy groups of the units of $R$, i.e. $\txt{Br}_{*}(R) =
\pi_{1-*}(\Omega^{\infty}R^{\times})$, where
$\Omega^{\infty}R^{\times}$ denotes the components of
$\Omega^{\infty}R$ lying over the units in $\pi_{0}R$; for $* < -1$ we
thus have $\txt{Br}_{*}(R) = \pi_{1-*}(R)$.

A fascinating question for future research is whether the spaces
$\mathfrak{Br}_{n}(R)$ for $R$ a (connective) commutative ring
spectrum satisfy étale descent in the same way as the Brauer spaces
$\mathfrak{Br}_{1}(R)$ (as proved by To\"{e}n~\cite{ToenAzumaya} and
Antieau-Gepner~\cite{AntieauGepnerBrauer}). 

If the étale-local triviality results of the same authors for
$\mathfrak{Br}_{1}(R)$ also extend to $n > 1$, it should be possible
to use the resulting descent spectral sequence to compute the higher
Brauer groups in some simple cases. In fact, this would imply that the 
higher Brauer groups are closely related to \'{e}tale cohomology;
generalizing the known results for $n = 1$ and $0$ one
might optimistically conjecture that in general 
\[ \txt{Br}_{n}(R) \cong \mathrm{H}^{n}_{\txt{\'{e}t}}(R; \mathbb{Z})
\times H^{n+1}_{\txt{\'{e}t}}(R; \mathbb{G}_{m}),\] where the first
factor occurs since we are considering non-connective $R$-modules (or
chain complexes of $R$-modules that are not required to be $0$ in
negative degrees).

\subsection{Topological Quantum Field Theories}
\emph{Topological quantum field theories} (or \emph{TQFTs}) were
introduced by Atiyah~\cite{AtiyahTQFT} as a way of formalizing
mathematically some particularly simple examples of quantum field
theories constructed by Witten. The original definition is quite
easy to state:
\begin{defn}
  Let $\txt{Bord}(n)$ be the category with objects closed
  $(n-1)$-manifolds and morphisms (diffeomorphism classes of)
  $n$-dimensional cobordisms between these (thus a morphism from $M$
  to $N$ is an $(n+1)$-manifold with boundary $B$, with an
  identification of $\partial B$ with $M \amalg N$). The disjoint union of manifolds gives a
  symmetric monoidal structure on $\txt{Bord}(n)$, and an
  \emph{$n$-dimensional topological quantum field theory} valued in a
  symmetric monoidal category $\mathbf{C}$ is a symmetric monoidal
  functor $\txt{Bord}(n) \to \mathbf{C}$. 
\end{defn}
Requiring the manifolds and cobordisms to be equipped with various
structures, such as orientations or framings, gives different variants
of the category $\txt{Bord}(n)$.  We get various flavours of TQFTs,
such as oriented or framed TQFTs, by considering these different
versions of $\txt{Bord}(n)$. In examples the category $\mathbf{C}$ is
usually the category $\txt{Vect}_{\mathbb{C}}$ of complex vector
spaces.

One reason mathematicians became interested in TQFTs is that they lead
to interesting invariants of manifolds: if $\mathcal{Z} \colon
\txt{Bord}(n) \to \txt{Vect}_{\mathbb{C}}$ is an n-dimensional TQFT,
then $\mathcal{Z}$ assigns a complex number to any closed $n$-manifold
$M$ --- we can consider $M$ as a cobordism from the empty set to the
empty set, and since this is the unit of the monoidal structure on
$\txt{Bord}(n)$, $\mathcal{Z}(M)$ is a linear map $\mathbb{C} \to
\mathbb{C}$, which is given by multiplication with a complex number.

To compute the number $\mathcal{Z}(M)$ we can cut $M$ along suitable
submanifolds of codimension 1 and use the functoriality of
$\mathcal{Z}$. This is enough to compute these invariants in very low
dimensions ($n \leq 2$). In higher dimensions, however, we would
like to be able to cut our manifolds in more flexible ways, for
example by choosing a triangulation of $M$, to make the invariants
more computable. This led mathematicians to consider the notion of
\emph{extended} topological quantum field theories; this was
formalized by Baez and Dolan \cite{BaezDolanTQFT} in the language of
\emph{$n$-categories} (building on earlier work by, among others,
Freed~\cite{FreedHigherAlg} and Lawrence~\cite{LawrenceTQFT}).

\begin{remark}
  For the definition of Baez and Dolan we consider an $n$-category
  $\txt{Bord}_{n}$ whose objects are compact $0$-manifolds, with
  morphisms given by 1-dimensional cobordisms between 0-manifolds, and
  in general $i$-morphisms for $i = 1,\ldots,n$ given by
  $i$-dimensional cobordisms between manifolds with corners. (For the
  $n$-morphisms we take diffeomorphism classes of these.) The disjoint
  union should equip this with a symmetric monoidal structure, but
  giving a precise definition of this symmetric monoidal $n$-category
  becomes increasingly intractable as $n$ increases. A complete
  definition has been given by
  Schommer-Pries~\cite{SchommerPries2TQFT} in the case $n = 2$, but
  for larger $n$ it seems that an appropriate notion of symmetric
  monoidal $n$-category has not even been defined
.
\end{remark}

\begin{defn}
  Given such a symmetric monoidal $n$-category $\txt{Bord}_{n}$, an
  \emph{$n$-dimensional extended TQFT} valued in a symmetric monoidal
  $n$-category $\mathcal{C}$ is a symmetric monoidal functor
  $\txt{Bord}_{n} \to \mathcal{C}$. As before, considering various
  structures on the manifolds in $\txt{Bord}_{n}$ gives different
  flavours of field theories, such as framed, oriented, and
  unoriented.
\end{defn}
Baez and Dolan also conjectured that there is a simple classification
of \emph{framed} extended topological quantum field theories:
\begin{conjecture}[Cobordism Hypothesis]
  A framed extended TQFT $\mathcal{Z} \colon \txt{Bord}^{\txt{fr}}_{n}
  \to \mathcal{C}$ is classified by the object $\mathcal{Z}(*) \in
  \mathcal{C}$. Moreover, the objects of $\mathcal{C}$ that correspond
  to framed TQFTs admit a simple algebraic description: they are
  precisely the \emph{$n$-dualizable objects}. (We refer to \cite[\S
  2.3]{LurieCob} for a precise definition of $n$-dualizable objects.)
\end{conjecture}
At the time, however, the foundations for higher category theory
required to realize their ideas did not yet exist. The necessary
foundations have only been developed during the past decade, with the
work of Barwick, Bergner, Joyal, Lurie, Rezk, and many others. The
resulting theory of \emph{$(\infty,n)$-categories} is often easier to
work with than the more restricted notion of $n$-category --- in
particular, it is not hard to give a good definition of symmetric
monoidal $(\infty,n)$-categories for arbitrary $n$.

We can then consider an $(\infty,n)$-category
$\txt{Bord}_{(\infty,n)}$ of cobordisms, where we take diffeomorphisms as
our $(n+1)$-morphisms, smooth homotopies as the $(n+2)$-morphisms, and
so on. This also turns out to be much easier to define than the
analogous $n$-category; a sketch of a definition is given by
Lurie~\cite{LurieCob}, and the full details of the construction have
recently been worked out by Calaque and
Scheimbauer~\cite{CalaqueScheimbauerCob}.

It is then natural to define extended TQFTs valued in a symmetric
monoidal $(\infty,n)$-category as symmetric monoidal functors from
$\txt{Bord}_{(\infty,n)}$. In this more general setting, Lurie was
able to prove the Cobordism Hypothesis (although so far only a
detailed sketch \cite{LurieCob} of the proof has appeared). In fact,
Lurie also proves classification theorems for other flavours of TQFTs,
such as oriented or unoriented ones, in terms of the homotopy fixed
points for an action of the orthogonal group $O(n)$ on the space of
$n$-dualizable objects in any symmetric monoidal
$(\infty,n)$-category.

The Cobordism Hypothesis works for an arbitrary symmetric monoidal
$(\infty,n)$-category, and so leaves open the question of what the
appropriate target is for the interesting field theories that arise in
physics and geometry. Motivation from physics
(cf. \cite{FreedHigherAlg,KapustinICM}) suggests that in general a
TQFT should assign an $(n - k-1)$-category enriched in vector spaces,
or more generally in chain complexes of vector spaces, to a closed
$k$-manifold.\footnote{To non-closed manifolds it should assign a
higher-categorical generalization of the notion of a bimodule or
profunctor between enriched categories, which is somewhat complicated
to define.}

The higher category of $\mathbb{E}_{n}$-algebras and iterated
bimodules we will construct here can be considered as a special case
of this general target: $\mathbb{E}_{n}$-algebras in some \icat{}
$\mathcal{C}$ are the same thing as $(\infty,n)$-categories enriched
in $\mathcal{C}$ that have one object, one 1-morphism, \ldots, and one
$(n-1)$-morphism. In fact, it is possible to extend the definitions we
consider here to get definitions of enriched $(\infty,n)$-categories
and iterated bimodules between them; we hope to use these to construct
an $(\infty,n+1)$-category of enriched $(\infty,n)$-categories in a
sequel to this paper.

Although not completely general, the TQFTs valued in the symmetric
monoidal $(\infty,n+1)$-category of $\mathbb{E}_{n}$-algebras are
still very interesting. This situation is discussed in
\cite[\S 4.1]{LurieCob}, where the following results are stated
without proof:
\begin{conjecture}\ 
  \begin{enumerate}[(i)]
  \item All $\mathbb{E}_{n}$-algebras in $\mathcal{C}$ are
    $n$-dualizable in $\fAlg_{n}(\mathcal{C})$, and so give rise to
    framed $n$-dimensional extended TQFTs. (More precisely, all
    objects of $\fAlg_{n}(\mathcal{C})$ are dualizable, and all
    $i$-morphisms have adjoints for $i = 1, \ldots, n-1$.)
  \item The framed $n$-dimensional extended TQFT associated to an
    $\mathbb{E}_{n}$-algebra $A$ is given by the \emph{factorization
      homology} or \emph{topological chiral homology} of $A$. (These
    invariants were first introduced by Lurie~\cite[\S 5.5]{HA} and
    also independently by Andrade~\cite{AndradeThesis}, and have since
    been extensively developed by a number of other authors, in
    particular Francis and collaborators; see for example
    \cite{FrancisTgtCx,AyalaFrancisTanakaFactHlgy,AyalaFrancisRozenblyumFactHlgy}. An
    overview can also be found in Ginot's lecture notes
    \cite{GinotFactNotes}.)
  \item An $\mathbb{E}_{n}$-algebra $A$ is $(n+1)$-dualizable \IFF{}
    it is dualizable as a module over its $S^{k}$-factorization
    homology for all $k = -1,0,1,\ldots,n-1$. (For $n = 1$, this is
    equivalent to $A$ being \emph{smooth} and \emph{proper} ---
    cf. \cite[\S 4.6.4]{HA}.)
  \end{enumerate}
\end{conjecture}
Scheimbauer~\cite{ScheimbauerThesis} has constructed factorization
homology as an extended TQFT valued in a geometric variant of
$\fAlg_{n}(\mathcal{C})$ (defined using locally constant factorization
algebras on certain stratifications of $\RR^{n}$), which confirms the
first two parts of this conjecture. It follows from
Theorem~\ref{thm:maps} that (i) is equivalent to the 1-morphisms in
$\fAlg_{n}(\mathcal{C})$ having adjoints for all $n \geq 2$, and we
hope to use this to give algebraic proofs of (i) and (iii).

\subsection{Related Work}
As already mentioned, a geometric construction of
$(\infty,n+1)$-categories closely related to $\fAlgn(\mathcal{C})$ has
been worked out by Scheimbauer~\cite{ScheimbauerThesis}. However, the
natural definition of bimodules in the factorization algebra setting
is not quite the same as ours: the bimodules that arise from
factorization algebras are \emph{pointed}. If
$\fAlgn^{\txt{FA}}(\mathcal{C})$ denotes Scheimbauer's
$(\infty,n+1)$-category of $\mathbb{E}_{n}$-algebras in $\mathcal{C}$,
we therefore expect the relation to our work to be as follows:
\begin{conjecture}
  Let $\mathcal{C}$ be a nice $\mathbb{E}_{n}$-monoidal \icat{}. Then
  $\fAlgn^{\txt{FA}}(\mathcal{C})$ is equivalent to
  $\fAlgn(\mathcal{C}_{I/})$.
\end{conjecture}
In order to carry out such a comparison, one would need to know that
the iterated bimodules we consider can equivalently be described as
algebras for \iopds{} of ``little discs'' on certain stratifications
of $\mathbb{R}^{n}$ --- this would be a generalization of the
Dunn-Lurie additivity theorem for $\mathbb{E}_{n}$-algebras. Such a
result appears to follow from forthcoming work of Ayala and Hepworth
(extending \cite{AyalaHepworthThetan}); we hope to use this to 
compare the algebraic version of $\fAlg_{n}(\mathcal{C})$ we
construct here to the factorization-algebra-based version of
Scheimbauer in a sequel to this paper.

An alternative geometric construction of $\fAlgn(\mathcal{C})$ is also
part of unpublished work of Ayala, Francis, and Rozenblyum, related to
the construction sketched in the work of Morrison and Walker on the
blob complex~\cite{MorrisonWalkerBlob}.

In the case $n = 1$, an alternative construction of the double
\icats{} $\fAlg_{1}(\mathcal{C})$ using symmetric \iopds{} can be
extracted from \cite[\S 4.4]{HA}. Indeed, many of the results in
\S\ref{sec:algbimod} are simply non-symmetric variants of Lurie's ---
the main advantage of our setup is that our results generalize
easily to $n > 1$.

A bicategory of dg-algebras and dg-bimodules, considered up to
quasi-isomorphism, is discussed in \cite{JohnsonDerMor}. This should
be the homotopy bicategory of our $(\infty,2)$-category of algebras
and bimodules in the corresponding ``derived \icat{}'' of chain
complexes.

Finally, an extension of our construction has been obtained by
Johnson-Freyd and Scheimbauer: in \cite{JohnsonFreydScheimbauerLax}
they show that given an $\mathbb{E}_{k}$-monoidal
$(\infty,n)$-category $\mathcal{C}$, our construction (as well as that
of Scheimbauer) can be used to obtain an $(\infty,n+k)$-category of
$\mathbb{E}_{k}$-algebras in $\mathcal{C}$.

\subsection{Overview}
We being by introducing our models for associative algebras, bimodules
and their tensor products in \S\ref{sec:cartalgbimod}; we discuss them
here only in the context of Cartesian monoidal \icats{}, i.e. ones where
the monoidal structure is the Cartesian product, as this allows us to
clarify their underlying meaning without introducing the machinery of
\iopds{}. Next, in \S\ref{sec:cartenalg} we discuss how iterating
these definitions give models for $\En$-algebras and iterated
bimodules, again in the Cartesian setting. In \S\ref{sec:algbimod} we then
construct the $(\infty,2)$-categories $\fAlg_{1}(\mathcal{C})$ for
$\mathcal{C}$ a general monoidal \icat{}, using \nsiopds{}.  By
working with a notion of \emph{\iopds{} over $\Dnop$} the technical
results we prove for associative algebras turn out to extend to the
setting of $\En$-algebras for $n > 2$, and so in \S\ref{sec:enalg} we
construct the $(\infty,n+1)$-categories $\fAlgn(\mathcal{C})$ without
much more work; we also consider the functoriality of these
$(\infty,n+1)$-categories and their natural monoidal structures, and
finish by identifying their mapping $(\infty,n)$-categories. Finally,
in the appendix we discuss the technical results we need about
\Dniopds{}; these are mostly straightforward variants of results from
\cite{HA}.

\subsection{Notation and Terminology}
This paper is written in the language of \icats{}, as developed in the
guise of quasicategories in the work of Joyal~\cite{JoyalQCNotes} and
Lurie~\cite{HTT,HA}. This means that terms such as ``colimit'', ``Kan
extension'' and ``commutative diagram'' are used (unless otherwise
specified) in their \icatl{} (or ``fully weak'') senses --- for
example, a commutative diagram of shape $\mathcal{I}$ in an \icat{}
$\mathcal{C}$ means a functor of \icats{} $\mathcal{I} \to
\mathcal{C}$, and thus means a diagram that commutes up to a coherent
choice of (higher) homotopies that is specified by this diagram. In
general, we reuse the notation and terminology used by Lurie in
\cite{HTT,HA}; here are some exceptions and reminders:
\begin{itemize}
\item $\simp$ is the simplicial indexing category, with objects the
  non-empty finite totally ordered sets $[n] := \{0, 1, \ldots, n\}$
  and morphisms order-preserving functions between them. Similarly,
  $\simp_{+}$ denotes the augmented simplicial indexing category,
  which also includes the empty set $[-1] = \emptyset$.
\item To avoid clutter, we write $\simp^{n}$ for the product
  $\simp^{\times n}$, and use $\DnIop$ to mean $((\simp^{\times n})_{/I})^{\op}$ for
  any $I \in \simp^{n}$.
\item $\bbGamma^{\op}$ is the category of pointed finite sets.
\item Generic categories are generally denoted by single capital
  bold-face letters ($\mathbf{A},\mathbf{B},\mathbf{C}$) and generic
  \icats{} by single caligraphic letters
  ($\mathcal{A},\mathcal{B},\mathcal{C}$). Specific categories and
  \icats{} both get names in the normal text font.
\item $\sSet$ is the category of simplicial sets, i.e. the category
  $\Fun(\simp^{\op}, \Set)$ of set-valued presheaves on $\simp$.
\item $\mathcal{S}$ is the \icat{} of spaces; this can be defined as
  the coherent nerve $\mathrm{N}\sSet^{\circ}$ of the full
  subcategory $\sSet^{\circ}$ of the category $\sSet$ spanned by the
  Kan complexes, regarded as a simplicial category via the internal Hom.
\item We make use of the theory of \emph{Grothendieck
    universes} to allow us to define ($\infty$-)categories without
  being limited by set-theoretical size issues; specifically, we fix
  three nested universes, and refer to sets contained in them as
  \emph{small}, \emph{large}, and \emph{very large}. When $\mathcal{C}$
  is an \icat{} of small objects of a certain type, we generally refer
  to the corresponding \icat{} of large objects as
  $\widehat{\mathcal{C}}$, without explicitly defining this
  object. For example, $\CatI$ is the (large) \icat{} of small
  \icats{}, and $\LCatI$ is the (very large) \icat{}
  of large \icats{}.
\item If $\mathcal{C}$ is an \icat{}, we write $\iota \mathcal{C}$ for
  the \emph{interior} or \emph{underlying space} of $\mathcal{C}$,
  i.e. the largest subspace of $\mathcal{C}$ that is a Kan complex.
\item If a functor $f \colon \mathcal{C} \to \mathcal{D}$ (of
  \icats{}) is left adjoint to a functor $g \colon \mathcal{D} \to
  \mathcal{C}$, we will refer to the adjunction as $f \dashv g$.
\item We will say that a functor $f \colon \mathcal{C} \to
   \mathcal{D}$ of \icats{} is \emph{coinitial} if the opposite
   functor $f^{\op} \colon \mathcal{C}^{\op} \to \mathcal{D}^{\op}$ is
   cofinal in the sense of \cite[\S 4.1.1]{HTT}.
 \item If $K$ is a simplicial set, we denote the cone points of the
   simplicial sets $K^{\triangleright}$ and $K^{\triangleleft}$,
   obtained by freely adjoining a final and an initial object to $K$,
   by $\infty$ and $-\infty$, respectively.
 \item We say an \icat{} (or more generally any simplicial set)
   $\mathcal{C}$ is \emph{weakly contractible} if the map $\mathcal{C}
   \to \Delta^{0}$ is a weak equivalence in the Kan-Quillen model
   structure (as opposed to the Joyal model structure). This is
   equivalent to the $\infty$-groupoid obtained by inverting all the
   morphisms in $\mathcal{C}$ being trivial, and to the geometric
   realization of the simplicial set $\mathcal{C}$ being a
   contractible topological space.
\end{itemize}

\subsection{Some Key Concepts}
As an aid to readers who are not intimately familiar with \cite{HTT},
in this subsection we briefly introduce some key concepts that we will
make use of throughout this paper, namely \emph{coCartesian
  fibrations}, \emph{cofinal functors}, and \emph{relative
  (co)limits}.

\begin{defn}
  If $f \colon \mathcal{E} \to \mathcal{B}$ is a functor of \icats{},
  a morphism $\epsilon \colon e \to e'$ in $\mathcal{E}$ lying over
  $\beta \colon b \to b'$ in $\mathcal{B}$ is \emph{$p$-coCartesian
    morphism} if for every $x \in \mathcal{E}$ the
  commutative\footnote{Recall that this means commutative in the
    \icatl{} sense, so the square really includes the data of a
    homotopy between the two composites, which we do not explicitly
    indicate.} square \csquare{\Map_{\mathcal{E}}(e',
    x)}{\Map_{\mathcal{E}}(e,x)}{\Map_{\mathcal{B}}(b',
    f(x))}{\Map_{\mathcal{B}}(b, f(x))}{\epsilon^{*}}{}{}{\beta^{*}}
  is Cartesian, i.e. it is a pullback\footnote{Note that this means
    that it is a pullback in the \icatl{} sense --- if we choose some
    concrete model for these mapping spaces as simplicial sets, this
    is equivalent to the corresponding diagram of simplicial sets
    being a \emph{homotopy} pullback.} square. 
\end{defn}
This is equivalent to the induced map on
fibres \[\Map_{\mathcal{E}}(e',x)_{f} \to
\Map_{\mathcal{E}}(e,x)_{\phi \circ \beta}\] being an equivalence for
all maps $\phi \colon b' \to f(x)$, so this definition gives a natural
\icatl{} generalization of coCartesian morphisms in ordinary category
theory.
\begin{defn}
  We say that a functor of \icats{} $f \colon \mathcal{E} \to
  \mathcal{B}$ is a \emph{coCartesian fibration} if for every $e \in
  \mathcal{E}$ and $\beta \colon f(e) \to b$ there exists an
  $f$-coCartesian morphism $e \to \beta_{!}e$ over $\beta$;
  coCartesian fibrations are thus the natural \icatl{} version of
  \emph{Grothendieck opfibrations}.
\end{defn}
If we think of $f$ as a map of simplicial sets, and
assume (as we are free to do up to equivalence) that it is an inner
fibration, then this definition can be reformulated more concretely in
terms of the existence of liftings for certain horns, which is the
definition given in \cite[\S 2.4.2]{HTT}.

If $f \colon \mathcal{E} \to \mathcal{B}$ is a coCartesian fibration,
then \cite[Corollary 3.2.2.12]{HTT} implies that the induced functor
$\Fun(K, \mathcal{E}) \to \Fun(K, \mathcal{B})$ is also a coCartesian
fibration, for any $K$. Given diagrams $p \colon K \to \mathcal{E}$
and $\overline{q} \colon K^{\triangleright} \to \mathcal{B}$ with
$f\circ p = q := \overline{q}|_{K}$ we can therefore define a
\emph{coCartesian pushforward} of $p$ to a diagram $p' \colon K \to
\mathcal{E}_{\overline{q}(\infty)}$ lying in the fibre over
$\overline{q}(\infty)$, by regarding $\overline{q}$ as a morphism in
$\Fun(K, \mathcal{B})$ from $q$ to the constant functor at
$\overline{q}(\infty)$ and choosing a coCartesian morphism over this
with source $p$.

Grothendieck proved (in \cite{SGA1}) that Grothendieck opfibrations
over a category $\mathbf{C}$ correspond to (pseudo)functors from
$\mathbf{C}$ to the category of categories. Lurie's
\emph{straightening equivalence} from \cite[\S 3.2]{HTT} establishes
an analogous equivalence between coCartesian fibrations over an
\icat{} $\mathcal{C}$ and functors from $\mathcal{C}$ to the \icat{}
$\CatI$ of \icats{}. For more details on coCartesian fibrations, and
the dual concept of \emph{Cartesian} fibrations, see \cite[\S\S
  2.4 and \S 3.2]{HTT}, especially \S 2.4.1--4.

\begin{defn}
  A functor $F \colon \mathcal{A} \to \mathcal{B}$ of \icats{} is
  \emph{cofinal} if for every diagram $p \colon
  \mathcal{B} \to \mathcal{C}$, the induced functor $\mathcal{C}_{p/}
  \to \mathcal{C}_{p\circ F/}$ is an equivalence. Dually, $F$ is
  \emph{coinitial} if $F^{\op} \colon \mathcal{A}^{\op} \to
  \mathcal{B}^{\op}$ is cofinal, i.e. the functor $\mathcal{C}_{/p}
  \to \mathcal{C}_{/p \circ F}$ is an equivalence for every $p$.
\end{defn}
Since a colimit of $p$ is the same thing as a final object in
$\mathcal{C}_{p/}$, we see that if $F$ is cofinal then $p$ has a
colimit \IFF{} $p \circ F$ has a colimit, and these colimits are
necessarily given by the same object in $\mathcal{C}$. The key
criterion for cofinality is \cite[Theorem 4.1.3.1]{HTT}: $F \colon
\mathcal{A} \to \mathcal{B}$ is cofinal \IFF{} for every $b \in
\mathcal{B}$ the slice \icat{} $\mathcal{A}_{b/} := \mathcal{A}
\times_{\mathcal{B}} \mathcal{B}_{b/}$ is weakly contractible. For
more details, see \cite[\S 4.1]{HTT}, especially \S 4.1.1.

\begin{defn}
  Given a functor of \icats{} $f \colon \mathcal{E} \to \mathcal{B}$
  we say a diagram $\overline{p} \colon K^{\triangleright} \to
  \mathcal{E}$ is a \emph{colimit relative to $f$} (or an
  \emph{$f$-colimit}) of $p := \overline{p}|_{K}$ if the commutative
  square of \icats{}
  \nolabelcsquare{\mathcal{E}_{\overline{p}/}}{\mathcal{B}_{f\overline{p}/}}{\mathcal{E}_{p/}}{\mathcal{B}_{fp/}}
  is Cartesian, i.e. the induced functor
  \[ \mathcal{E}_{\overline{p}/} \to \mathcal{E}_{p/}
  \times_{\mathcal{B}_{fp/}} \mathcal{B}_{f\overline{p}/} \]
  is an equivalence.
\end{defn}
Ordinary colimits in $\mathcal{E}$ are the same thing as colimits
relative to the functor $\mathcal{E} \to *$ to the terminal
\icat{}. Notice also that if $\overline{p} \colon K^{\triangleright}
\to \mathcal{E}$ is a diagram such that $f\overline{p}$ is a colimit
in $\mathcal{B}$, then $\overline{p}$ is an $f$-colimit \IFF{} it is a
colimit in $\mathcal{E}$.

We can also reformulate the definition in terms of mapping spaces:
$\overline{p} \colon K^{\triangleright} \to \mathcal{E}$ is an
$f$-colimit \IFF{} for every $e \in \mathcal{E}$, the commutative
square \nolabelcsquare{\Map_{\mathcal{E}}(\overline{p}(\infty),
  e)}{\lim_{k \in K} \Map_{\mathcal{E}}(p(k),
  e)}{\Map_{\mathcal{B}}(f\overline{p}(\infty), f(e))}{\lim_{k \in K}
  \Map_{\mathcal{E}}(fp(k), f(e))} is Cartesian, or equivalently
(since limits commute) \IFF{} for every map $\phi \colon
f\overline{p}(\infty) \to f(e)$ the map on fibres
\[ \Map_{\mathcal{E}}(\overline{p}(\infty), e)_{\phi} \to \lim_{k
  \in K} \Map_{\mathcal{E}}(p(k), e)_{\phi \circ (fp(k) \to
  f\overline{p}(\infty))} \]
is an equivalence.

If $f$ is a coCartesian fibration, then it follows from
\cite[Propositions 4.3.1.9 and 4.3.1.10]{HTT} that a diagram
$\overline{p} \colon K^{\triangleright} \to \mathcal{E}$ with $x =
f\overline{p}(\infty)$ is an $f$-colimit \IFF{} the coCartesian
pushforward of $\overline{p}$ to the fibre over $x$ is a colimit in
$\mathcal{E}_{x}$, and for every morphism $\phi \colon x \to y$ in
$\mathcal{B}$ the functor $\phi_{!} \colon \mathcal{E}_{x}\to
\mathcal{E}_{y}$ induced by the coCartesian morphisms over $\phi$
preserves this colimit. If $f\overline{p}$ is a colimit diagram in
$\mathcal{B}$, then this gives a useful criterion for relating
colimits in $\mathcal{E}$ to colimits in the fibres of $f$. For more
on relative colimits (and the dual concept of relative limits), see
\rcite{HTT}{\S 4.3.1}.

\subsection{Acknowledgements}
I thank Clark Barwick and Chris Schommer-Pries for sharing their work
on operator categories, of which much of the material in
\S\ref{sec:dnop} is a special case. In addition, I thank David Gepner
and Owen Gwilliam for helpful discussions about this project.

\section{Algebras and Bimodules in the Cartesian
  Setting}\label{sec:cartalgbimod}

Our goal in this section is to introduce the models for algebras and
bimodules we will use in this paper, and to motivate our approach to
defining an $(\infty,2)$-category of these. Here we will only consider
the case where the monoidal \icat{} these take values in has the
Cartesian product as its tensor product --- to consider general
monoidal \icats{} we must work in the context of (non-symmetric)
\iopds{}, and this extra layer of formalism can potentially obscure
the simple underlying meaning of our definitions. In
\S\ref{subsec:assalg} we recall how associative monoids can be
modelled as certain simplicial objects, and in
\S\ref{subsec:cartbimod} we will see that bimodules can similarly be
described as certain presheaves on the slice category
$\simp_{/[1]}$. Next, in \S\ref{subsec:simp2tens} we discuss how
relative tensor products of bimodules can be described in this
context, using presheaves on $\simp_{/[2]}$. In \S\ref{subsec:segsp}
we recall that a more general class of simplicial objects can be used
to model internal categories in an \icat{} --- in particular, we
review Rezk's \emph{Segal spaces}, which are a model for \icats{}. We
then indicate in \S\ref{subsec:cartcatalg} how, by considering certain
presheaves on $\simp_{/[n]}$ for arbitrary $n$, we can construct a
Segal space that describes an \icat{} of algebras or bimodules --- or
more generally a double \icat{} of these, from which the desired
$(\infty,2)$-category can be extracted.

\subsection{$\simp$ and Associative Algebras}\label{subsec:assalg}
The observation that simplicial spaces satisfying a certain ``Segal
condition'' give a model for $A_{\infty}$-spaces, i.e. spaces equipped
with a homotopy-coherently associative multiplication, goes back to
unpublished work of Segal. Formulated in the language of \icats{},
Segal's definition of a homotopy-coherently associative monoid, which
in the \icatl{} setting is the only meaningful notion of an
associative monoid, is the following:
\begin{defn}
  Let $\mathcal{C}$ be an \icat{} with finite products. An
  \emph{associative monoid} in $\mathcal{C}$ is a simplicial object
  $A_{\bullet} \colon \simp^{\op} \to \mathcal{C}$ such that for every
  $[n]$ in $\simp$ the natural map
  \[ A_{n} \to A_{1} \times \cdots \times A_{1}, \] induced by the
  maps $\rho_{i} \colon [1] \to [n]$ in $\simp$ that send $0$ to $i-1$
  and $1$ to $i$, is an equivalence.
\end{defn}
To see that this definition makes sense, observe that the inner face
map $d_{1} \colon [1] \to [2]$ induces a multiplication
\[ A_{1}\times A_{1} \isofrom A_{2} \xto{d_{1}} A_{1},\]
and the degeneracy $s_{0} \colon [1] \to [0]$ induces a unit
\[ * \isofrom A_{0} \xto{s_{0}} A_{1}. \] To see that the
multiplication is associative, observe that the
comutative square
\csquare{A_{3}}{A_{2}}{A_{2}}{A_{1},}{d_{1}}{d_{2}}{d_{1}}{d_{1}}
exhibits a homotopy between the two possible multiplications
$A_{1}^{\times 3} \to A_{1}$. Similarly, the higher-dimensional cubes
giving compatibilities between the different composites of face maps
$[1] \to [n]$ exhibit the higher coherence homotopies for the
associative monoid.

\subsection{$\simp_{/[1]}$ and Bimodules}\label{subsec:cartbimod}
We will now see that, just as simplicial objects give a natural notion of
associative monoids, presheaves on the slice category $\simp_{/[1]}$
give a model for bimodules between associative monoids:
\begin{defn}
  Let $\mathcal{C}$ be an \icat{} with finite products. A
  \emph{bimodule} in $\mathcal{C}$ is a functor \[M \colon \Dop_{/[1]}
  \to \mathcal{C}\] such that for every object $\phi \colon [n] \to [1]$ in
  $\simp_{/[1]}$, the natural map
  \[ M(\phi) \to M(\phi\rho_{1}) \times \cdots \times
  M(\phi\rho_{n}), \]
  induced by composition with the maps $\rho_{i} \colon [1] \to [n]$,
  is an equivalence.
\end{defn}
To see that such objects can indeed be interpreted as bimodules,
observe that the category $\simp_{/[1]}$ can be described as having
objects sequences $(i_{0},\ldots,i_{n})$ where $0 \leq i_{k}\leq
i_{k+1} \leq 1$, with a unique morphism
$(i_{\phi(0)},\ldots,i_{\phi(n)}) \to (i_{0},\ldots,i_{m})$ for every
$\phi \colon [n] \to [m]$ in $\simp$. In terms of this description a functor $M \colon \Dop_{/[1]} \to \mathcal{C}$ is a bimodule
\IFF{} the object $M(i_{0},\ldots,i_{n})$ decomposes as
$M(i_{0},i_{1}) \times \cdots \times M(i_{n-1},i_{n})$. Thus every
object decomposes as a product of $M(0,0)$, $M(0,1)$ and $M(1,1)$. 

The two maps $[0] \to [1]$ induce functors $\simp \to \simp_{/[1]}$
--- these are the inclusions of the full subcategories of
$\simp_{/[1]}$ with objects of the form $(0,\ldots,0)$ and
$(1,\ldots,1)$. Restricting along these we see that $M(0,0)$ and
$M(1,1)$ are associative monoids. The maps $(0,1) \to (0,0,1)$ and
$(0,1) \to (0,1,1)$ in $\simp_{/[1]}$ give multiplications
\[ M(0,0) \times M(0,1) \isofrom M(0,0,1) \to M(0,1), \]
\[ M(0,1) \times M(1,1) \isofrom M(0,1,1) \to M(0,1), \] which exhibit
$M(0,1)$ as a left $M(0,0)$-module and a right
$M(1,1)$-module. Moreover, the commutative square
\nolabelcsquare{M(0,0,1,1)}{M(0,1,1)}{M(0,0,1)}{M(0,1)} implies that
these module structures are compatible. The remaining data given by
$M$ shows that these actions are homotopy-coherently associative and
compatible with the multiplications in $M(0,0)$ and $M(1,1)$.

\subsection{$\simp_{/[2]}$ and Tensor Products of Bimodules}\label{subsec:simp2tens}
We can similarly define \emph{$\simp_{/[2]}$-monoids} as
certain presheaves on $\simp_{/[2]}$. If we think of $\simp_{/[2]}$ as
having objects sequences $(i_{0},\ldots,i_{m})$ with $0 \leq i_{k}
\leq i_{k+1} \leq 2$, then we can phrase the definition as follows:
\begin{defn}
  Let $\mathcal{C}$ be an \icat{} with finite products. Then a
  \defterm{$\simp_{/[2]}$-monoid} in $\mathcal{C}$ is a functor
  $M \colon \simp_{/[2]}^{\op} \to \mathcal{C}$ such that for every
  object $(i_{0},\ldots,i_{m})$, the natural map
  \[ M(i_{0},\ldots,i_{m}) \to M(i_{0},i_{1}) \times \cdots \times
  M(i_{n-1},i_{n}),\]
  induced by composition with the maps $\rho_{i}$, is an equivalence.
\end{defn}
Unravelling this definition, we see that a $\simp_{/[2]}$-monoid $M$ in
$\mathcal{C}$ is given by the data of
\begin{itemize}
\item three associative monoids $M_{0} = M(0,0), M_{1} = M(1,1)$ and
  $M_{2} = M(2,2)$, given
  by the restrictions of $M$ along the three natural inclusions
  $\simp^{\op} \to \simp^{\op}_{/[2]}$,
\item three bimodules: an $M_{0}$-$M_{1}$-bimodule $M(0,1)$, an
  $M_{1}$-$M_{2}$-bimodule $M(1,2)$, and an
  $M_{0}$-$M_{2}$-bimodule $M(0,2)$, given by the restrictions of
  $M$ along the three natural inclusions $\simp^{\op}_{/[1]} \to
  \simp^{\op}_{/[2]}$,
\item an $M_{1}$-balanced map $M(0,1) \times M(1,2) \simeq M(0,1,2)
  \to M(0,2)$, which we can think of as the restriction of $M$ along
  the inclusion $j \colon \simp_{+}^{\op} \to \simp^{\op}_{/[2]}$ that
  sends $[n]$ to $(0,1,\ldots, 1,2)$ (with $(n+1)$ 1's, for $n =
  -1,0,\ldots$).
\end{itemize}
We would like to understand what it means for the bimodule $M(0,2)$ to be the
tensor product $M(0,1) \otimes_{M_{1}} M(1,2)$ in terms of this data.
In classical algebra, if $A$ is an assocative algebra and $M$ is a
right and $N$ a left $A$-module, the tensor product $M \otimes_{A} N$
can be defined as the reflexive coequalizer of the two multiplication
maps $M \times A \times N \to M \times N$. As usual, in the \icatl{}
setting this coequalizer must be replaced by its ``derived'' version,
namely the colimit of a simplicial diagram, commonly known as the
``bar construction'': specifically, this is the diagram $B(M, A,
N)_{\bullet} := M \times A^{\times \bullet} \times N$ with face maps
given by multiplications and degeneracies determined by the unit of
$A$. 

For a $\simp_{/[2]}$-monoid $M$, this diagram is precisely the
restriction of the augmented simplicial diagram $j$ to
$\simp^{\op}$. Thus, the bimodule $M(0,2)$ is a tensor product
precisely when $j$ is a colimit diagram, which leads us to make the
following definition:
\begin{defn}
  We say a $\Dop_{/[2]}$-monoid $M$ in $\mathcal{C}$ is
  \emph{composite} if the map
\[ \simp_{+}^{\op} \xto{j} \simp^{\op}_{/[2]} \xto{M} \mathcal{C}\] is a
colimit diagram.
\end{defn}

\subsection{$\simp$ and $\infty$-Categories}\label{subsec:segsp}
As originally observed by Rezk~\cite{RezkCSS}, a generalization of
Segal's definition of associative monoids gives a model for \icats{},
namely \emph{Segal spaces}. In the \icatl{} context, these are a
special case of the natural definition of an internal category or
category object:
\begin{defn}
  Let $\mathcal{C}$ be an \icat{} with finite limits. A
  \defterm{category object} in $\mathcal{C}$ is a simplicial object
  $X_{\bullet} \colon \simp^{\op} \to \mathcal{C}$ such that for all
  $[n] \in \simp$ the natural map
  \[ X_{n} \to X_{1} \times_{X_{0}} \cdots \times_{X_{0}} X_{1},\]
  induced by the maps $\rho_{i} \colon [1] \to [n]$ and the maps $[0]
  \to [n]$, is an equivalence. We write $\Cat(\mathcal{C})$ for the
  full subcategory of $\Fun(\Dop, \mathcal{C})$ spanned by the
  category objects.
\end{defn}
A \emph{Segal space} is a category object in the \icat{} $\mathcal{S}$
of spaces. We can think of a Segal space $X_{\bullet}$ as having a
space $X_{0}$ of ``objects'' and a space $X_{1}$ of ``morphisms''; the
face maps $X_{1} \rightrightarrows X_{0}$ assign the source and target
object to each morphism, and the degeneracy $s_{0} \colon X_{0} \to
X_{1}$ assigns an identity morphism to every object. Then $X_{n}\simeq
X_{1} \times_{X_{0}} \cdots \times_{X_{0}} X_{1}$ is the space of
composable sequences of $n$ morphisms, and the face map $d_{1} \colon
[1] \to [2]$ gives a composition
\[ X_{1}\times_{X_{0}} X_{1} \isofrom X_{2} \xto{d_{1}} X_{1}.\] The
remaining data in $X_{\bullet}$ gives the homotopy-coherent
associativity data for this composition and its compatibility with
the identity maps.

\begin{remark}
We can regard the \icat{} $\CatI$ of \icats{} as the localization of
the \icat{} of Segal spaces at the fully faithful and essentially
surjective functors (in the appropriate homotopically correct
sense). The main theorem of \cite{RezkCSS} is that this localization
is given by the full subcategory $\txt{CSS}(\mathcal{S})$ of
$\Cat(\mathcal{S})$ spanned by the \emph{complete} Segal spaces. It
was proved by Joyal and Tierney~\cite{JoyalTierney} that the model
category of complete Segal spaces is Quillen equivalent to Joyal's
model category of quasicategories, and so the \icat{} $\CatI$, defined
using quasicategories, is equivalent to $\txt{CSS}(\mathcal{S})$.
\end{remark}

We will also make use of category objects in $\CatI$. These give a
notion of \emph{double \icats{}}, just as double categories can be
thought of as internal categories in $\Cat$. We will see below in
\S\ref{subsec:inftyn} that, just as a double category has two
underlying bicategories, a double \icat{} has two underlying
$(\infty,2)$-categories.

\subsection{$\simp_{/[n]}$ and the $(\infty,2)$-Category of Algebras
  and Bimodules}\label{subsec:cartcatalg}

As a preliminary to discussing the $(\infty,2)$-category of algebras
and bimodules in an \icat{} $\mathcal{C}$ with finite products, let us
consider the underlying \icat{} $\falg_{1}(\mathcal{C})$ of algebras
and bimodules as a Segal space. From our discussion so far, we have an
obvious choice for the space $\falg_{1}(\mathcal{C})_{0}$ of objects,
namely the space of associative monoids in $\mathcal{C}$, and for the
space $\falg_{1}(\mathcal{C})_{1}$ of morphisms, namely the space of
$\simp_{/[1]}$-monoids in $\mathcal{C}$. These spaces are simply the
appropriate collections of connected components in the spaces
$\Map(\simp^{\op}, \mathcal{C})$ and $\Map(\simp^{\op}_{/[1]},
\mathcal{C})$, respectively. The source and target maps are induced by
composition with $d_{1}$ and $d_{0} \colon [0] \to [1]$, and
composition with $s_{0} \colon [1] \to [0]$ sends a monoid $A$
to $A$ considered as an $A$-$A$-bimodule, giving the correct
identity morphisms. 

In order to construct a Segal space, the space 
$\falg_{1}(\mathcal{C})_{2}$ must be equivalent to
$\falg_{1}(\mathcal{C})_{1} \times_{\falg_{1}(\mathcal{C})_{0}}
\falg_{1}(\mathcal{C})_{1}$. On the other hand, the composition in
$\falg_{1}(\mathcal{C})$ should be given by relative tensor products
of bimodules, which we saw above corresponds to taking a composite
$\simp_{/[2]}$-monoid and composing with the middle face map $d_{1}
\colon [1] \to [2]$ --- this suggests that the space
$\falg_{1}(\mathcal{C})_{2}$ should be the space of composite
$\simp_{/[2]}$-monoids. Luckily, it will turn out that the space of
composite $\simp_{/[2]}$-monoids is indeed equivalent to $\falg_{1}(\mathcal{C})_{1} \times_{\falg_{1}(\mathcal{C})_{0}}
\falg_{1}(\mathcal{C})_{1}$ via the appropriate forgetful maps, so
this does actually make sense.

To define the spaces $\falg_{1}(\mathcal{C})_{n}$ for general $n$, we
similarly consider \emph{composite $\simp_{/[n]}$-monoids} for
arbitrary $n$. If we think of $\simp_{/[n]}$ as having objects
sequences $(i_{0},\ldots,i_{m})$ with $0 \leq i_{k} \leq i_{k+1} \leq
n$, we have the following definition:
\begin{defn}
  Let $\mathcal{C}$ be an \icat{} with finite products. Then a
  \defterm{$\simp_{/[n]}$-monoid} in $\mathcal{C}$ is a functor
  $M \colon (\simp_{/[n]})^{\op} \to \mathcal{C}$ such that for every
  object $(i_{0},\ldots,i_{m})$, the natural map
  \[ M(i_{0},\ldots,i_{m}) \to M(i_{0},i_{1}) \times \cdots \times
  M(i_{n-1},i_{n}),\]
  induced by composition with the maps $\rho_{i}$, is an equivalence.
\end{defn}
A $\simp_{/[n]}$-monoid in $\mathcal{C}$ describes
\begin{itemize}
\item $n+1$ associative monoids $M_{0} = M(0,0), M_{1} =
  M(1,1),\ldots, M_{n} = M(n,n)$,
\item for each pair $(i,j)$ with $0 \leq i < j \leq n$, an
  $M_{i}$-$M_{j}$-bimodule $M(i,j)$,
\item for each triple $(i,j,k)$ with $0 \leq i < j < k \leq n$, an
  $M_{j}$-balanced map $M(i,j) \times M(j,k) \to M(i,k)$, compatible
  with the actions of $M_{i}$ and $M_{k}$,
\item such that these bilinear maps are compatible, e.g. if $0 \leq i
  < j < k < l \leq n$ then the diagram \nolabelcsquare{M(i,j) \times
    M(j,k) \times M(k,l)}{M(i,j) \times M(j,l)}{M(i,k) \times
    M(k,l)}{M(i,l)} commutes.
\end{itemize}
Composition with the maps $\phi_{*} \colon \simp_{/[n]} \to
\simp_{/[m]}$ given by composition with a map $\phi \colon [n] \to
[m]$ in $\simp$ takes $\simp_{/[m]}$-monoids in $\mathcal{C}$ to
$\simp_{/[n]}$-monoids.

We say that a $\simp_{/[n]}$-monoid $M$ is \emph{composite} if these
maps exhibit the bimodule $M(i,j)$ as the iterated tensor product 
\[ M(i,i+1) \otimes_{M_{i+1}} M(i+1,i+2) \otimes_{M_{i+2}} \cdots
\otimes_{M_{j-1}} M(j-1,j).\] As in the case $n = 2$, this condition
can be formulated precisely in terms of certain (multi)simplicial
diagrams being colimits --- we will discuss this in more detail below in
\S\ref{subsec:bimodopd}.

If $\falg_{1}(\mathcal{C})_{n}$ denotes the
space of composite $\simp_{/[n]}$-monoids, then the main results of
the \S\ref{sec:algbimod} will tell us:
\begin{itemize}
\item The composite monoids are preserved under composition with the
    maps $\simp_{/[n]} \to \simp_{/[m]}$ coming from maps in
    $\simp$. Thus the spaces $\falg_{1}(\mathcal{C})_{\bullet}$
    fit together into a simplicial space.
\item The spaces $\falg_{1}(\mathcal{C})_{\bullet}$ satisfy the Segal
  condition, i.e. the map
  \[ \falg_{1}(\mathcal{C})_{n} \to \falg_{1}(\mathcal{C})_{1}
  \times_{\falg_{1}(\mathcal{C})_{0}} \cdots
  \times_{\falg_{1}(\mathcal{C})_{0}} \falg_{1}(\mathcal{C})_{1} \] is
  an equivalence for all $n$.
\end{itemize}
In other words, $\falg_{1}(\mathcal{C})_{\bullet}$ is a Segal space.
This (or more precisely its completion) is our \icat{} of algebras
and bimodules. 

We can just as easily consider the \emph{\icats{}}
$\fALG_{1}(\mathcal{C})_{n}$ of composite $\simp_{/[n]}$-monoids,
i.e. the appropriate full subcategories of $\Fun(\simp_{/[n]}^{\op},
\mathcal{C})$. We'll show that these form a category object
$\fALG_{1}(\mathcal{C})$ in $\CatI$, i.e. a double \icat{} --- this
has associative monoids as objects, algebra homomorphisms as vertical
morphisms, bimodules as horizontal morphisms, and bimodule
homomorphisms as commutative squares. As we will see below in
\S\ref{subsec:inftyn}, from this double \icat{} we can then extract an
$(\infty,2)$-category $\fAlg_{1}(\mathcal{C})$ of algebras, bimodules,
and bimodule homomorphisms.

\section{$\mathbb{E}_{n}$-Algebras and Iterated Bimodules in the
  Cartesian Setting}\label{sec:cartenalg}

The definitions we considered in \S\ref{sec:cartalgbimod} can be iterated,
and in this section we will discuss how this leads to an
$(\infty,n+1)$-category of $\En$-algebras, again in the Cartesian
case. In \S\ref{subsec:carten} we consider iterated $\simp$-monoids,
which gives a model for $\En$-algebras. Then in \S\ref{subsec:uple} we
see that, similarly, iterating the notion of category object gives
\emph{$n$-uple} \icats{}, in the form of $n$-uple Segal spaces. This
leads to a notion of $(\infty,n)$-categories in the form of Barwick's
\emph{iterated Segal spaces}, which we review in
\S\ref{subsec:inftyn}; this is the model of $(\infty,n)$-categories we
will use below in \S\ref{sec:enalg}. Finally, in
\S\ref{subsec:cartitbimod} we indicate how the definition of the
double \icat{} of algebras and bimodules can be iterated to get
$(n+1)$-uple \icats{} of $\En$-algebras in a Cartesian monoidal
\icat{}.

\subsection{$\Dn$ and $\mathbb{E}_{n}$-Algebras}\label{subsec:carten}
The Dunn-Lurie Additivity Theorem \rcite{HA}{Theorem 5.1.2.2} implies
that, in the \icatl{} setting, $\mathbb{E}_{n}$-algebras in some
\icat{} $\mathcal{C}$ are equivalent to associative algebras in
$\mathbb{E}_{n-1}$-algebras in $\mathcal{C}$. In the Cartesian case
we would thus expect that associative monoids in associative monoids
in \ldots{} in $\mathcal{C}$ give a model for $\En$-algebras in
$\mathcal{C}$ --- we will prove a precise version of this claim below
in \S\ref{subsec:dnmonoid}. Unwinding the definition, we see that
these objects can be described as certain multisimplicial objects in
$\mathcal{C}$:
\begin{defn}
  Let $\mathcal{C}$ be an \icat{} with finite products. A
  \emph{$\Dn$-monoid} in $\mathcal{C}$ is a multisimplicial object
  \[ A_{\bullet,\ldots,\bullet} \colon \Dnop \to  \mathcal{C} \]
  such that for every object $([i_{1}],\ldots,[i_{n}]) \in \Dn$,
  the natural map
  \[ A_{i_{1},\ldots,i_{n}} \to \prod_{j_{1} = 1}^{i_{1}} \cdots
  \prod_{j_{n} = 1}^{i_{n}} A_{1,\ldots,1},\]
  induced by the maps $(\rho_{j_{1}},\ldots,\rho_{j_{n}})$,
  is an equivalence.
\end{defn}

\begin{remark}
  It is convenient to introduce some notation to simplify this
  definition: let $C_{n}$ denote the object $([1],\ldots,[1])$ in
  $\Dnop$, and for $I \in \Dnop$ let $|I|$ denote the set of
  (levelwise) inert maps $C_{n} \to I$, i.e. the maps
  $(\rho_{i_{1}},\ldots,\rho_{i_{n}})$. Then the Segal condition for
  a $\Dn$-monoid $A$ can be stated as: for every $I \in \Dnop$, the
  natural map $A_{I} \to A^{\times |I|}_{C_{n}}$ induced by the maps
  in $|I|$ is an equivalence.
\end{remark}

\subsection{$\Dn$ and $n$-uple $\infty$-Categories}\label{subsec:uple}
Just as we can iterate the notion of associative monoid to get a
definition of $\En$-algebras in the Cartesian setting, we can iterate
the definition of a category object to get a definition of
\emph{$n$-uple internal categories}. To state this definition more
explicitly, it is useful to first introduce some notation:

\begin{defn}\label{defn:activeinert}
  A morphism $f \colon [n] \to [m]$ in $\simp$ is \defterm{inert} if
  it is the inclusion of a sub-interval of $[m]$, i.e. $f(i) = f(0)+i$
  for all $i$, and \defterm{active} if it preserves the extremal
  elements, i.e.  $f(0) = 0$ and $f(n) = m$. More generally, we say a
  morphism $(f_{1},\ldots,f_{n})$ in $\Dn$ is \emph{inert} or
  \emph{active} if each $f_{i}$ is inert or active. We write
  $\Dn_{\txt{act}}$ and $\Dn_{\txt{int}}$ for the subcategories of
  $\Dn$ with active and inert morphisms, respectively.
\end{defn}

\begin{lemma}
  The active and inert morphisms form a factorization system on
  $\Dn$.
\end{lemma}
\begin{proof}
  This is a special case of \rcite{BarwickOpCat}{Lemma 8.3}; it is
  also easy to check by hand.
\end{proof}

\begin{remark}
  Since the objects of $\Dn$ have no non-trivial automorphisms, the
  factorizations into active and inert morphisms are actually
  \emph{strictly} unique, rather than just unique up to isomorphism.
\end{remark}

\begin{defn}
  Let $S$ be a subset of $\{1,\ldots,n\}$. We write $C_{S} :=
  ([i_{1}], \ldots, [i_{n}])$ where $i_{j}$ is $1$ for $j \in S$ and
  $0$ otherwise. We refer to the objects $C_{S}$ as \emph{cells} and
  write $\Celln$ for the full subcategory of $\Dn_{\txt{int}}$ spanned
  by the objects $C_{S}$ for all $S \subseteq \{1,\ldots,n\}$. Note
  that we have $C_{n} = C_{\{1,\ldots,n\}}$.
\end{defn}

\begin{remark}
  The category $\Celln$ is equivalent to the product
  $(\Cell^{1})^{\times n}$, where $\Cell^{1}$ is the category with
  objects $[0]$ and $[1]$ and the two inclusions $[0] \to [1]$ as its
  only non-identity morphisms.
\end{remark}

\begin{defn}
  For $I \in \Dn$, we write $\CellnI$ for the category
  $(\Dn_{\txt{int}})_{/I} \times_{\Dn_{\txt{int}}} \Celln$ of
  inert morphisms from cells to $I$.
\end{defn}

\begin{defn}
  Let $\mathcal{C}$ be an \icat{} with finite limits. An
  \defterm{$n$-uple category object} in $\mathcal{C}$ is a
  multisimplicial object
  $X_{\bullet,\ldots,\bullet} \colon \Dnop \to \mathcal{C}$ such that for all
  $I = ([i_{1}],\ldots,[i_{n}]) \in \simp$ the natural map
  \[ X_{I} \to \lim_{C \to I \in \CellnIop} X_{C} \]
  is an equivalence. We write $\Cat^{n}(\mathcal{C})$ for the full
  subcategory of $\Fun(\Dnop, \mathcal{C})$ spanned by the $n$-uple
  category objects.
\end{defn}

\begin{remark}
  To see that this is equivalent to iterating the definition of a
  category object in $\mathcal{C}$, observe that for $I =
  ([i_{1}],\ldots,[i_{n}])$ in $\Dn$, the category $\CellnI$ is simply
  the product $\Cell^{1}_{/[i_{1}]} \times \cdots \times
  \Cell^{1}_{/[i_{n}]}$, and so decomposing the limit we see that
  $X_{\bullet,\ldots,\bullet}$ is an $n$-uple category object \IFF{}
  $X_{i,\bullet,\ldots,\bullet}$ is an $(n-1)$-uple category object
  for all $i$, and $X_{\bullet}$ is a category object in
  $(n-1)$-simplicial objects in $\mathcal{C}$.
\end{remark}

If $\mathcal{C}$ is the \icat{} $\mathcal{S}$ of spaces, an $n$-uple category
object $X_{\bullet,\ldots,\bullet}$ can be thought of as consisting of
\begin{itemize}
 \item a space $X_{0,\ldots,0}$ of objects
  \item spaces $X_{1,0,\ldots,0}$, \ldots,
    $X_{0,\ldots,0,1}$ of $n$ different kinds of 1-morphism,
        each with a source and target in $X_{0,\ldots,0}$,
  \item spaces $X_{1,1,0,\ldots,0}$, etc., of ``commutative
    squares'' between any two kinds of 1-morphism,
  \item spaces $X_{1,1,1,0,\ldots,0}$, etc., of
    ``commutative cubes'' between any three kinds of 1-morphism,
  \item \ldots
  \item a space $X_{1,1,\ldots,1}$ of ``commutative
    $n$-cubes'',
\end{itemize}
together with units and coherently homotopy-associative composition
laws for all these different types of morphisms. In other words, an
$n$-uple category object in $\mathcal{S}$ can be regarded as an
\emph{$n$-uple \icat{}}.

\begin{remark}
  Since \icats{} can be thought of as (complete) Segal spaces,
  i.e. category objects in $\mathcal{S}$, we can think of $n$-uple
  category objects in $\CatI$ as \emph{$(n+1)$-uple \icats{}}. More
  precisely, regarding $\CatI$ as the \icat{} of complete Segal spaces
  we have an inclusion $\CatI \hookrightarrow \Cat(\mathcal{S})$, and
  this induces an inclusion $\Cat^{n}(\CatI) \hookrightarrow
  \Cat^{n+1}(\mathcal{S})$.
\end{remark}

\subsection{$\Dn$ and $(\infty,n)$-Categories}\label{subsec:inftyn}
We can view $(\infty,n)$-categories as given by the same kind of data
as an $n$-uple \icat{}, except that there is only one type of
1-morphism, so to define $(\infty,n)$-categories as a special kind of
$n$-uple \icat{} we want to require certain spaces to be
``trivial''. This leads to Barwick's definition of an $n$-fold Segal
object in an \icat{}:
\begin{defn}
  Suppose $\mathcal{C}$ is an \icat{} with finite limits. A
  \emph{1-fold Segal object} in $\mathcal{C}$ is just a category
  object in $\mathcal{C}$. For $n > 1$ we inductively define an
  \emph{$n$-fold Segal object} in $\mathcal{C}$ to be an $n$-uple
  category object $\mathcal{D}$ such that
  \begin{enumerate}[(i)]
  \item the $(n-1)$-uple category object
    $\mathcal{D}_{0,\bullet,\ldots,\bullet}$ is constant,
  \item the $(n-1)$-uple category object
    $\mathcal{D}_{k,\bullet,\ldots,\bullet}$ is an $(n-1)$-fold Segal
    object for all $k$.
  \end{enumerate}
  We write $\Seg_{n}(\mathcal{C})$ for the full subcategory of
  $\Cat^{n}(\mathcal{C})$ spanned by the $n$-fold Segal objects. When
  $\mathcal{C}$ is the \icat{} $\mathcal{S}$ of spaces, we refer to
  $n$-fold Segal objects in $\mathcal{S}$ as \emph{$n$-fold Segal spaces}.
\end{defn}
\begin{remark}
  Unwinding the definition, we see that an $n$-fold Segal space $X$ consists of
\begin{itemize}
\item a space $X_{0,\ldots,0}$ of objects,
\item a space $X_{1,0,\ldots,0}$ of 1-morphisms,
\item a space $X_{1,1,0,\ldots,0}$ of 2-morphisms,
\item \ldots
\item a space $X_{1,\ldots,1}$ of $n$-morphisms,
\end{itemize}
together with units and coherently homotopy-associative composition
laws for these morphisms.
\end{remark}

Given a double category object $X \colon \simp^{2,\op} \to \mathcal{C}$,
there is a canonical way to extract a 2-fold Segal object $X'$:
\begin{itemize}
\item We take $X'_{0,\bullet}$ to be the constant simplicial object at $X_{0,0}$.
\item For $n > 0$ we define $X'_{n,\bullet}$ to be the pullback
  \nolabelcsquare{X'_{n,\bullet}}{X_{n,\bullet}}{X'_{0,\bullet}}{X_{0,\bullet},}
  where the bottom horizontal map is induced by the degeneracies. This
  amounts to forgetting the objects of $X_{0,1}$ that are not in the
  image of the degeneracy map $X_{0,0} \to X_{0,1}$ --- i.e. we are
  forgetting all the non-trivial 1-morphisms of one kind.
\end{itemize}
This construction can be iterated to extract an $n$-fold Segal object
from an $n$-uple category object --- in fact, by permuting the $n$
coordinates we can extract $n$ different Segal objects. More formally,
we have:
\begin{propn}[\rcite{spans}{Proposition 2.13}]
  Let $\mathcal{C}$ be an \icat{} with finite limits. The inclusion $\Seg_{n}(\mathcal{C}) \hookrightarrow
  \Cat^{n}(\mathcal{C})$ has a right adjoint $U_{\Seg} \colon
  \Cat^{n}(\mathcal{C}) \to \Seg_{n}(\mathcal{C})$. \qed
\end{propn}

Although $n$-fold Segal spaces describe $(\infty,n)$-categories, the
\icat{} $\Seg_{n}(\mathcal{S})$ is not the correct homotopy theory of
$(\infty,n)$-categories, as we have not inverted the appropriate class
of fully faithful and essentially surjective maps. This localization
can be obtained by restricting to the full subcategory
$\txt{CSS}_{n}(\mathcal{S})$ of \emph{complete} $n$-fold Segal spaces,
as proved by Barwick~\cite{BarwickThesis}; we denote the localization
$\Seg_{n}(\mathcal{S}) \to \txt{CSS}_{n}(\mathcal{S})$ by $L_{n}$, but
we will not need the details of the definition in this paper.

\begin{remark}\label{rmk:complsegnmon}
  There is a canonical way to extract an $(\infty,n)$-category from an
  $n$-uple \icat{} $\mathcal{C}$, namely the completion
  $L_{n}U_{\Seg}\mathcal{C}$ of the underlying $n$-fold Segal space of
  $\mathcal{C}$. Moreover, the functor $L_{n}U_{\Seg}\colon
  \Cat^{n}(\mathcal{S}) \to \Cat_{(\infty,n)}$ is symmetric monoidal
  with respect to the Cartesian product --- since $U_{\Seg}$ is a
  right adjoint it preserves products, and $L_{n}$ preserves products
  by \rcite{spans}{Lemma 2.21}. In particular, if $\mathcal{C}$ is an
  $\mathbb{E}_{m}$-monoidal $n$-uple \icat{}, then
  $L_{n}U_{\Seg}\mathcal{C}$ is an $\mathbb{E}_{m}$-monoidal
  $(\infty,n)$-category. Similarly, we can extract an underlying
  $(\infty,n+1)$-category from an $n$-uple category object
  $\mathcal{C}$ in $\CatI$ as $L_{n+1}U_{\Seg}i\mathcal{C}$ where $i$
  denotes the inclusion $\Cat^{n}(\CatI) \hookrightarrow
  \Cat^{n+1}(\mathcal{S})$. The functor $L_{n+1}U_{\Seg}i \colon
  \Cat^{n}(\CatI) \to \Cat_{(\infty,n+1)}$ also preserves products,
  since $i$ is another right adjoint.
\end{remark}

\subsection{$\DnI$ and Iterated Bimodules}\label{subsec:cartitbimod}
We will now consider how to extend the definition of the double
\icat{} $\fALG_{1}(\mathcal{C})$ of algebras, algebra homomorphisms,
and bimodules in $\mathcal{C}$ we outlined above to get an
$(n+1)$-uple \icat{} $\fALG_{n}(\mathcal{C})$ of $\En$-algebras.  We
take the \icat{} $\fALG_{1}(\mathcal{C})_{0,\ldots,0}$ of objects to
be the \icat{} of $\Dn$-monoids in $\mathcal{C}$ --- a full
subcategory of $\Fun(\Dnop, \mathcal{C})$. To define the remaining
structure, we first observe that we can iterate the definition of
$\simp_{/[i]}$-monoids to get a notion of $\DnI$-monoids for all $I
\in \Dn$:
\begin{defn}
  Let $\mathcal{C}$ be an \icat{} with products, and suppose $I \in
  \Dn$. A \emph{$\DnI$-monoid} in $\mathcal{C}$ is a functor
  $X \colon \DnIop \to \mathcal{C}$ such that for every object $\phi \colon J
  \to I$, the natural map
  \[ X(\phi) \to \prod_{\alpha \in |J|} X(\phi\circ\alpha)\] is
  an equivalence.
\end{defn}
Just as in the case $n = 1$, however, we do not want
$\fALGn(\mathcal{C})_{I}$ to contain \emph{all} the $\DnI$-monoids,
only those that are ``composite'' in the sense that they decompose
appropriately as tensor products. We will define this notion precisely
below in \S\ref{subsec:Enbimod}. The main result of this paper,
restricted to the Cartesian case, is then that this does indeed give
an $(n+1)$-uple \icat{}. More precisely, if for every $I \in \Dn$, we let
$\fALG_{n}(\mathcal{C})_{I}$ denote the \icat{} of composite
$\DnI$-monoids (a full subcategory of $\Fun(\DnIop,
\mathcal{C})$), then:
  \begin{itemize}
  \item The composite monoids are preserved under composition with the
    maps $\DnI \to \Dn_{/J}$ coming from maps in
    $\Dn$. Thus the objects $\fALG_{n}(\mathcal{C})_{\bullet,\ldots,\bullet}$
    define a multisimplicial \icat{}.
  \item The \icats{} $\fALG_{n}(\mathcal{C})_{\bullet,\ldots,\bullet}$
    satisfy the Segal condition, i.e. the map
  \[ \fALG_{n}(\mathcal{C})_{I} \to \lim_{C\to I \in \CellnIop}
  \fALG_{n}(\mathcal{C})_{C}\] is an equivalence for all $n$.
  \end{itemize}
In other words, $\fALG_{n}(\mathcal{C})$ is an $n$-uple category object in
$\CatI$. From this we can then extract an $(\infty,n+1)$-category
$\fAlg_{n}(\mathcal{C})$ as the underlying complete $(n+1)$-fold Segal
space $L_{n+1}U_{\Seg}i\fALG_{n}(\mathcal{C})$, as discussed above.

\section{Algebras and Bimodules}\label{sec:algbimod}
In \S\ref{sec:cartalgbimod} we sketched our approach to constructing a
double \icat{} of algebras and bimodules in the Cartesian case,
i.e. when the algebras are defined with respect to the monoidal
structure given by the Cartesian product. However, although this case
is certainly not without interest, many key examples of symmetric
monoidal \icats{} where we want to consider algebras and bimodules
have non-Cartesian tensor products --- for example: spectra, modules
over a ring spectrum, or the ``derived \icat{}'' of chain complexes in
an abelian category with quasi-isomorphisms inverted. To extend our
definitions to apply also to such non-Cartesian examples, we will work
with the theory of \emph{$\infty$-operads}. Specifically, in this
section we will make use of the theory of \emph{non-symmetric
  \iopds{}} to construct a double \icat{} $\fALG_{1}(\mathcal{C})$ of
associative algebras in any nice monoidal \icat{} $\mathcal{C}$, with
algebra homomorphisms and bimodules as the two kinds of 1-morphisms.

In \S\ref{subsec:nsop} we recall the basics of \nsiopds{}, and then in
\S\ref{subsec:bimodopd} we observe that using these the definition of
bimodules we discussed above in \S\ref{sec:cartalgbimod} has a
natural extension to the non-Cartesian setting, which lets us define
the \icats{} $\fALG_{1}(\mathcal{C})_{k}$ that will make up the
simplicial \icat{} $\fALG_{1}(\mathcal{C})$. In
\S\ref{subsec:alg1segcond} we check that these \icats{} satisfy the Segal
condition, and in \S\ref{subsec:functcomp} we show that they are
functorial and so do indeed form a simplicial object in
$\CatI$. Finally, in \S\ref{subsec:bimodfib} we study the forgetful
functor from bimodules to pairs of algebras in more detail --- the
results we prove here will be used below in \S\ref{subsec:Algnmaps}.

\subsection{Non-Symmetric $\infty$-Operads}\label{subsec:nsop}
In this subsection we will review some basic notions from the theory
of non-symmetric \iopds{}. For more motivation for these definitions,
we refer the reader to the extensive discussion in \rcite{enr}{\S
  2.1--2.2}. 

In ordinary category theory a monoidal category can be viewed as being
precisely an associative monoid in the 2-category of categories,
provided we interpret ``associative monoid'' in an appropriately
2-categorical sense. Similarly, we can define a \emph{monoidal
  \icat{}} to be an associative monoid in the \icat{} $\CatI$ of
\icats{}. As we saw in \S\ref{subsec:assalg}, we can take this to mean
a simplicial object in $\CatI$ satisfying a ``Segal condition''. Using
Lurie's straightening equivalence, we get an equivalent definition of monoidal
\icats{} as certain coCartesian fibrations over $\Dop$:
\begin{defn}
  A \emph{monoidal \icat{}} is a coCartesian fibration
  $\mathcal{C}^{\otimes}\to \simp^{\op}$ such that for each $[n]$ the map
  $\mathcal{C}^{\otimes}_{[n]} \to
  (\mathcal{C}^{\otimes}_{[1]})^{\times n}$, induced by the
  coCartesian morphisms over the maps $\rho_{i}$ in $\simp^{\op}$, is
  an equivalence.
\end{defn}

One advantage of this definition is that it  can be weakened to give a
definition of \emph{non-symmetric \iopds{}}:
\begin{defn}\label{defn:nsiopd}
  A \emph{non-symmetric \iopd{}} is a functor of \icats{} $\pi \colon
  \mathcal{O} \to \simp^{\op}$ such that:
  \begin{enumerate}[(i)]
  \item For every inert morphism $\phi \colon [m] \to [n]$  in
    $\simp^{\op}$ and every $X \in \mathcal{O}_{[n]}$ there exists
    a $\pi$-coCartesian morphism $X \to \phi_{!}X$ over $\phi$.
  \item For every $[n] \in \simp^{\op}$ the functor
    \[\mathcal{O}_{[n]} \to
    (\mathcal{O}_{[1]})^{\times n}\] induced by the
    coCartesian morphisms over the inert maps $\rho_{i}$ ($i =
    1,\ldots,n$) is an equivalence of \icats{}.
  \item For every morphism $\phi \colon [n] \to [m]$ in $\simp^{\op}$,
    $X \in \mathcal{O}_{[n]}$,
    and $Y \in \mathcal{O}_{[m]}$, composition with the
    coCartesian morphisms $Y \to Y_{i}$ over the inert morphisms
    $\rho_{i}$ gives an equivalence
    \[ \Map_{\mathcal{O}}^{\phi}(X, Y) \isoto \prod_{i}
    \Map_{\mathcal{O}}^{\rho_{i} \circ \phi}(X, Y_{i}),\] where
    $\Map_{\mathcal{O}}^{\phi}(X,Y)$ denotes the subspace of
    $\Map_{\mathcal{O}}(X,Y)$ of morphisms that map to $\phi$ in
    $\simp^{\op}$. (Equivalently, $Y$ is a \emph{$\pi$-limit} of the
    $Y_{i}$'s in the sense of \cite[\S 4.3.1]{HTT}.)
  \end{enumerate}
\end{defn}

\begin{remark}
  To see how this definition is related to the usual notion of
  non-symmetric (coloured) operad (or multicategory), recall that to
  any non-symmetric (coloured) operad (or multicategory) in sets we
  can associate its \emph{category of operators}, which is a category
  over $\Dop$. These categories of operators are charaterized
  precisely by the 1-categorical analogues of conditions (i)--(iii)
  above --- for more details see \rcite{enr}{\S 2.2}.
\end{remark}

\begin{remark}
  This definition is a special case of Barwick's notion of an \iopd{}
  over an operator category \cite{BarwickOpCat}, namely the case where
  the operator category is the category $\mathbb{O}$ of finite ordered
  sets.
\end{remark}

\begin{remark}
  Since $\Dop$ is an ordinary category, a map $\mathcal{O} \to \Dop$
  where $\mathcal{O}$ is an \icat{} is automatically an inner
  fibration by \rcite{HTT}{Proposition 2.3.1.5}.
\end{remark}

\begin{defn}
  If $\mathcal{O}$ and $\mathcal{P}$ are
  \nsiopds{}, a \defterm{morphism of \nsiopds{}} from
  $\mathcal{O}$ to $\mathcal{P}$ is a commutative
  diagram
  \opctriangle{\mathcal{O}}{\mathcal{P}}{\simp^{\op}}{\phi}{}{}
  such that $\phi$ carries coCartesian morphisms in
  $\mathcal{O}$ that map to inert morphisms in $\simp^{\op}$
  to coCartesian morphisms in $\mathcal{P}$. We will also
  refer to a morphism of \nsiopds{} $\mathcal{O} \to
  \mathcal{P}$ as an \emph{$\mathcal{O}$-algebra}
  in $\mathcal{P}$. We write $\Alg^{1}_{\mathcal{O}}(\mathcal{P})$ for
  the \icat{} of $\mathcal{O}$-algebras in $\mathcal{P}$, defined as a
  full subcategory of the \icat{} of functors from $\mathcal{O}$ to
  $\mathcal{P}$ over $\Dop$.
\end{defn}

We will actually need to work with a somewhat more general notion than
that of non-symmetric \iopd{}. To introduce this, recall from
\S\ref{subsec:segsp} that a \emph{double \icat{}} can be defined as a
simplicial object in $\CatI$ that satisfies a more general variant of
the Segal condition that defines monoids. Reformulating this in terms
of coCartesian fibrations, we get the following analogue of our
definition of a monoidal \icat{} above:
\begin{defn}
  A \emph{double \icat{}} is a coCartesian fibration $\mathcal{M} \to
  \simp^{\op}$ such that for each $[n]$ the map
  \[\mathcal{M}_{[n]} \to \mathcal{M}_{[1]} \times_{\mathcal{M}_{[0]}}
  \cdots \times_{\mathcal{M}_{[0]}} \mathcal{M}_{[1]},\] induced by
  the coCartesian morphisms over the maps $\rho_{i}$ and the maps $[n]
  \to [0]$ in $\simp^{\op}$, is an equivalence.
\end{defn}
Now we can contemplate the analogous variant of the definition of a
\nsiopd{}:
\begin{defn}\label{defn:gnsiopd}
  A \emph{generalized non-symmetric \iopd{}} is a functor of \icats{}
  $\pi \colon \mathcal{O} \to \simp^{\op}$ such that:
  \begin{enumerate}[(i)]
  \item For every inert morphism $\phi \colon [m] \to [n]$  in
    $\simp^{\op}$ and every $X \in \mathcal{O}_{[n]}$ there exists
    a $\pi$-coCartesian morphism $X \to \phi_{!}X$ over $\phi$.
  \item For every $[n] \in \simp^{\op}$ the functor
    \[\mathcal{O}_{[n]} \to \mathcal{O}_{[1]}
    \times_{\mathcal{O}_{[0]}} \cdots \times_{\mathcal{O}_{[0]}}
    \mathcal{O}_{[1]}\] induced by the coCartesian arrows over the
    inert maps $\rho_{i}$ ($i = 1,\ldots,n$) and the maps $[n] \to
      [0]$ is an equivalence of \icats{}.
    \item Given $Y \in \mathcal{O}_{[m]}$, choose a coCartesian lift
      of the diagram of inert morphisms from $[m]$ to $[1]$ and $[0]$:
      let $Y \to Y_{(i-1)i}$ be a coCartesian morphism over the map
      $\rho_{i} \colon [m] \to [1]$ ($i = 1,\ldots,m$) and let $Y \to
      Y_{i}$ ($i =0,\ldots,m$) be a coCartesian morphism over the map
      $\sigma_{i}\colon [m] \to [0]$ corresponding to the inclusion of
      $\{i\}$ in $[m]$. Then for any map $\phi \colon [n] \to [m]$ in
      $\Dop$ and $X \in \mathcal{O}_{[n]}$, composition with these
      coCartesian morphisms induces an equivalence
      \[ \Map_{\mathcal{O}}^{\phi}(X, Y) \isoto
      \Map_{\mathcal{O}}^{\rho_{1} \circ \phi}(X, Y_{01})
      \times_{\Map_{\mathcal{O}}^{\sigma_{1} \circ \phi}(X, Y_{1})}
      \cdots \times_{\Map_{\mathcal{O}}^{\sigma_{m-1} \circ \phi}(X,
        Y_{m-1})} \Map_{\mathcal{O}}^{\rho_{1} \circ \phi}(X,
      Y_{(m-1)m}).\] (Equivalently, any coCartesian lift of the
      diagram of inert maps from $[m]$ to $[1]$ and $[0]$ is a
      $\pi$-limit diagram in $\mathcal{O}$.)
  \end{enumerate}
\end{defn}
\begin{remark}
  As discussed in \rcite{enr}{\S 2.3--2.4}, \gnsiopds{} are an
  \icatl{} analogue of the \textbf{fc}-\emph{multicategories} of Leinster
  \cite{LeinsterHigherOpds} (also called \emph{virtual double
    categories} in \cite{CruttwellShulman}), which are a common
  generalization of double categories and multicategories.
\end{remark}
We can define morphisms of \gnsiopds{} in the same way as we define
morphisms of \nsiopds{}, i.e. as maps over $\Dop$ that preserve
coCartesian morphisms over inert morphisms. Again, we will refer to a
morphism $\mathcal{M} \to \mathcal{N}$ of \gnsiopds{} as an
\emph{$\mathcal{M}$-algebra} in $\mathcal{N}$, and define an \icat{}
$\Algns_{\mathcal{M}}(\mathcal{N})$ of these as a full subcategory of
the \icat{} of functors from $\mathcal{M}$ to $\mathcal{N}$ over
$\Dop$.

\subsection{Bimodules and their Tensor Products}\label{subsec:bimodopd}
We now have a natural way to extend the definitions of
\S\ref{sec:cartalgbimod} to the non-Cartesian setting because of the
following observation:
\begin{lemma}
  The projection $\simp_{/[n]}^{\op} \to \simp^{\op}$ is a double
  \icat{} for all $[n] \in \simp$.
\end{lemma}
\begin{proof}
  This projection is the opfibration  associated to the functor
  \[ \Hom_{\simp}(\blank, [n]) \colon \simp^{\op} \to \Set.\] It thus
  suffices to check that this functor satisfies the Segal
  condition, which it does since $[k]$ is the iterated pushout $[1]
  \amalg_{[0]} \cdots \amalg_{[0]} [1]$ in $\simp$.
\end{proof}

\begin{remark}
  As a double ($\infty$-)category, $\Dop_{/[n]}$ is rather degenerate:
  it is the double category corresponding to the category (or
  partially ordered set)
  \[ 0 \to 1 \to \cdots \to n.\] 
  In particular, it has no non-trivial
  morphisms in one direction.
\end{remark}

\begin{defn}
  Let $\mathcal{C}$ be a monoidal \icat{}. An \emph{associative
    algebra object} in $\mathcal{C}$ is a $\simp^{\op}$-algebra, and a
  \emph{bimodule} in $\mathcal{C}$ is a $\simp^{\op}_{/[1]}$-algebra.
\end{defn}
Thus, to define the double \icat{} $\fALG_{1}(\mathcal{C})$, a
natural choice for the \icat{} of objects is
$\Alg^{1}_{\Dop}(\mathcal{C})$ and for the \icat{} of morphisms it is
$\Alg^{1}_{\Dop_{/[1]}}(\mathcal{C})$. At the next level, we want to
consider a full subcategory of
$\Alg^{1}_{\simp^{\op}_{/[2]}}(\mathcal{C})$ consisting of
``composite'' $\simp^{\op}_{/[2]}$-algebras. We want the composition
of bimodules in $\fALG_{1}(\mathcal{C})$ to be given by tensor
products, so the composite $\simp_{/[2]}^{\op}$-algebras should be
those algebras $M$ where $M(0,2)$ is exhibited as the tensor product
$M(0,1)\otimes_{M(1,1)}M(1,2)$. As discussed in
\S\ref{subsec:simp2tens}, this amounts to the diagram
$\simp_{+}^{\op} \to \mathcal{C}$,
obtained by taking the coCartesian pushforward of 
\[ \Dop_{+} \xto{j} \simp_{/[2]}^{\op} \to \mathcal{C}^{\otimes} \] to
the fibre over $[1]$, being a colimit diagram. To get a more
convenient version of this condition, and its generalization to
$\simp_{/[n]}$-algebras, it will be useful to reformulate it in terms
of operadic Kan extensions. In order to do this, we must first
introduce some notation:
\begin{defn}
  A morphism $\phi \colon [k] \to [m]$ in $\simp$ is \emph{cellular}
  if $\phi(i+1) \leq \phi(i)+1$ for all $i = 0,\ldots,k$. We write
  $\Lbrn$ for the full subcategory of $\simp_{/[n]}$ spanned by the
  cellular maps. (In other words, $\Lbrn$ is the full subcategory of
  $\simp_{/[n]}$ spanned by the objects $(i_{0},\ldots,i_{k})$ where
  $i_{t+1}-i_{t} \leq 1$.)
\end{defn}

\begin{lemma}\label{lem:Lbrnopgnsiopd}
  The projection $\Lbrnop \to \simp^{\op}$ is a \gnsiopd{}, and the
  inclusion $\tau_{n} \colon \Lbrnop \hookrightarrow
  \simp^{\op}_{/[n]}$ is a morphism of \gnsiopds{}.
\end{lemma}

This is a special case of the following observation:
\begin{lemma}\label{lem:fullsubgnsiopd}
  Suppose $\pi \colon \mathcal{O} \to \Dop$ is a \gnsiopd{} and $\mathcal{C}$ a
  full subcategory of $\mathcal{O}_{[1]}$. Let $\mathcal{P}$
  be the full subcategory of $\mathcal{O}$ spanned by the objects $X$
  such that $\rho_{i,!}X$ lies in $\mathcal{C}$ for all inert maps
  $\rho_{i} \colon [1] \to \pi(X)$. Then the restricted projection
  $\mathcal{P} \to \Dop$ is also a \gnsiopd{}, and the inclusion
  $\mathcal{P} \hookrightarrow \mathcal{O}$ is a morphism of \gnsiopds{}.
\end{lemma}
\begin{proof}
  If $X \in \mathcal{P}_{[n]}$, $\phi \colon [m] \to [n]$ is an inert
  map, and $X \to \phi_{!}X$ is a coCartesian morphism over $\phi$ in
  $\mathcal{O}$, then $\phi_{!}X$ is also in $\mathcal{P}$. Hence
  $\mathcal{P}$ has coCartesian morphisms over inert morphisms in
  $\Dop$, which is condition (i) in Definition~\ref{defn:gnsiopd}, and
  the inclusion $\mathcal{P} \hookrightarrow \mathcal{O}$ preserves
  these. Moreover, for every $[n]$ we have a pullback diagram
  \nolabelcsquare{\mathcal{P}_{[n]}}{\mathcal{C}
    \times_{\mathcal{O}_{[0]}} \cdots \times_{\mathcal{O}_{[0]}}
    \mathcal{C}}{\mathcal{O}_{[n]}}{\mathcal{O}_{[1]}
    \times_{\mathcal{O}_{[0]}} \cdots \times_{\mathcal{O}_{[0]}}
    \mathcal{O}_{[1]}} which implies condition (ii) since the bottom
  horizontal map is an equivalence. Condition (iii) is also satisfied,
  since $\mathcal{P}$ is a full subcategory.
\end{proof}

\begin{proof}[Proof of Lemma~\ref{lem:Lbrnopgnsiopd}]
  A map $\phi \colon [m] \to [n]$ is cellular \IFF{}
  all its composites $\phi \rho_{i} \colon [1] \to [n]$ with the inert
  maps $[1] \to [m]$ is cellular. Thus $\Lbrnop$ is the full
  subcategory of $\simp_{/[n]}^{\op}$ determined by a full subcategory
  over $[1]$ and so is a \gnsiopd{} by Lemma~\ref{lem:fullsubgnsiopd}.
\end{proof}

The $\simp_{/[n]}^{\op}$-algebras that are given by tensor products in
the appropriate way will turn out to be those that are left operadic
Kan extensions along the inclusion $\tau_{n} \colon \Lbrnop
\hookrightarrow \simp_{/[n]}^{\op}$. For this to make sense, we must
first check that the map $\tau_{n}$ is extendable in the sense of Definition~\ref{defn:extendable}, so that we can apply
Proposition~\ref{propn:lokeexist}:
\begin{propn}\label{propn:tau1ext}
  The inclusion $\tau_{i}\colon \bbLambda^{\op}_{/[i]} \to \simp^{\op}_{/[i]}$ is
  extendable for all $i$.
\end{propn}
\begin{proof}
  We must show that for any map $\xi \colon [j]\to [i]$ in $\simp$,
  the map
  \[ (\bbLambda_{/[i]}^{\op})^{\txt{act}}_{/\xi} \to \prod_{p = 1}^{j}
  (\bbLambda_{/[i]}^{\op})^{\txt{act}}_{/\xi \rho_{p}} \]
  is cofinal, or equivalently that the map
  \[ (\bbLambda_{/[i]})^{\txt{act}}_{\xi/} \to \prod_{p = 1}^{j}
  (\bbLambda_{/[i]})^{\txt{act}}_{\xi \rho_{p}/} \]
  is coinitial, where $\rho_{p} \colon [1] \to [j]$ is the inert map
  sending $0$ to $p-1$ and $1$ to $p$. By 
  \rcite{HTT}{Theorem 4.1.3.1}, to see this it suffices to show
  that for every $X \in \prod_{p = 1}^{j}
  (\bbLambda_{/[i]})^{\txt{act}}_{\xi \rho_{p}/}$, the \icat{} 
  $((\bbLambda_{/[i]})^{\txt{act}}_{\xi/})_{/X}$ is weakly
  contractible.
 
  The object $X$ is given by diagrams 
  \opctriangle{{[1]}}{{[n_{p}]}}{{[i]}}{f_{p}}{\xi \rho_{p}}{c_{p}}
  for $p = 1,\ldots,j$, where $f_{p}$ is active and $c_{p}$ is
  cellular. But since the $f_{p}$'s are active we see that
  \[c_{p}(n_{p}) = c_{p}f_{p}(1) = \xi (p) = c_{p+1}f_{p+1}(0) =
  c_{p+1}(0), \]
  so the $c_{p}$'s glue together to a unique map $c \colon [n] \to
  [i]$ such that $c\eta_{p} = c_{p}$, where $n = \sum_{p = 1}^{j}
  n_{p}$ and $\eta_{p}\colon [n_{p}] \to [n]$ is the inert map
  $\eta_{p}(q) = n_{1}+\cdots+n_{p-1} + q$. Moreover, $c$ is clearly
  cellular. The maps $f_{p}$ then glue to an active map $f \colon [j]
  \to [n]$ given by $f(p) = n_{1}+ \ldots + n_{p}$. The
  resulting object
  \opctriangle{{[j]}}{{[n]}}{{[i]}}{f}{\xi}{c}
  is then final in $((\bbLambda_{/[i]})^{\txt{act}}_{\xi/})_{/X}$, hence
  this \icat{} is indeed weakly contractible.
\end{proof}

The following observation lets us analyze operadic Kan extensions
along $\tau_{n}$:
\begin{lemma}\label{lem:Lbrnsift}
  For all $(i,i+k) \in (\simp^{\op}_{/[n]})_{[1]}$ (with $k \geq 1$)
  the functor $\simp^{(k-1),\op} \to (\Lbrnop)^{\txt{act}}_{/(i,i+k)}$ that
  sends $([a_{1}],\ldots,[a_{k-1}])$ to $(i,
  i+1,\ldots,i+1,\ldots,i+(k-1),\ldots,i+(k-1),i+k)$, where there are
  $a_{j}+1$ copies of $i+j$, is cofinal. In particular, there is a
  cofinal map from a product of copies of
  $\simp^{\op}$ to $(\Lbrnop)^{\txt{act}}_{/(i,j)}$ for all $i,j$, and
  so a cofinal map from $\simp^{\op}$ by \rcite{HTT}{Lemma
    5.5.8.4}; the simplicial set $(\Lbrnop)^{\txt{act}}_{/(i,i+k)}$ is
  thus sifted.
\end{lemma}
\begin{proof}
  This follows from \rcite{HTT}{Theorem 4.1.3.1}, since the category
  $(\simp^{(k-1),\op})_{X/}$ has an initial object for all $X
  \in (\Lbrnop)^{\txt{act}}_{/(i,i+k)}$.
\end{proof}

\begin{defn}\label{defn:mongrtp}
  We say a monoidal \icat{} has \emph{good relative tensor products}
  if it is $\tau_{n}$-compatible (in the sense of
  Definition~\ref{defn:icompatible}) for all $n$. Similarly, we say a
  monoidal functor is \emph{compatible with relative tensor products}
  if it is $\tau_{n}$-compatible (in the sense of
  Definition~\ref{defn:ftricomp}) for all $n$.
\end{defn}

\begin{lemma}\label{lem:mongrtpcond}
  Let $\mathcal{C}$ be a monoidal \icat{}. Then $\mathcal{C}$ has good
  relative tensor products \IFF{} for every algebra $A \colon
  \bbLambda_{/[2]}^{\op} \to \mathcal{C}^{\otimes}$, the diagram
  $\simp \to (\bbLambda_{/[2]}^{\op})^{\txt{act}}_{/(0,2)} \to
  \mathcal{C}$, obtained from $A$ by coCartesian pushforward to the
  fibre over $[1]$, has a colimit, and this colimit is preserved by
  tensoring (on either side) with any object of
  $\mathcal{C}$. Moreover, a monoidal functor is compatible with
  relative tensor products \IFF{} it preserves these colimits.
\end{lemma}
\begin{proof}
  Follows from Lemma~\ref{lem:Lbrnsift} and
  Corollary~\ref{cor:moncolimprod}.
\end{proof}

Applying Corollary~\ref{cor:opdkanext}, we get:
\begin{cor}
  Supppose $\mathcal{C}$ is a monoidal \icat{} with good relative
  tensor products. Then the restriction $\tau_{n}^{*} \colon
  \Alg^{1}_{\simp^{\op}_{/[n]}}(\mathcal{C}) \to
  \Alg^{1}_{\Lbrnop}(\mathcal{C})$ has a fully faithful left adjoint
  $\tau_{n,!}$. A $\simp^{\op}_{/[n]}$-algebra $M$ is in the image of
  $\tau_{n,!}$ \IFF{} $M$ exhibits $M(i,j)$ as the tensor product
  \[M(i,i+1) \otimes_{M(i+1,i+1)} M(i+1,i+2) \otimes_{M(i+2,i+2)}
  \cdots \otimes_{M(j-1,j-1)} M(j-1,j).\qed\]
\end{cor}

Thus, the following is a good definition of the \icats{}
$\fALG_{1}(\mathcal{C})_{n}$ for all $n$:
\begin{defn}
  Let $\mathcal{C}$ be a monoidal \icat{} with good relative tensor
  products. We say that a $\simp_{/[n]}^{\op}$-algebra $M$ in
  $\mathcal{C}$ is \emph{composite} if the counit map
  $\tau_{n,!}\tau_{n}^{*}M \to M$ is an equivalence, or equivalently
  if $M$ is in the essential image of the functor $\tau_{n,!}$. We
  write $\fALG_{1}(\mathcal{C})_{n}$ for the full subcategory of
  $\Alg^{1}_{\simp^{\op}_{/[n]}}(\mathcal{C})$ spanned by the
  composite $\simp^{\op}_{/[n]}$-algebras.
\end{defn}

\subsection{The Segal Condition}\label{subsec:alg1segcond}
Our goal in this subsection is to prove that the \icats{}
$\fALG_{1}(\mathcal{C})_{i}$ satisfy the Segal condition, i.e. that
the natural map
\[ \fALG_{1}(\mathcal{C})_{i} \to \fALG_{1}(\mathcal{C})_{1}
\times_{\fALG_{1}(\mathcal{C})_{0}} \cdots
\times_{\fALG_{1}(\mathcal{C})_{0}} \fALG_{1}(\mathcal{C})_{1} \] is
an equivalence of \icats{}. We will prove this by showing that for
every $i$ the \gnsiopd{} $\Lbriop$ is equivalent to the colimit
$\simp^{\op}_{/[1]} \amalg_{\simp^{\op}_{/[0]}} \cdots
\amalg_{\simp^{\op}_{/[0]}} \simp^{\op}_{/[1]}$ in $\OpdInsg$. To do
this we use the model category
$(\sSet^{+})_{\mathfrak{O}_{1}^{\txt{gen}}}$ defined in
\S\ref{subsec:dniopdcat} and check that $\Lbriop$ is a homotopy
colimit; this boils down to checking that a certain
map is a trivial cofibration.

We write $\simp_{/[i]}^{\amalg,\op}$ for the ordinary colimit
$\simp^{\op}_{/[1]} \amalg_{\simp^{\op}_{/[0]}} \cdots
\amalg_{\simp^{\op}_{/[0]}} \simp^{\op}_{/[1]}$ in (marked) simplicial
sets (over $\simp^{\op}$). Since this colimit can be written as an
iterated pushout along injective maps of simplicial sets, this
colimit in simplicial sets is a homotopy colimit corresponding to the
\icatl{} colimit we're interested in. Moreover, there is an obvious
inclusion $\simp_{/[i]}^{\amalg,\op} \hookrightarrow \Lbriop$. Our aim
in this subsection is then to prove the following:

\begin{propn}\label{propn:cellmodel1}
  The inclusion $\simp_{/[i]}^{\amalg,\op} \hookrightarrow \Lbriop$ is
  a trivial cofibration in the model category
  $(\sSet^{+})_{\mathfrak{O}_{1}^{\txt{gen}}}$.
\end{propn}

Before we turn to the proof, let us first see that this does indeed
imply the Segal condition for $\fALG_{n}(\mathcal{C})$:
\begin{cor}\label{cor:algLbrnseg}
  Let $\mathcal{M}$ be a \gnsiopd{}. The restriction map
  \[\Algns_{\Lbrnop}(\mathcal{M}) \to
  \Algns_{\simp^{\op}_{/[1]}}(\mathcal{M})
  \times_{\Algns_{\simp^{\op}}(\mathcal{M})} \cdots
  \times_{\Algns_{\simp^{\op}}(\mathcal{M})}
  \Algns_{\simp^{\op}_{/[1]}}(\mathcal{M})\] is an equivalence of \icats{}.
\end{cor}
\begin{proof}
  Since the model category
  $(\sSet^{+})_{\mathfrak{O}_{1}^{\txt{gen}}}$ is enriched in
  marked simplicial sets and the inclusion $\simp_{/[n]}^{\amalg,\op}
  \hookrightarrow \Lbrnop$ is a trivial cofibration by
  Proposition~\ref{propn:cellmodel1}, for any \gnsiopd{} $\mathcal{M}$ the
  restriction map $\Algns_{\Lbrnop}(\mathcal{M}) \to
  \Algns_{\simp_{/[n]}^{\amalg,\op}}(\mathcal{M})$ is a trivial Kan
  fibration, and the map
  \[\Algns_{\simp_{/[n]}^{\amalg,\op}}(\mathcal{M}) \to   \Algns_{\simp^{\op}_{/[1]}}(\mathcal{M})
  \times_{\Algns_{\simp^{\op}}(\mathcal{M})} \cdots
  \times_{\Algns_{\simp^{\op}}(\mathcal{M})}
  \Algns_{\simp^{\op}_{/[1]}}(\mathcal{M})\] is an equivalence of
  \icats{} since
  $\simp_{/[n]}^{\amalg,\op}$ is a homotopy colimit.
\end{proof}

\begin{cor}\label{cor:ALG1seg}
  Let $\mathcal{C}$ be a monoidal \icat{} with good relative tensor
  products. Then the natural restriction map
  \[ \fALG_{1}(\mathcal{C})_{n} \to \fALG_{1}(\mathcal{C})_{1}
  \times_{\fALG_{1}(\mathcal{C})_{0}} \cdots
  \times_{\fALG_{1}(\mathcal{C})_{0}} \fALG_{1}(\mathcal{C})_{1}\]
  is an equivalence.
\end{cor}
\begin{proof}
  This map factors as a composite of the maps
  \[ \fALG_{1}(\mathcal{C})_{n} \to \Algns_{\Lbrnop}(\mathcal{C}) \to
  \fALG_{1}(\mathcal{C})_{1} \times_{\fALG_{1}(\mathcal{C})_{0}}
  \cdots \times_{\fALG_{1}(\mathcal{C})_{0}}
  \fALG_{1}(\mathcal{C})_{1},\] where the first is an equivalence by
  definition and the second by Corollary~\ref{cor:algLbrnseg}.
\end{proof}

We will deduce Proposition~\ref{propn:cellmodel1} from a rather
technical result about trivial cofibrations in
$(\sSet^{+})_{\mathfrak{O}_{1}^{\txt{gen}}}$. To state this, we first
need to introduce some terminology for simplices in the nerve of
$\Dop$.

\begin{warning}
  Throughout the remainder of this section we are really working with
  \emph{marked} simplicial sets. However, to simplify the notation we
  will not indicate the marking in any way --- thus if
  e.g. $\mathcal{O}$ is a \gnsiopd{} we are really thinking of it as
  the marked simplicial set $(\mathcal{O}, I)$ where $I$ is the
  collection of inert morphisms. Similarly, all simplicial subsets of
  \gnsiopds{} are really marked by the inert morphisms that they
  contain.
\end{warning}

\begin{defn}
  Let $\sigma$ be an $n$-simplex in $\mathrm{N}\Dop$, i.e. a diagram
  \[\sigma = [r_{0}] \xto{f_{1}} [r_{1}] \xto{f_{2}} \cdots
  \xto{f_{n}} [r_{n}]\] in $\Dop$ (where, in terms of the category
  $\simp$, each $f_{i}$ is a map of ordered sets from $[r_{i}]$ to
  $[r_{i-1}]$); for convenience, we will let the symbols $[r_{i}]$ and
  $f_{i}$
  denote
  the objects and morphisms in any such $n$-simplex we encounter from
  now on. We say that $\sigma$ is \emph{narrow} if $r_{n} = 1$ and
  \emph{wide} if $r_{n} > 1$. If $\sigma$ is wide, we have an induced
  diagram $\pi_{\sigma} \colon \Delta^{n} \star
  \mathrm{N}(\Cell^{1}_{/[r_{n}]})^{\op} \to \mathrm{N}\Dop$ by adding the inert
  morphisms from $[r_{n}]$ to $[1]$ and $[0]$. The \emph{decomposition
    simplices} of $\sigma$ are the simplices in the image of this
  diagram.
\end{defn}

\begin{defn}
  We say a morphism $\phi$ in $\simp^{\op}$ is \emph{neutral} if it is
  neither active nor inert. If $\sigma$ is an $n$-simplex of
  $\mathrm{N}\Dop$ such that $f_{k}$ is neutral, we say that $\sigma$
  is \emph{$k$-factorizable}. The \emph{$k$-factored $(n+1)$-simplex}
  of $\sigma$ is then that obtained by taking the inert-active
  factorization of $f_{k}$.
\end{defn}

From the definition of the model structure for a categorical pattern
$\mathfrak{P}$ in  \rcite{HA}{\S B.2} it follows that the
$\mathfrak{P}$-anodyne morphisms defined in \rcite{HA}{Definition
  B.1.1} are trivial cofibrations. In the case $\mathfrak{P} =
\mathfrak{O}_{1}^{\txt{gen}}$, we have in particular that:
\begin{itemize}
\item If $\sigma$ is a wide $n$-simplex in $\mathrm{N}\Dop$ and
  $\pi_{\sigma} \colon \Delta^{n} \star \mathrm{N}(\Cell^{1}_{/[r_{n}]})^{\op}
  \to \mathrm{N}\Dop$ is the diagram as above, then the inclusion
  \opctriangle{\partial \Delta^{n} \star  \mathrm{N}(\Cell^{1}_{/[r_{n}]})^{\op}}{\Delta^{n} \star  \mathrm{N}(\Cell^{1}_{/[r_{n}]})^{\op}}{\mathrm{N}\Dop}{}{}{\pi_{\sigma}} is a trival cofibration.
\item If $\sigma$ is a $k$-factorizable $n$-simplex and $\sigma'$ is
  its $k$-factored $(n+1)$-simplex, then the inclusion
  \opctriangle{\Lambda^{n+1}_{k}}{\Delta^{n+1}}{\mathrm{N}\Dop}{}{}{\sigma'}
  is a trivial cofibration.
\end{itemize}
We will prove Proposition~\ref{propn:cellmodel1} by constructing a
rather intricate filtration where each inclusion is a pushout of
a trivial cofibration of one of these two types. To define this we need some
more notation:
\begin{notation}
  We define the following sets of simplices in $\mathrm{N}\Dop$:
  \begin{itemize}
  \item For $1 \leq r < k \leq n$, let $A_{n}(k,r)$ be the set of
    non-degenerate narrow $n$-simplices $\sigma$ such that $f_{r}$ is
    inert, $f_{k}$ is neutral, and $f_{p}$ is active for $r < p < k$
    and $p > k$.
  \item For $1 \leq r < k \leq n$, let $A'_{n}(k,r)$ be
    the set of non-degenerate $(n+1)$-simplices $\sigma$ such that $r_{n} = 1$,
    $r_{n+1} = 0$, $f_{r}$ is inert, $f_{k}$ is neutral, and $f_{p}$
    is active for $r < p < k$ and $p > k$.
  \item For $1 \leq k \leq n$, let $B_{n}(k)$ be the set of
    non-degenerate narrow $n$-simplices $\sigma$ such that $f_{k}$ is
    neutral, $f_{p}$ is active for $p > k$, and $\sigma$ is not
    contained in $A_{n}(k,r)$ for any $r$.
  \item For $1 \leq k \leq n$, let $B'_{n}(k)$ be the set of
    non-degenerate $(n+1)$-simplices $\sigma$ such that $r_{n} = 1$,
    $r_{n+1} = 0$, $f_{k}$ is neutral, $f_{p}$ is active for $k < p <
    n+1$, and $\sigma$ is not contained in $A'_{n}(k,r)$ for
    any $r$.
  \end{itemize}
  Now define $\mathfrak{F}_{n} \subseteq \mathrm{N}\Dop$ to be the
  simplicial subset containing all the non-degenerate $i$-simplices
  for $i \leq n$ together with
  \begin{itemize}
  \item for every wide $i$-simplex, $i \leq n$, its decomposition
    simplices
  \item the $k$-factored $(i+1)$-simplices of the simplices in
    $A_{i}(k,r)$ and $B_{i}(k)$ for all $k,r$ and $i \leq n$.
  \item the $k$-factored $(i+2)$-simplices of the simplices in
    $A'_{i}(k,r)$ and $B'_{i}(k)$ for all $k,r$ and $i \leq n$.
  \end{itemize}
  Then let $\mathfrak{F}_{n}^{+}$ denote the simplicial subset
  containing the simplices in $\mathfrak{F}_{n}$ together with the
  narrow \emph{active} $(n+1)$-simplices, meaning those such that all
  the morphisms $f_{i}$ are active.
\end{notation}

A ``prototype'' version of our technical result is then: For every
$n$, the inclusion $\mathfrak{F}_{n-1}^{+} \hookrightarrow
\mathfrak{F}_{n}$ is a trivial cofibration in the \gnsiopd{} model
structure. We actually need a slightly more general ``relative''
version of this, which we are ready to state and prove after
introducing a little more notation:

\begin{notation}
  Let $\mathbf{O}$ be an ordinary category whose objects have no
  non-trivial automorphisms, equipped with a map $\mathbf{O} \to \Dop$
  that exhibits $\mathbf{O}$ as a \gnsiopd{}. We say a simplex in
  $\mathrm{N}\mathbf{O}$ is \emph{narrow}, \emph{wide} or
  \emph{$k$-factorizable} if this is true of its image in
  $\mathrm{N}\Dop$. For such $\mathbf{O}$ the inert-active
  factorizations in $\mathbf{O}$ are strictly unique (rather than just
  unique up to isomorphism), and we can define the decomposition
  simplices of a wide simplex and the $k$-factored $(n+1)$-simplex of
  a $k$-factorizable $n$-simplex just as before. If
  $\mathrm{N}\mathbf{O}_{0}$ is a simplicial subset of
  $\mathrm{N}\mathbf{O}$ we (slightly abusively) write
  $\mathfrak{F}_{n}\mathbf{O}$ for the simplicial subset of
  $\mathrm{N}\mathbf{O}$ containing the simplices in
  $\mathrm{N}\mathbf{O}_{0}$ together with those lying over the
  simplices in $\mathfrak{F}_{n}$; we also define
  $\mathfrak{F}_{n}^{+}\mathbf{O}$ similarly.
\end{notation}

\begin{propn}\label{propn:Otrivcofib}
  Let $\mathbf{O}$ be as above. Suppose $\mathrm{N}\mathbf{O}_{0}$ is
  a simplicial subset of $\mathrm{N}\mathbf{O}$ such that
  \begin{itemize}
  \item for every wide simplex contained in
    $\mathrm{N}\mathbf{O}_{0}$, its decomposition simplices are also
    contained in $\mathrm{N}\mathbf{O}_{0}$,
  \item for every $n$-simplex in $\mathrm{N}\mathbf{O}_{0}$ whose
    image in $\mathrm{N}\Dop$ is in $A_{n}(k,r)$ and $B_{n}(k)$ for
    some $k,r$, its $k$-factored $(n+1)$-simplex is also in
    $\mathrm{N}\mathbf{O}_{0}$,
  \item for every $(n+1)$-simplex in $\mathrm{N}\mathbf{O}_{0}$ whose image in
    $\mathrm{N}\Dop$ is in $A'_{n}(k,r)$ and $B'_{n}(k)$ for some $k,r$,
    its $k$-factored $(n+2)$-simplex is also in $\mathrm{N}\mathbf{O}_{0}$.
  \end{itemize}
  Then the inclusion
  \[\mathfrak{F}_{n-1}^{+}\mathbf{O}
  \hookrightarrow
  \mathfrak{F}_{n}\mathbf{O}\] is a trivial
  cofibration in the \gnsiopd{} model structure.
\end{propn}
\begin{remark}
  It is not really necessary to assume that the objects of
  $\mathbf{O}$ have no automorphisms for the proof to go through: it
  suffices, as in the proof of \rcite{HA}{Theorem 3.1.2.3}, to assume
  that the inert-active factorization system can be refined to a
  \emph{strict} factorization system, i.e. one where the
  factorizations are defined uniquely, not just up to
  isomorphism. This slight generalization is not needed for any of our
  applications, however.
\end{remark}

\begin{proof}
  The basic idea of the proof is to define a filtration of
  $\mathfrak{F}_{n}\mathbf{O}$, starting with
  $\mathfrak{F}_{n-1}^{+}\mathbf{O}$, such that each step in the
  filtration is a pushout of a trivial cofibration of one of the two
  types we discussed above.

  Let us say that a simplex in $\mathfrak{F}_{n}\mathbf{O}$ is
  \emph{old} if it is contained in $\mathfrak{F}_{n-1}^{+}\mathbf{O}$,
  and \emph{new} if it is not. We also write $\overline{A}_{n}(k,r)$
  for the set of new $n$-simplices whose image in $\mathrm{N}\Dop$
  lies in $A_{n}(k,r)$, and define $\overline{A}'_{n}(k,r)$,
  $\overline{B}_{n}(k)$ and $\overline{B}'_{n}(k)$
  similarly.
  The filtration is then defined as follows:
  \begin{itemize}
  \item Set $\mathcal{F}_{0} := \mathfrak{F}_{n-1}^{+}\mathbf{O}$.
  \item Let $S_{1}$ be the set of non-degenerate wide new
    $n$-simplices such that $f_{n}$ is inert. We let $\mathcal{F}_{1}$
    be the simplicial subset of $\mathfrak{F}_{n}\mathbf{O}$
    containing $\mathcal{F}_{0}$ together with the $n$-simplices in
    $S_{1}$ as well as their decomposition $(n+1)$- and $(n+2)$-simplices.
  \item Let $S_{2}(r)$ be the set of non-degenerate wide new
    $n$-simplices such that $f_{r}$ is inert and $f_{p}$ is active for
    $p > r$. We set $\mathcal{F}_{2}$ to be the simplicial subset of
    $\mathfrak{F}_{n}\mathbf{O}$ containing $\mathcal{F}_{1}$ together
    with:
    \begin{itemize}[--]
    \item the $n$-simplices in $S_{2}(r)$ for all $r$ and their
      decomposition $(n+1)$- and $(n+2)$-simplices,
    \item the $n$-simplices in $\overline{A}_{n}(k,r)$ for all $k,r$
      and their $k$-factored $(n+1)$-simplices,
    \item the $(n+1)$-simplices in $\overline{A}'_{n}(k,r)$ for all $k,r$
      and their $k$-factored $(n+2)$-simplices.
    \end{itemize}
  \item Let $\mathcal{F}_{3}$ be the simplicial subset of
    $\mathfrak{F}_{n}\mathbf{O}$ containing $\mathcal{F}_{2}$ together
    with the $n$-simplices in $\overline{B}_{n}(k)$ for all $k$ and
    their $k$-factored $(n+1)$-simplices, as well as the
    $(n+1)$-simplices in $\overline{B}'_{n}(k)$ for all $k$ and their
    $k$-factored $(n+2)$-simplices.
  \item Let $S_{4}$ be the set of non-degenerate wide new
    $n$-simplices that are not contained in $\mathcal{F}_{3}$. Then
    $\mathcal{F}_{4} := \mathfrak{F}_{n}\mathbf{O}$ consists of the
    simplices in $\mathcal{F}_{3}$ together with the $n$-simplices in
    $S_{4}$ and their decomposition $(n+1)$- and $(n+2)$-simplices.
  \end{itemize}
  We then need to prove that the four inclusions $\mathcal{F}_{m-1}
  \hookrightarrow \mathcal{F}_{m}$ are all trivial cofibrations.
  
  $m = 1$: If $\sigma$ is an $n$-simplex in $\mathrm{N}\mathbf{O}$, we
  write $\pi_{\sigma}$ for the induced diagram $\Delta^{n} \star
  \mathrm{N}(\Cell^{1}_{/[r_{n}]})^{\op} \to \mathrm{N}\mathbf{O}$ and
  $\pi^{\partial}_{\sigma}$ for the restriction of this map to
  $\partial\Delta^{n} \star \mathrm{N}(\Cell^{1}_{/[r_{n}]})^{\op}$. For $\sigma$
  in $S_{1}$, observe that since any narrow new $n$-simplex whose
  final map is inert is contained in $\mathcal{F}_{0}$, as is any new
  $(n+1)$-simplex whose final map is $[1] \to [0]$ and whose
  penultimate map is inert, the map $\pi^{\partial}_{\sigma}$ factors
  through $\mathcal{F}_{0}$. Thus we have a pushout diagram
  \nolabelcsquare{ \coprod_{\sigma \in S_{1}} \partial \Delta^{n}
    \star \mathrm{N}(\Cell^{1}_{/[r_{n}]})^{\op}}{ \coprod_{\sigma \in
      S_{1}} \Delta^{n} \star
    \mathrm{N}(\Cell^{1}_{/[r_{n}]})^{\op}}{\mathcal{F}_{0}}{\mathcal{F}_{1}.}
  Since the upper horizontal map is
  $\mathfrak{O}_{1}^{\txt{gen}}$-anodyne, so is the lower horizontal
  map.

  $m = 2$: This is the most convoluted step, as we must consider
  several subsidiary filtrations for the inclusion $\mathcal{F}_{1}
  \hookrightarrow \mathcal{F}_{2}$. We will inductively define a
  filtration
  \[ \mathcal{F}_{1} = \mathcal{G}_{n} \subseteq \mathcal{G}'_{n-1}
  \subseteq \mathcal{G}_{n-1} \subseteq \cdots \subseteq
  \mathcal{G}'_{1} \subseteq \mathcal{G}_{1} = \mathcal{F}_{2},\]
  where $\mathcal{G}'_{r}$ is itself defined via a filtration
  \[ \mathcal{G}_{r+1} = \mathcal{I}_{r,r} \subseteq
  \mathcal{I}'_{r,r+1} \subseteq \mathcal{I}_{r,r+1} \subseteq \cdots
  \subseteq \mathcal{I}'_{r,n} \subseteq \mathcal{I}_{r,n} =
  \mathcal{G}'_{r}.\]
  This goes as follows:
  \begin{itemize}
  \item We define $\mathcal{I}'_{r,k}$ to be the simplicial subset of
    $\mathcal{F}_{2}$ containing the simplices in $\mathcal{I}_{r,k-1}$
    together with the $n$-simplices in $\overline{A}_{n}(k,r)$ as well
    as their $k$-factored $(n+1)$-simplices.
  \item We define $\mathcal{I}_{r,k}$ to be the simplicial subset of
    $\mathcal{F}_{2}$ containing the simplices in $\mathcal{I}'_{r,k}$
    together with the $(n+1)$-simplices in $\overline{A}'_{n}(k,r)$ as well
    as their $k$-factored $(n+2)$-simplices.
  \item We define $\mathcal{G}_{r}$ to be the simplicial subset of
    $\mathcal{F}_{2}$ containing $\mathcal{G}'_{r}$ together with the
    $n$-simplices in $S_{2}(r)$ as well as their decomposition
    $(n+1)$- and $(n+2)$-simplices.
  \end{itemize}
  Then it suffices to show that the inclusions $f_{r,k} \colon
  \mathcal{I}_{r,k-1} \hookrightarrow \mathcal{I}'_{r,k}$, $g_{r,k}
  \colon \mathcal{I}'_{r,k} \hookrightarrow \mathcal{I}_{r,k}$, and
  $h_{r} \colon \mathcal{G}'_{r} \hookrightarrow \mathcal{G}_{r}$ are
  all trivial cofibrations.

  Observe that for $\sigma$ in $\overline{A}_{n}(k,r)$ with
  $k$-factored $(n+1)$-simplex $\tau$, the faces $d_{j}\tau$ with $j
  \neq k$ are contained in $\mathcal{I}_{r,k-1}$.
  Thus we get a pushout diagram \nolabelcsquare{ \coprod_{\sigma \in
      \overline{A}_{n}(k,r)} \Lambda^{n+1}_{k}}{ \coprod_{\sigma \in
      \overline{A}_{n}(k,r)}
    \Delta^{n+1}}{\mathcal{I}_{r,k-1}}{\mathcal{I}'_{r,k},} and so
  $f_{r,k}$ is a trivial cofibration.  Similarly, for $\sigma$ in
  $\overline{A}'_{n}(k,r)$ with $k$-factored $(n+2)$-simplex $\tau$
  the faces $d_{j}\tau$ with $j \neq k$ are contained in
  $\mathcal{I}'_{r,k}$. We therefore have another pushout diagram
  \nolabelcsquare{ \coprod_{\sigma \in \overline{A}'_{n}(k,r)}
    \Lambda^{n+2}_{k}}{ \coprod_{\sigma \in \overline{A}'_{n}(k,r)}
    \Delta^{n+2}}{\mathcal{I}'_{k,r}}{\mathcal{I}_{k,r},}  hence
  $g_{r,k}$ is also a trivial cofibration.

  Now for $\sigma \in S_{2}(r)$ the map $\pi_{\sigma}^{\partial}$
  factors through $\mathcal{G}'_{r}$, so we have a pushout square
  \nolabelcsquare{ \coprod_{\sigma \in S_{2}(k)} \partial \Delta^{n}
    \star \mathrm{N}(\Cell^{1}_{/[r_{n}]})^{\op}}{ \coprod_{\sigma \in S_{2}(k)}
    \Delta^{n} \star
    \mathrm{N}(\Cell^{1}_{/[r_{n}]})^{\op}}{\mathcal{G}'_{r}}{\mathcal{G}_{r},}
  which implies that $h_{r}$ is a trivial cofibration.

  $m = 3$: We again need to define a subsidiary filtration
  \[ \mathcal{F}_{2} = \mathcal{H}_{0} \subseteq \mathcal{H}'_{1}
  \subseteq \mathcal{H}_{1}
  \subseteq \cdots \subseteq \mathcal{H}'_{n} \subseteq \mathcal{H}_{n} = \mathcal{F}_{3}.\]
  Here we inductively define $\mathcal{H}'_{k}$ to be the subset of
  $\mathcal{F}_{3}$ containing $\mathcal{H}_{k-1}$ together with 
  the $n$-simplices in $\overline{B}_{n}(k)$ as well as their
  $k$-factored $(n+1)$-simplices, and then define $\mathcal{H}_{k}$ to
  be that containing $\mathcal{H}'_{k}$ together with the
  $(n+1)$-simplices in $\overline{B}'_{n}(k)$ as well as their
  $k$-factored $(n+1)$-simplices. It then suffices to prove that the
  inclusions $\mathcal{H}_{k-1} \hookrightarrow \mathcal{H}'_{k}$ and
  $\mathcal{H}'_{k} \hookrightarrow \mathcal{H}_{k}$ are trivial
  cofibrations. If $\sigma \in \overline{B}_{n}(k)$ and $\tau$ is its
  $k$-factored simplex, then $d_{j}\tau$ lies in $\mathcal{H}_{k-1}$
  for $j \neq k$, so we have a pushout square
 \nolabelcsquare{ \coprod_{\sigma \in \overline{B}_{n}(k)} \Lambda^{n+1}_{k}}{\coprod_{\sigma \in \overline{B}_{n}(k)} \Delta^{n+1}}
  {\mathcal{H}_{k-1}}{\mathcal{H}'_{k},}
  and hence the inclusion $\mathcal{H}_{k-1} \hookrightarrow
  \mathcal{H}'_{k}$ is a trivial cofibration. Similarly, we have a
  pushout square
  \nolabelcsquare{ \coprod_{\sigma \in \overline{B}'_{n}(k)}
    \Lambda^{n+2}_{k}}{\coprod_{\sigma \in \overline{B}'_{n}(k)}
    \Delta^{n+2}} {\mathcal{H}'_{k}}{\mathcal{H}_{k},} 
  so the inclusion $\mathcal{H}'_{k} \hookrightarrow \mathcal{H}_{k}$
  is also a trivial cofibration.

  $m = 4$: Observe that for $\sigma$ in $S_{4}$ the diagram
  $\pi^{\partial}_{\sigma}$ factors through $\mathcal{F}_{3}$, and so
  we have a pushout diagram \nolabelcsquare{ \coprod_{\sigma \in
      S_{4}} \partial \Delta^{n} \star \mathrm{N}(\Cell^{1}_{/[r_{n}]})^{\op}}{
    \coprod_{\sigma \in S_{4}} \Delta^{n} \star
    \mathrm{N}(\Cell^{1}_{/[r_{n}]})^{\op}}{\mathcal{F}_{3}}{\mathcal{F}_{4}.}
  The inclusion $\mathcal{F}_{3} \to \mathcal{F}_{4}$ is therefore also
  a trivial cofibration, which completes the proof.
\end{proof}

\begin{cor}\label{cor:Otrivcofib}
  Let $\mathbf{O}$ be an ordinary category whose objects have no
  non-trivial automorphisms, equipped with a map $\mathbf{O} \to
  \Dop$ that exhibits $\mathbf{O}$ as a \gnsiopd{}. Suppose
  $\mathrm{N}\mathbf{O}_{0}$ is a simplicial subset of
  $\mathrm{N}\mathbf{O}$ such that
  \begin{itemize}  
  \item every narrow active simplex in $\mathrm{N}\mathbf{O}$ is
    contained in $\mathrm{N}\mathbf{O}_{0}$,
  \item for every wide simplex contained in
    $\mathrm{N}\mathbf{O}_{0}$, its decomposition simplices are also
    contained in $\mathrm{N}\mathbf{O}_{0}$,
  \item for every $n$-simplex in $\mathrm{N}\mathbf{O}_{0}$ whose image in
    $\mathrm{N}\Dop$ is in $A_{n}(k,r)$ and $B_{n}(k)$ for some $k,r$,
    its $k$-factored $(n+1)$-simplex is also in
    $\mathrm{N}\mathbf{O}_{0}$,
  \item for every $(n+1)$-simplex in $\mathrm{N}\mathbf{O}_{0}$ whose image in
    $\mathrm{N}\Dop$ is in $A'_{n}(k,r)$ and $B'_{n}(k)$ for some $k,r$,
    its $k$-factored $(n+2)$-simplex is also in $\mathrm{N}\mathbf{O}_{0}$.
  \end{itemize}
  Then the inclusion $\mathrm{N}\mathbf{O}_{0} \hookrightarrow
  \mathrm{N}\mathbf{O}$ is a trivial cofibration of
  \gnsiopds{}.
\end{cor}
\begin{proof}
  Since $\mathrm{N}\mathbf{O}_{0}$ contains all narrow active
  simplices in $\mathrm{N}\mathbf{O}$, the simplicial subsets
  $\mathfrak{F}_{n}\mathbf{O}$ and $\mathfrak{F}_{n}^{+}\mathbf{O}$ of
  $\mathrm{N}\mathbf{O}$ coincide. The inclusion
  $\mathrm{N}\mathbf{O}_{0} \hookrightarrow \mathrm{N}\mathbf{O}$ is
  therefore the composite of the inclusions
  $\mathfrak{F}_{n-1}\mathbf{O} = \mathfrak{F}_{n-1}^{+}\mathbf{O}
  \hookrightarrow \mathfrak{F}_{n}\mathbf{O}$, which are all trivial
  cofibrations by Proposition~\ref{propn:Otrivcofib}.
\end{proof}

\begin{proof}[Proof of Proposition~\ref{propn:cellmodel1}]
  We apply Corollary~\ref{cor:Otrivcofib} to the inclusion
  $\simp_{/[i]}^{\amalg,\op} \hookrightarrow \Lbriop$. The required
  hypotheses hold since a simplex of $\Lbriop$ lies in
  $\simp_{/[i]}^{\amalg,\op}$ \IFF{} its source is of the form
  $(i_{0},\ldots,i_{n})$ with $i_{n} - i_{0} \leq 1$.
\end{proof}

\subsection{The Double $\infty$-Category of Algebras}\label{subsec:functcomp}
Our goal in this subsection is to prove that the \icats{}
$\fALG_{1}(\mathcal{C})_{n}$ fit together into a simplicial
\icat{}. We will do this by checking that composite
$\simp_{/[n]}^{\op}$-algebras map to composite
$\simp^{\op}_{/[m]}$-algebras under composition with the map
\[\phi_{*} \colon \simp^{\op}_{/[m]} \to \simp^{\op}_{/[n]}\] induced by
a map $\phi \colon [m] \to [n]$ in $\simp$.

\begin{defn}
  Suppose $\mathcal{C}$ is a monoidal \icat{}. Let
 $\ofALG_{1}(\mathcal{C}) \to
  \simp^{\op}$ denote a coCartesian fibration associated to the
  functor $\simp^{\op} \to \CatI$ that sends $[n]$ to
  $\Algns_{\simp^{\op}_{/[n]}}(\mathcal{C})$.  Write
  $\fALG_{1}(\mathcal{C})$ for the full subcategory of
  $\ofALG_{1}(\mathcal{C})$ spanned by the objects of
  $\fALG_{1}(\mathcal{C})_{n}$ for all $n$, i.e. the composite
  $\simp^{\op}_{/[n]}$-algebras for all $n$.
\end{defn}

We wish to prove that the restricted projection
$\fALG_{1}(\mathcal{C}) \to \simp^{\op}$ is a coCartesian fibration,
with the coCartesian morphisms inherited from
$\ofALG_{1}(\mathcal{C})$. The key step in the proof is showing that a certain
functor is cofinal; to state the
required result we first need the following technical generalization
of cellular maps:
\begin{defn}
  Suppose $\phi \colon [m] \to [n]$ is an injective morphism in
  $\simp$. We say a morphism $\alpha \colon [k] \to [n]$ is
  \emph{$\phi$-cellular} if
  \begin{enumerate}[(i)]
  \item for $\alpha(i) < \phi(0)$ we have $\alpha(i+1) \leq \alpha(i)+1$,
  \item for $\phi(j) \leq \alpha(i) < \phi(j+1)$ we have $\alpha(i+1)
    \leq \phi(j+1)$,
  \item for $\alpha(i) \geq \phi(m)$ we have $\alpha(i+1) \leq \alpha(i)+1$.
  \end{enumerate}
\end{defn}
\begin{remark}
  We recover the previous notion of \emph{cellular} maps to $[n]$ as the
  $\phi$-cellular maps with $\phi = \id_{[n]}$.
\end{remark}

\begin{defn}
  For $[n] \in \simp$ and $\phi \colon [m] \to [n]$ any injective
  morphism in $\simp$, we write $\LnP$ for the full subcategory of
  $\simp_{/[n]}$ spanned by the $\phi$-cellular maps to $[n]$.
\end{defn}

\begin{propn}\label{propn:cell1coinitial}\ 
  \begin{enumerate}[(i)]
  \item If $\phi \colon [m] \to [n]$ is an injective morphism in
    $\simp$, then for any $\gamma \colon [k] \to [m]$ the map \[
    \phi_{*} \colon (\bbLambda_{/[m]})^{\txt{act}}_{\gamma/} \to
    (\LnP)^{\txt{act}}_{\phi\gamma/} \] given by composition with
    $\phi$ is coinitial.
  \item If $\phi \colon [m] \to [n]$ is a surjective morphism in
    $\simp$, then for any $\gamma \colon [k] \to [m]$ the map \[
    \phi_{*} \colon (\bbLambda_{/[m]})^{\txt{act}}_{\gamma/} \to
    (\bbLambda_{/[n]})^{\txt{act}}_{\phi\gamma/} \] given by
    composition with $\phi$ is coinitial.
  \end{enumerate}
\end{propn}

\begin{proof}
  We first prove (i), i.e. we consider an injective map $\phi$. To show that $\phi_{*}$ is coinitial, recall that by
  \rcite{HTT}{Theorem 4.1.3.1} it suffices
  to prove that for each $X \in
  (\LIP)^{\txt{act}}_{\phi\gamma/}$, the category
  $((\bbLambda_{/[l]})^{\txt{act}}_{\gamma/})_{/X}$ is weakly
  contractible. The object $X$ is a diagram \csquare{ \lbrack k
    \rbrack }{ \lbrack p \rbrack }{ \lbrack l \rbrack }{ \lbrack n
    \rbrack }{\alpha}{\gamma}{\xi}{\phi} where $\xi$ is a
  $\phi$-cellular map and $\alpha$ is active, and an object $\bar{X}
  \in ((\bbLambda_{/[l]})^{\txt{act}}_{\gamma/})_{/X}$ is a
  diagram
  \[ %
\begin{tikzpicture} %
\matrix (m) [matrix of math nodes,row sep=3em,column sep=2.5em,text height=1.5ex,text depth=0.25ex] %
{  \lbrack k \rbrack & \lbrack q \rbrack & \lbrack p \rbrack \\
   \lbrack l \rbrack &   & \lbrack n \rbrack \\};
\path[->,font=\footnotesize] %
(m-1-1) edge node[below] {$\pi$} (m-1-2)
(m-1-2) edge node[below]{$\lambda$} (m-1-3)
(m-1-1) edge node[left] {$\gamma$} (m-2-1)
(m-1-2) edge node[below right] {$\theta$} (m-2-1)
(m-1-3) edge node[right] {$\xi$} (m-2-3)
(m-2-1) edge node[below] {$\phi$} (m-2-3)
(m-1-1) edge[bend left=30] node[above] {$\alpha$} (m-1-3);
\end{tikzpicture}%
\]%
where $\theta$ is a cellular map and $\pi$ and $\lambda$ are
active.

Since $\phi$ is injective, this category has a final object, given as
follows: Let $[q] = \xi^{-1}([l])$, let $\lambda$ be the inclusion
$[q] \to [p]$, and let $\theta$ be the induced projection $[q] \to
[l]$ --- since $\xi$ is $\phi$-cellular, $\theta$ hits everything in
$[l]$ and so is cellular. Moreover, $\alpha$ factors through a map
$\pi \colon [k] \to [q]$, since $\xi \alpha = \phi \gamma$ and so the
image of $\alpha$ in $[p]$ maps to the image of $\phi$ in $[n]$. But
then, since $\alpha$ is active, the maps $\pi$ and $\lambda$ must also
be active, so we have defined an object of
$((\bbLambda_{/[l]})^{\txt{act}}_{\gamma/})_{/X}$. Any other object of
the category has a unique map to this, i.e. this is a final
object. This implies that the category
$((\bbLambda_{/[l]})^{\txt{act}}_{\gamma/})_{/X}$ is weakly
contractible.

We now consider (ii), the surjective case. We can write $\phi$ as a composite of
elementary degeneracies, and so it suffices to consider the case where
$\phi$ is an elementary degeneracy $s_{t} \colon [l+1] \to [l]$. We
again wish to apply \rcite{HTT}{Theorem 4.1.3.1} and show that for
each $X \in (\bbLambda_{/[l]})^{\txt{act}}_{s_{t}\gamma/}$ the
category $((\bbLambda_{/[l+1]})^{\txt{act}}_{\gamma/})_{/X}$ is
weakly contractible. Let $X$ be as above, and let $\Lambda_{X}$ denote
the partially ordered set of pairs $(a, b)$ where
\begin{itemize}
\item  $a, b \in [p]$,
\item $\xi(a) = \xi(b) = t$,
\item $a \leq b$,
\item if $i \in [k]$ satisfies $\gamma(i) = t$, then $\alpha(i) \leq a$,
\item if $i \in [k]$ satisfies $\gamma(i) = t+1$, then $\alpha(i) \geq b$,
\end{itemize}
where $(a, b) \leq (a',b')$ if $a \leq a' \leq b' \leq b$. Define a
functor $G_{X} \colon \Lambda_{X} \to
((\bbLambda_{/[l+1]})^{\txt{act}}_{\gamma/})_{/X}$ by sending $(a, b)$
to the diagram
  \[ %
\begin{tikzpicture} %
\matrix (m) [matrix of math nodes,row sep=3em,column sep=2.5em,text height=1.5ex,text depth=0.25ex] %
{  \lbrack k \rbrack & \lbrack p+(1+a-b) \rbrack & \lbrack p \rbrack \\
   \lbrack l+1 \rbrack &   & \lbrack l \rbrack, \\};
\path[->,font=\footnotesize] %
(m-1-1) edge node[below] {$\pi_{(a,b)}$} (m-1-2)
(m-1-2) edge node[below]{$\lambda_{(a,b)}$} (m-1-3)
(m-1-1) edge node[left] {$\gamma$} (m-2-1)
(m-1-2) edge node[below right] {$\theta_{(a,b)}$} (m-2-1)
(m-1-3) edge node[right] {$\xi$} (m-2-3)
(m-2-1) edge node[below] {$s_{t}$} (m-2-3)
(m-1-1) edge[bend left=30] node[above] {$\alpha$} (m-1-3);
\end{tikzpicture}%
\]%
where:
\[\theta_{(a,b)}(i) =
\begin{cases}
\xi(i), & i \leq a \\
\xi(i)+1, & i > a,  
\end{cases}\]
\[ \lambda_{(a,b)}(i) = \begin{cases}
i, & i \leq a \\
i - (1+a-b), & i > a,
\end{cases}\]
\[ \pi_{(a,b)}(i) =
\begin{cases}
\alpha(i), & i \leq a, \\
\alpha(i)+(1+a-b) , & i > a.
\end{cases}
\]
Here $\theta_{(a,b)}$ is cellular, the maps $\lambda_{(a,b)}$ and
$\pi_{(a,b)}$ are active, and the diagram commutes. The maps from
$(a,b)$ to
$(a,b-1)$ and $(a+1,b)$ are sent by $G_{X}$ to the obvious
transformations of diagrams including the face maps $d_{b},d_{a}\colon [p+(1+a-b)]
\to [p+(2+a-b)]$, respectively.

Now observe that $G_{X}$ has a left adjoint $F_{X} \colon 
((\bbLambda^{\txt{act}}_{/[l]})_{\gamma/})_{/X} \to
\Lambda_{X}$. This sends a diagram as above to $(a, b)$ where $a$ is
maximal such that there exists $i \in [q]$ with $\theta(i) = t$ and
$\lambda(i) = a$, and $b$ is minimal such that there exists $i$ with
$\theta(i) = t+1$ and $\lambda(i) = b$. We have
$F_{X}G_{X} = \id$, and the unit map $\id \to G_{X}F_{X}$ is given by
the natural diagram containing the map $\bar{\lambda} \colon [q] \to
[p+(1+a-b)]$ defined by 
\[\bar{\lambda}(i) =
\begin{cases}
  \lambda(i) & i \leq a, \\
  \lambda(i) + (1+a-b) & i > a.
\end{cases}\]

Since adjunctions of \icats{} are in particular weak homotopy equivalences of
simplicial sets, it follows that $((\simp^{\txt{cell,act}}_{/[l]})_{\gamma/})_{/X}$ is
weakly contractible \IFF{} $\Lambda_{X}$ is. But $\Lambda_{X}$ has an
initial object, namely $(A,B)$ where $A$ is minimal such that $\xi(A)
= t$ and $A \geq \alpha(i)$ for any $i \in [k]$ such that $\gamma(i) =
t$, and $B$ is maximal such that $\xi(B) = t+1$ and $B \leq \alpha(i)$
for any $i \in [k]$ such that $\gamma(i) = t+1$. This implies that
$\Lambda_{X}$ is indeed weakly contractible, which completes the proof.
\end{proof}

\begin{cor}\label{cor:ALG1coCart}
  Suppose $\mathcal{C}$ is a monoidal \icat{} with good relative
  tensor products. Then the projection $\fALG_{1}(\mathcal{C}) \to
  \simp^{\op}$ is a coCartesian fibration.
\end{cor}
\begin{proof}
  Since $\ofALG_{1}(\mathcal{C}) \to \simp^{\op}$ is a coCartesian
  fibration, it suffices to show that if $X$ is an object of
  $\fALG_{1}(\mathcal{C})$ over $[n] \in \simp^{\op}$, and $X \to
  \bar{X}$ is a coCartesian morphism in $\ofALG_{1}(\mathcal{C})$ over
  $\phi \colon [m] \to [n]$ in $\simp$, then $\bar{X}$ is also in
  $\fALG_{1}(\mathcal{C})$. In other words, we must show that if $X$
  is a composite $\simp^{\op}_{/[n]}$-module, then $(\phi_{*})^{*}X$
  is a composite $\simp^{\op}_{/[m]}$-module for any map $\phi \colon
  [m] \to [n]$, i.e. the counit map
  $\tau_{m,!}\tau_{m}^{*}(\phi_{*})^{*}X \to (\phi_{*})^{*}X$ is an
  equivalence, where $\tau_{m}$ is the inclusion
  $\bbLambda_{/[m]}^{\op} \to \simp^{\op}_{/[m]}$. Using the
  definition of $\tau_{m,!}$ as an operadic left Kan extension and the
  criterion of Lemma~\ref{lem:lokerecognize}, it suffices
  to show that for each $\gamma \in \bbLambda_{/[m]}^{\op}$, the
  natural map
  \[ \colim_{\eta\colon \gamma \to \gamma' \in
    ((\bbLambda_{/[m]})^{\txt{act}}_{\gamma/})^{\op}} \Xi(\phi\gamma')
  \to \Xi(\phi\gamma) \]
  is an equivalence, where $\Xi \colon \simp_{/[n]}^{\op,\txt{act}}
  \to \mathcal{C} \simeq \mathcal{C}^{\otimes}_{[1]}$ denotes the
  coCartesian pushforward along the unique active maps to $[1]$ of the
  restriction of $X$ to $\simp_{/[n]}^{\op,\txt{act}}$.

  It suffices to consider separately the cases where $\phi$ is either
  surjective or injective. If $\phi$ is surjective, then the map 
  $\phi_{*} \colon (\bbLambda_{/[m]})^{\txt{act}}_{\gamma/} \to
  (\bbLambda_{/[n]})^{\txt{act}}_{\phi\gamma/}$ gives a factorization
  of this map as
  \[ \colim_{\eta\colon \gamma \to \gamma' \in
    ((\bbLambda_{/[m]})^{\txt{act}}_{\gamma/})^{\op}} \Xi(\phi\gamma')
  \to \colim_{\eta \colon \phi\gamma \to \gamma'' \in
    (\bbLambda_{/[n]})^{\txt{act}}_{\phi\gamma/})^{\op}} \Xi(\gamma'') \to
  \Xi(\phi\gamma).\]
  Here the first map is an equivalence by
  Proposition~\ref{propn:cell1coinitial}(ii) and the second map is an
  equivalence since $X$ is composite.

  Now suppose $\phi$ is injective. Then the functor
  $(\bbLambda_{/[m]})^{\txt{act}}_{\gamma/} \to (\LnP)^{\txt{act}}_{\phi\gamma/}$ gives a
  factorization of the map above as
  \[ \colim_{\eta\colon \gamma \to \gamma' \in
    ((\bbLambda_{/[m]})^{\txt{act}}_{\gamma/})^{\op}} X(\phi\gamma')
  \to \colim_{\eta \colon \phi\gamma \to \gamma'' \in
    (\LnP)^{\txt{act}}_{\phi\gamma/})^{\op}} X(\gamma'') \to
  X(\phi\gamma).\] Here the first map is an equivalence by
  Proposition~\ref{propn:cell1coinitial}(i). Moreover, since $X$ is a
  composite $\simp^{\op}_{/[n]}$-algebra and the inclusions
  \[((\Lbrn)^{\txt{act}}_{\phi\gamma/})^{\op} \to
  (\LnP)^{\txt{act}}_{\phi\gamma/})^{\op} \to
  ((\simp_{/[n]})^{\txt{act}}_{\phi\gamma/})^{\op}\] are fully
  faithful, the map \[\colim_{\eta \colon \phi\gamma \to \gamma'' \in
    ((\LnP)^{\txt{act}}_{\phi\gamma/})^{\op}} \Xi(\gamma'') \to
  \Xi(\phi\gamma)\] is also an equivalence, since $\Xi$ is a left Kan
  extension of its restriction to $(\Lbrn)^{\txt{act},\op}$ by
  Lemma~\ref{lem:lokerecognize}.
\end{proof}

Combining Corollary~\ref{cor:ALG1coCart} with
Corollary~\ref{cor:ALG1seg}, we have proved:
\begin{thm}\label{thm:ALG1double}
  Suppose $\mathcal{C}$ is a monoidal \icat{} with good relative
  tensor products. Then the projection $\fALG_{1}(\mathcal{C}) \to
  \simp^{\op}$ is a double \icat{}.\qed
\end{thm}

\begin{defn}
  Let $\mathcal{C}$ be a monoidal \icat{} with good
  relative tensor products. Then we define $\fAlg_{1}(\mathcal{C})$ to be the
  $(\infty,2)$-category underlying the double \icat{}
  $\fALG_{1}(\mathcal{C})$, i.e. the completion
  $L_{2}U_{\Seg}i\fALG_{1}(\mathcal{C})$ of the underlying 2-fold Segal
  space of $\fALG_{1}(\mathcal{C})$ (cf. Remark~\ref{rmk:complsegnmon}).
\end{defn}

\subsection{The Bimodule Fibration}\label{subsec:bimodfib}
Let $\mathcal{C}$ be a monoidal \icat{}. We will write
$\Bimod(\mathcal{C})$ for the \icat{} $\Algns_{\simp^{\op}_{/[1]}}(\mathcal{C})$
and $\Ass(\mathcal{C})$ for the \icat{}
$\Algns_{\simp^{\op}}(\mathcal{C})$. There is a projection 
\[\pi \colon \Bimod(\mathcal{C}) \to \Ass(\mathcal{C})^{\times 2}\] that sends an
$A$-$B$-bimodule $M$ in $\mathcal{C}$ to the pair $(A, B)$. Our goal in
this subsection is to analyze this functor, as well as the
projection
\[ U \colon \Bimod(\mathcal{C}) \to \Ass(\mathcal{C}) \times
\mathcal{C} \times \Ass(\mathcal{C}) \] that sends an $A$-$B$-bimodule
$M$ in $\mathcal{C}$ to $(A, X, B)$, where $X$ is the object of
$\mathcal{C}$ underlying $M$; we will make use of this work below in
\S\ref{subsec:Algnmaps}. Our first task is to prove that the map $U$
is given by restriction along an extendable map of \gnsiopds{}, from
which it will follow that $U$ has a left adjoint.

\begin{defn}
  We can identify objects of $\simp_{/[1]}$ with lists
  $(i_{0},\ldots,i_{n})$ where $0 \leq i_{j} \leq i_{j+1} \leq 1$, and
  for every $\phi \colon [m] \to [n]$ in $\simp$ there is a unique
  morphism $(i_{\phi(0)},\ldots, i_{\phi(m)}) \to (i_{0}, \ldots,
  i_{n})$ over $\phi$ in $\simp^{\op}$. Let $\mathcal{U}$ denote the
  subcategory of $\simp_{/[1]}$ containing all the objects and the
  morphisms $(i_{\phi(0)},\ldots, i_{\phi(m)}) \to (i_{0}, \ldots,
  i_{n})$ as before where, if $t$ is the largest index $j$ such that
  $i_{j} = 0$ and $s$ is the largest index $j$ such that $i_{\phi(j)}
  = 0$, then either $t = -1$, $t = m$, or the image of $\phi$ contains
  both $s$ and $s+1$.
\end{defn}
\begin{remark}
  It is easy to see that the projection $\mathcal{U}^{\op} \to
  \simp^{\op}$ is a \gnsiopd{}. A $\mathcal{U}$-algebra $A$ in
  $\mathcal{C}$ contains the information of two associative algebras
  in $\mathcal{C}$, since we have retained the full subcategories of
  $\Dop_{/[1]}$ on the objects of the form $(0,\ldots,0)$ and
  $(1,\ldots,1)$. The algebra $A$ also determines an object $A(0,1)
  \in \mathcal{C}$, but we have omitted the maps in $\Dop_{/[1]}$ that
  describe the action of the two algebras on this objects. Indeed, as
  we will see in Proposition~\ref{propn:Utrivcofib} below, a
  $\mathcal{U}$-algebra consists precisely of this information --- two
  associative algebras, and an additional object.
\end{remark}

\begin{lemma}\label{lem:bimodgr}
  Let $i$ denote the inclusion $\mathcal{U} \hookrightarrow
  \Dop_{/[1]}$. 
  \begin{enumerate}[(i)]
  \item $i$ is an extendable morphism of \gnsiopds{}.
  \item For every monoidal \icat{} $\mathcal{C}$, the functor $i^{*}
    \colon \Bimod(\mathcal{C})\to \Algns_{\mathcal{U}}(\mathcal{C})$
    has a left adjoint $i_{!}$.
  \end{enumerate}
\end{lemma}
\begin{proof}
  The \icat{} $\mathcal{U}^{\op,\txt{act}}_{/X}$ has a final object
  for every $X \in \Dop_{/[1]}$ (e.g. $(0,0,1,1)$ is final in
  $\mathcal{U}^{\op,\txt{act}}_{/(0,1)}$). This implies that $i$ is
  extendable and, by Proposition~\ref{propn:lokeexist}, that operadic
  left Kan extensions along $i$ always exist. The left adjoint $i_{!}$
  therefore always exists by Corollary~\ref{cor:opdkanext}.
\end{proof}

Next, we want to prove that the adjunction $i_{!} \dashv i^{*}$ is
monadic. For this, we need a criterion for the existence of colimits
for sifted diagrams of algebras:
\begin{propn}\label{propn:siftedcolimalg}
  Suppose $\mathcal{M}$ is a \gnsiopd{} and $\mathcal{C}$ is a
  monoidal \icat{}, and let $K$ be a sifted simplicial set. Then a
  diagram $p \colon K \to \Algns_{\mathcal{M}}(\mathcal{C})$ has a
  colimit if for every $x \in \mathcal{M}_{[1]}$ the diagram
  $\txt{ev}_{x} \circ p \colon K \to \mathcal{C}$ has a monoidal
  colimit in $\mathcal{C}$ (i.e. it has a colimit that is preserved by
  tensoring with objects of $\mathcal{C}$). Moreover, if this holds
  then this colimit is preserved by the forgetful functors
  $\txt{ev}_{x}$.
\end{propn}
\begin{proof}
  This follows from the same argument as in the proof of 
 \rcite{enr}{Theorem A.5.3}.\end{proof}
\begin{cor}\label{cor:bimodcolim}
  Let $\mathcal{C}$ be a monoidal \icat{} and suppose $K$ is a sifted
  simplicial set. A diagram $p \colon K \to \Bimod(\mathcal{C})$ has a
  colimit if the functors $\txt{ev}_{(i,j)}\circ p \colon K \to
  \mathcal{C}$ for $(i,j) = (0,0), (0,1), (1,1)$ all have monoidal colimits in
  $\mathcal{C}$. Moreover, such colimits are preserved by $i^{*}
  \colon \Bimod(\mathcal{C}) \to \Algns_{\mathcal{U}}(\mathcal{C})$.
\end{cor}
\begin{proof}
  This follows by applying Proposition~\ref{propn:siftedcolimalg} to
  $\Bimod(\mathcal{C})$ and $\Algns_{\mathcal{U}}(\mathcal{C})$.
\end{proof}

\begin{cor}
  For any monoidal \icat{} $\mathcal{C}$, the adjunction 
  \[ i_{!} : \Algns_{\mathcal{U}}(\mathcal{C}) \rightleftarrows
  \Bimod(\mathcal{C}) : i^{*}\]
  is monadic.
\end{cor}
\begin{proof}  
  Suppose given a diagram $F \colon \Dop \to \Bimod(\mathcal{C})$ that
  is $i^{*}$-split in the sense of \rcite{HA}{Definition 4.3.7.2},
  i.e. the diagram $i^{*}F$ extends to a diagram $F' \colon
  \Dop_{-\infty} \to \Algns_{\mathcal{U}}(\mathcal{C})$. A split
  simplicial object is always a colimit diagram by \rcite{HTT}{Lemma
    6.1.3.16}, so $i^{*}F$ has a colimit in
  $\Algns_{\mathcal{U}}(\mathcal{C})$. Moreover, for the same reason
  the underlying diagrams in $\mathcal{C}$ are monoidal colimit
  diagrams, since tensoring with a fixed object of $\mathcal{C}$ again
  gives a split simplicial diagram. It then follows from
  Corollary~\ref{cor:bimodcolim} that $F$ has a colimit in
  $\Bimod(\mathcal{C})$ and this colimit is preserved by $i^{*}$.
  The forgetful functors from $\Bimod(\mathcal{C})$ and
  $\Algns_{\mathcal{U}}(\mathcal{C})$  to $\Fun(\{(0,0), (0,1), (1,1)\},
   \mathcal{C})$ are conservative
  by \rcite{enr}{Lemma A.5.5}; since the diagram
  \opctriangle{\Bimod(\mathcal{C})}{\Algns_{\mathcal{U}}(\mathcal{C})}{\Fun(\{(0,0),
    (0,1), (1,1)\}, \mathcal{C})}{i^{*}}{}{}
  commutes, it follows that $i^{*}$ also conservative. The Barr-Beck Theorem for \icats{},
  i.e. \rcite{HA}{Theorem 4.7.4.5}, now implies that the adjunction
  $i_{!} \dashv i^{*}$ is monadic.
\end{proof}

We now wish to identify the functor $i^{*} \colon
\Bimod(\mathcal{C})\to \Algns_{\mathcal{U}}(\mathcal{C})$ with the
projection $U$. To do this, we define $\mathcal{X}$ to be the full
subcategory of $\simp_{/[1]}$ spanned by the objects $(0)$, $(1)$ and
$(0,1)$. The projection $\mathcal{X}^{\op} \to \simp^{\op}$ is a
\gnsiopd{}, and the functor $\Alg^{1}_{\mathcal{X}^{\op}}(\mathcal{C})
\to \mathcal{C}$ given by evaluation at $(0,1)$ is an equivalence for
any monoidal \icat{} $\mathcal{C}$. We can thus identify the
projection $U$ with the map induced by composition with the inclusion
$\simp^{\op} \amalg_{\{(0)\}} \mathcal{X}^{\op} \amalg_{\{(1)\}}
\simp^{\op} \hookrightarrow \Dop_{/[1]}$.

\begin{propn}\label{propn:Utrivcofib}
  The inclusion $\simp^{\op} \amalg_{\{(0)\}} \mathcal{X}
  \amalg_{\{(1)\}} \simp^{\op} \to \mathcal{U}$ is a trivial cofibration
  in $(\sSet^{+})_{\mathfrak{O}_{\txt{ns}}^{\txt{gen}}}$.
\end{propn}
\begin{proof}
  We apply Corollary~\ref{cor:Otrivcofib} --- it is clear from the
  definition of $\mathcal{U}$ as a subset of $\Dop_{/[1]}$ that the
  required hypotheses hold.
\end{proof}

\begin{cor}\label{cor:BimodLeftAdj}
  Let $\mathcal{C}$ be a monoidal \icat{}. The projection $U \colon
  \Bimod(\mathcal{C}) \to \Ass(\mathcal{C}) \times \mathcal{C} \times
  \Ass(\mathcal{C})$ has a left adjoint $F$ such that $UF(A, M, B)
  \simeq (A, A \otimes M \otimes B, B)$. Moreover, the adjunction $F
  \dashv U$ is monadic.\qed
\end{cor}

\begin{cor}\label{cor:BimodABLeftAdj}
  For any $A, B \in \Ass(\mathcal{C})$, let
  $\Bimod_{A,B}(\mathcal{C})$ denote the fibre of \[\pi \colon
  \Bimod(\mathcal{C}) \to \Ass(\mathcal{C})^{\times 2}\] at
  $(A,B)$. Then:
  \begin{enumerate}[(i)]
  \item The pullback $U_{A,B} \colon \Bimod_{A,B}(\mathcal{C}) \to
    \mathcal{C}$ of $U$ has a left adjoint $F_{A,B}$ such that the
    unit map $M \to U_{A,B}F_{A,B}(M)$ is the map $M \to A \otimes M
    \otimes B$ given by tensoring with the unit maps of $A$ and $B$.
  \item If $K$ is a sifted simplicial set, then a diagram $p \colon K
    \to \Bimod_{A,B}(\mathcal{C})$ has a colimit if the underlying
    diagram $U_{A,B}\circ p \colon K \to \mathcal{C}$ has a monoidal
    colimit. Moreover, the forgetful functor $U_{A,B}$ detects such
    colimits.
  \item The adjunction $F_{A,B} \dashv U_{A,B}$ is monadic.
  \end{enumerate}
\end{cor}
\begin{proof}
  The existence of the adjunction $F_{A,B} \dashv U_{A,B}$
  follows from Corollary~\ref{cor:BimodLeftAdj} and
  \rcite{HA}{Proposition 7.3.2.5}.

  Suppose $p \colon K \to \mathcal{C}$ is as in (ii). Since $K$ is
  weakly contractible, a constant diagram in $\mathcal{C}$ indexed by
  $K^{\triangleright}$ is a colimit diagram, and for the same reason it is also
  a monoidal colimit diagram. The composite diagram $p \colon K \to
  \Bimod(\mathcal{C})$ therefore has a colimit $K^{\triangleright} \to
  \Bimod(\mathcal{C})$ by
  Corollary~\ref{cor:bimodcolim}, and this factors through
  $\Bimod_{A,B}(\mathcal{C})$. Since $\Bimod_{A,B}(\mathcal{C})$ is a
  pullback, and the projections of the diagram to $\mathcal{C}$ and
  $\Ass(\mathcal{C}) \times \mathcal{C} \times \Ass(\mathcal{C})$ are
  colimits, it follows that this diagram is also a colimit
  diagram in $A$-$B$-bimodules. This proves (ii).

  Since a $U_{A,B}$-split diagram in $\Bimod_{A,B}(\mathcal{C})$ gives
  a $U$-split diagram in $\mathcal{C}$, it now follows from
  Corollary~\ref{cor:BimodLeftAdj} that $\Bimod_{A,B}(\mathcal{C})$
  has colimits of $U_{A,B}$-split simplicial diagrams and these are
  preserved by $U_{A,B}$. Since the inclusions $\{A\} \times
  \mathcal{C} \times \{B\} \hookrightarrow \Ass(\mathcal{C}) \times
  \mathcal{C} \times \Ass(\mathcal{C})$ and $\Bimod_{A,B}(\mathcal{C})
  \hookrightarrow \Bimod(\mathcal{C})$ also detect
  equivalences, it follows that the adjunction $F_{A,B} \dashv
  U_{A,B}$ is monadic by \rcite{HA}{Theorem 4.7.4.5}.
\end{proof}

\begin{cor}\label{cor:BimodIIeqC}
  Let $\mathcal{C}$ be a monoidal \icat{}, and let $I$ be the unit of
  $\mathcal{C}$ regarded as an associative algebra. Then the
  projection $U_{I,I} \colon \Bimod_{I,I}(\mathcal{C}) \to
  \mathcal{C}$ is an equivalence.
\end{cor}
\begin{proof}
  By Corollary~\ref{cor:BimodABLeftAdj} the functor $U_{I,I}$ has a
  left adjoint $F_{I,I}$ and the adjunction $F_{I,I} \dashv U_{I,I}$
  is monadic. Moreover, the unit map $M \to U_{I,I}F_{I,I}M$ is the
  canonical equivalence $M \isoto I \otimes M \otimes I$. It follows
  from \rcite{HA}{Corollary 4.7.4.16} applied to the diagram
  \opctriangle{\Bimod_{I,I}(\mathcal{C})}{\mathcal{C}}{\mathcal{C}.}{U_{I,I}}{U_{I,I}}{\id}
  that $U_{I,I}$ is an equivalence of \icats{}.
\end{proof}

Our next goal is to show that the projection $\pi \colon \Bimod(\mathcal{C}) \to
\Ass(\mathcal{C})^{\times 2}$ is a coCartesian fibration if
$\mathcal{C}$ has good relative tensor products. This
requires some technical preliminary observations:

\begin{propn}\label{propn:rightadjcocartmor}
  Suppose $p \colon \mathcal{E} \to \mathcal{C}$ is an inner
  fibration, and that $p$ has a left adjoint $F \colon \mathcal{C} \to
  \mathcal{E}$. Then a morphism $\phi \colon e \to e'$ in
  $\mathcal{E}$ is $p$-coCartesian \IFF{} the commutative square
  \csquare{Fp(e)}{Fp(e')}{e}{e'}{Fp(\phi)}{c_{e}}{c_{e'}}{\phi} is a
  pushout square.
\end{propn}
\begin{proof}
  For any $x \in \mathcal{E}$ we have a commutative diagram
  \[
  \ltikzcd{\Map_{\mathcal{E}}(e', x)  \arrow{r} \arrow{d}\pgfmatrixnextcell 
    \Map_{\mathcal{E}}(e, x) \arrow{d} \\
    \Map_{\mathcal{E}}(Fpe', x) \arrow{r} \arrow{d}{\sim}  \pgfmatrixnextcell
    \Map_{\mathcal{E}}(Fpe, x) \arrow{d}{\sim} \\
    \Map_{\mathcal{C}}(pe', px) \arrow{r} \pgfmatrixnextcell \Map_{\mathcal{C}}(pe, px),
    }
  \]
  where the vertical composites are equivalent to the maps coming from
  the functor $p$ by the adjunction identitites. The map $\phi$ is
  thus $p$-coCartesian \IFF{} the composite square is Cartesian for
  all $x$, and the commutative square
  \csquare{Fp(e)}{Fp(e')}{e}{e'}{Fp(\phi)}{c_{e}}{c_{e'}}{\phi} is a
  pushout \IFF{} the top square is Cartesian for all $x$. But since
  the lower vertical maps are equivalences the bottom square is always
  Cartesian, hence the top square is Cartesian \IFF{} the composite
  square is.
\end{proof}

\begin{cor}\label{cor:rightadjcocart}
  Suppose $p \colon \mathcal{E} \to \mathcal{C}$ is a categorical
  fibration between \icats{}, and that $p$ has a left adjoint $F
  \colon \mathcal{C} \to \mathcal{E}$. Then the following are
  equivalent:
  \begin{enumerate}[(1)]
  \item $p$ is a coCartesian fibration.
  \item For every $e \in \mathcal{E}$ and every morphism $\phi \colon
    p(e) \to x$ in $\mathcal{C}$, there is a pushout square
    \csquare{Fp(e)}{F(x)}{e}{\bar{x}}{F(\phi)}{c_{e}}{v}{\bar{\phi}}
    in $\mathcal{E}$, where $c$ is the counit for the adjunction, such
    that the composite 
    \[ x \xto{u_{x}} pF(x) \xto{p(v)} p(\bar{x})\]
    is an equivalence, where $u$ is the unit of the adjunction.
  \end{enumerate}
\end{cor}
\begin{proof}
  Suppose (2) holds. Given $e \in \mathcal{E}$ and $\phi \colon p(e)
  \to x$, we must show that there exists a $p$-coCartesian morphism $e
  \to \phi_{!}e$ over $\phi$. By assumption, there exists a pushout
  square
  \csquare{Fp(e)}{F(x)}{e}{\bar{x}}{F(\phi)}{c_{e}}{v}{\bar{\phi}} in
  $\mathcal{E}$ such that the composite
  \[ x \xto{u_{x}} pF(x) \xto{p(v)} p(\bar{x})\] is an
  equivalence. The adjunction identities imply that the map $v$
  factors as 
  \[ Fx \xto{F(u_{x} \circ p(v))} Fp\bar{x} \xto{c_{\bar{x}}}
  \bar{x} \]
  where the first map is an equivalence, and that the composite 
  \[ Fp(e) \xto{F\phi} F(x) \xto{F(u_{x} \circ p(v))} Fp(\bar{x}) \]
  is $Fp(\bar{\phi})$. Thus we have a pushout square
  \csquare{Fp(e)}{Fp(\bar{x})}{e}{\bar{x},}{Fp(\bar{\phi})}{c_{e}}{c_{\bar{x}}}{\bar{\phi}}
  which implies that $\bar{\phi}$ is $p$-coCartesian by
  Proposition~\ref{propn:rightadjcocartmor}. Since $p$ is a
  categorical fibration, by \rcite{HTT}{Corollary
    2.4.6.5} there exists an equivalence $\bar{x} \to \bar{x}'$ lying
  over the equivalence $(u_{x} \circ p(v))^{-1}$ in $\mathcal{C}$,
  and the composite $e \to \bar{x}'$ is a $p$-coCartesian morphism
  over $\phi$, which proves (1).
  
  Conversely, if (1) holds, then for any $e \in \mathcal{E}$ and $\phi
  \colon p(e) \to x$ in $\mathcal{C}$ there exists a $p$-coCartesian
  morphism $\bar{\phi} \colon e \to \bar{x}$ in $\mathcal{E}$ over
  $\phi$. By Proposition~\ref{propn:rightadjcocartmor} this means we
  have a pushout square
  \csquare{Fp(e)}{Fp(\bar{x})}{e}{\bar{x}.}{Fp(\bar{\phi})}{c_{e}}{c_{\bar{x}}}{\bar{\phi}}
  But as $\bar{\phi}$ lies over $\phi$, this gives (2).  
\end{proof}

\begin{propn}\label{propn:BimodFibcoCart}
  Suppose $\mathcal{C}$ is a monoidal \icat{} with good relative
  tensor products. Then the restriction
  $\pi \colon \Bimod(\mathcal{C}) \to \Ass(\mathcal{C})^{\times 2}$ is
  a coCartesian fibration. Moreover, if $M$ is an $A$-$B$-bimodule and
  $f \colon A \to A'$ and $g \colon B \to B'$ are morphisms of
  algebras in $\mathcal{C}$, then the coCartesian pushforward
  $(f,g)_{!}M$ is the tensor product $A' \otimes_{A} M \otimes_{B}
  B'$.
\end{propn}
\begin{proof}
  Let us first assume that $\mathcal{C}$ has an initial object
  $\emptyset$ and the monoidal structure is compatible with this
  (i.e. $c \otimes \emptyset \simeq \emptyset \otimes c \simeq \emptyset$ for all $c \in
  \mathcal{C}$). Then the projection $\Ass(\mathcal{C}) \times
  \mathcal{C} \times \Ass(\mathcal{C}) \to \Ass(\mathcal{C})^{\times
    2}$ has a left adjoint, which sends $(A, B)$ to $(A, \emptyset,
  B)$. By Corollary~\ref{cor:BimodLeftAdj} it follows that $\pi$ has a
  left adjoint $F'$, which sends $(A,B)$ to $F(A, \emptyset, B)$.

  Moreover, for any $M \in \Bimod(\mathcal{C})$ and any morphism
  $(f,g) \colon (A, B) \simeq \pi(M) \to (A',B')$, the pushout
  \nolabelcsquare{(UF)^{n}(A, \emptyset, B)}{(UF)^{n}(A', \emptyset,
    B')}{(UF)^{n}U(M)}{X_{n}} exists in $\Ass(\mathcal{C}) \times
  \mathcal{C} \times \Ass(\mathcal{C})$: since $\mathcal{C}$ is
  compatible with initial objects, the top horizontal morphism can be
  identified with $(A, \emptyset, B) \to (A', \emptyset, B')$ and the
  left vertical morphism with $(A, \emptyset, B) \to (A, A^{\otimes n}
  \otimes M \otimes B^{\otimes n}, B)$, so that $X_{n}$ is simply
  $(A', A^{\otimes n} \otimes M \otimes B^{\otimes n}, B')$. We then
  get a simplicial object $F(X_{\bullet})$ in
  $\Bimod(\mathcal{C})$. Evaluated at $(0,0)$ and $(1,1)$ this is
  constant at $A'$ and $B'$, respectively, and at $(0,1)$ we get $A'
  \otimes A^{\otimes \bullet} \otimes M \otimes B^{\otimes \bullet}
  \otimes B'$. Since $\mathcal{C}$ has good relative tensor products,
  the colimit of this simplicial diagram exists, is monoidal, and can
  be identified with the relative tensor product $A' \otimes_{A} M
  \otimes_{B} B'$. It follows from Corollary~\ref{cor:bimodcolim} that
  the diagram $F(X_{\bullet})$ has a colimit in
  $\Bimod(\mathcal{C})$. Moreover, since $F$ is a left adjoint and
  colimits commute we can identify this colimit as
  \[
  \begin{split}
|F(X_{\bullet})| & \simeq |F(UF)^{\bullet}U(M)
  \amalg_{F(UF)^{\bullet}(A,\emptyset,B)}
  F(UF)^{\bullet}(A',\emptyset,B')| \\ & \simeq |F(UF)^{\bullet}U(M)|
  \amalg_{|F(UF)^{\bullet}(A,\emptyset,B)|}
  |F(UF)^{\bullet}(A',\emptyset,B')| \\ & \simeq M
  \amalg_{F(A,\emptyset,B)} F((A',\emptyset,B').    
  \end{split}
\] Thus the pushout
  $M \amalg_{F'(A,B)} F'(A',B')$ exists in $\Bimod(\mathcal{C})$. It
  then follows from Corollary~\ref{cor:rightadjcocart} that $\pi$ is a
  coCartesian fibration, and that the object of $\mathcal{C}$
  underlying $(f,g)_{!}M$ is $A' \otimes_{A} M \otimes_{B} B'$.

  Now consider a general monoidal \icat{} $\mathcal{C}$. By
  \rcite{HA}{Proposition 4.8.1.10} (or by a direct construction) the
  \icat{} $\mathcal{C}^{\triangleleft}$ has a monoidal structure that
  is compatible with the initial object $-\infty$ and such that the
  inclusion $\mathcal{C} \hookrightarrow \mathcal{C}^{\triangleleft}$
  is monoidal. Moreover, this inclusion  preserves geometric
  realizations (and in general colimits other than the initial
  object), and thus $\mathcal{C}^{\triangleleft}$
  also has good relative tensor products. By our previous argument we
  then have a coCartesian fibration
  $\Bimod(\mathcal{C}^{\triangleleft}) \to
  \Ass(\mathcal{C}^{\triangleleft})^{\times 2}$. The initial object in
  $\mathcal{C}^{\triangleleft}$ does not admit an associative algebra
  structure (since it has no map from the unit), so the inclusion
  $\Ass(\mathcal{C}) \to \Ass(\mathcal{C}^{\triangleleft})$ is an
  equivalence. We thus wish to show that the restriction of the
  projection $\Bimod(\mathcal{C}^{\triangleleft}) \to
  \Ass(\mathcal{C})^{\times 2}$ to $\Bimod(\mathcal{C})$ is still a
  coCartesian fibration. For this it suffices to show that if $M$ is
  an $A$-$B$-bimodule in $\mathcal{C}$ and $f \colon A \to A'$ and $g
  \colon B \to B'$ are maps of associative algebras, then $(f,g)_{!}M$
  (computed in $\Bimod(\mathcal{C}^{\triangleleft})$ is also in
  $\Bimod(\mathcal{C})$. But this is true since the underlying object
  of $(f,g)_{!}M$ is given by a relative tensor product that cannot be
  the initial object $-\infty$.
\end{proof}

\section{$\mathbb{E}_{n}$-Algebras and Iterated
  Bimodules}\label{sec:enalg}

In this section we extend the results of \S\ref{sec:algbimod} to the
case $n > 1$: if $\mathcal{C}$ is a nice $\En$-monoidal \icat{} we
will construct an $(n+1)$-fold \icat{} $\fALG_{n}(\mathcal{C})$ of
$\En$-algebras; we can then define the $(\infty,n+1)$-category
$\fAlg_{n}(\mathcal{C})$ of $\mathbb{E}_{n}$-algebras in $\mathcal{C}$
as the completion of the underlying $(n+1)$-fold Segal space of
$\fALG_{n}(\mathcal{C})$. 

In order to iterate our results in the case $n = 1$ it is convenient
to work with a theory of \emph{\iopds{} over $\Dnop$} (or
\emph{\Dniopds{}}); we will introduce these objects in
\S\ref{subsec:dniopdintr} (with the more technical results we need
delegated to the appendix). Then in \S\ref{subsec:Enbimod} we observe
that the definitions of \S\ref{subsec:bimodopd} can be iterated and
use this to define the \icats{} $\fALG_{n}(\mathcal{C})_{I}$ for $I
\in \Dnop$, and in \S\ref{subsec:np1foldcat} we prove that these
\icats{} satisfy the Segal condition and give a functor $\Dnop \to
\CatI$. In \S\ref{subsec:ALGnmon} we then show that
$\fALG_{n}(\mathcal{C})$ is a lax monoidal functor in $\mathcal{C}$
and conclude from this that if $\mathcal{C}$ is an
$\mathbb{E}_{n+m}$-monoidal \icat{} then $\fALG_{n}(\mathcal{C})$
inherits an $\mathbb{E}_{m}$-monoidal structure. Finally, in
\S\ref{subsec:Algnmaps} we identify the $(\infty,n)$-category of maps
from $A$ to $B$ in $\fAlg_{n}(\mathcal{C})$ with
$\fAlg_{n-1}(\Bimod_{A,B}(\mathcal{C}))$.

\subsection{$\infty$-Operads over $\Dnop$}\label{subsec:dniopdintr}
In this subsection we will introduce the notion of \emph{\iopds{} over
  $\Dnop$} or \emph{\dniopds{}}, which is the setting in which we will
iterate the constructions of \S\ref{sec:algbimod}.

In \S\ref{subsec:carten} we introduced \emph{$\Dn$-monoids} in an
\icat{} $\mathcal{C}$ with finite products, by iterating the
definition of an associative monoid. Applying this to the \icat{}
$\CatI$ of \icats{}, we get a notion of \emph{$\Dn$-monoidal
  \icat{}}. Using the straightening equivalence, we can reinterpret
these as certain coCartesian fibrations over $\Dnop$:
\begin{defn}
  A \emph{$\Dn$-monoidal \icat{}} is a coCartesian fibration
  $\mathcal{C}^{\otimes} \to \Dnop$ such that for any object $I \in
  \Dnop$, the functor \[\mathcal{C}^{\otimes}_{I} \to
  (\mathcal{C}^{\otimes}_{C_{n}})^{\times |I|},\] induced by the
  coCartesian morphisms over the maps in $|I|$, is an equivalence.
\end{defn}

\begin{remark}
  $\Dn$-monoidal \icats{} can be interpreted as \icats{} equipped with
  $n$ compatible associative monoid structures, i.e. as
  \emph{$n$-tuply monoidal \icats{}}. We will see below in
  Corollary~\ref{cor:DnmonisEnmon} that they are also equivalent to
  \emph{$\En$-monoidal} \icats{} as defined in \cite{HA}, i.e. to
  algebras for the $\En$-$\infty$-operad in $\CatI$.
\end{remark}

In \cite{HA} Lurie defines symmetric \iopds{} by weakening the
definition of a symmetric monoidal \icat{} as a coCartesian fibration
over $\bbGamma^{\op}$, and above in Definition~\ref{defn:nsiopd} we
defined non-symmetric \iopds{} by analogously weakening the definition
of a monoidal \icat{} as a coCartesian fibration over $\Dop$. Applying
the same idea to $\Dn$-monoidal \icats{} gives a definition of
\emph{$\Dn$-\iopds{}}:
\begin{defn}\label{defn:Dniopd}
  A \defterm{\Dniopd} is a functor of \icats{} $\pi \colon \mathcal{O}
  \to \Dnop$ such that:
  \begin{enumerate}[(i)]
  \item For each inert map  $\phi \colon I \to J$ in $\Dnop$ and
    every $X \in \mathcal{O}$ such that $\pi(X) = I$, there exists
    a $\pi$-coCartesian morphism $X \to \phi_{!}X$ over $\phi$.
  \item For every $I$ in $\Dnop$, the functor
    \[ \mathcal{O}_{I} \to \mathcal{O}_{C_{n}}^{\times |I|} \]
    induced by the coCartesian morphisms over the inert maps $C_{n}
    \to I$ in $\Dnop$ is an equivalence.
  \item For every morphism $\phi \colon I \to J$ in $\Dnop$, $X \in
    \mathcal{O}_{I}$ and $Y \in \mathcal{O}_{J}$, composition with the
    coCartesian morphisms $Y \to Y_{i}$ over the inert morphisms $i
    \colon I \to C_{n}$ gives an equivalence
    \[ \Map_{\mathcal{O}}^{\phi}(X, Y) \isoto \prod_{i}
    \Map_{\mathcal{O}}^{i \circ \phi}(X, Y_{i}),\]
    where $\Map_{\mathcal{O}}^{\phi}(X, Y)$ denotes the subspace of
    $\Map_{\mathcal{O}}(X, Y)$ of morphisms that map to $\phi$ in
    $\Dnop$. (Equivalently, $Y$ is a $\pi$-limit of the $Y_{i}$'s.)
  \end{enumerate}
\end{defn}

\begin{remark}
  We will see in \S\ref{subsec:symDn} that there is an adjunction
  between \dniopds{} and symmetric \iopds{} over $\En$. In the case $n
  = 1$ this adjunction is an equivalence by \rcite{HA}{Proposition
    4.7.1.1}. We expect that this is true also for $n > 1$. Thus,
  \dniopds{} should be thought of as a more combinatorial or explicit
  model for symmetric \iopds{} over $\En$, where we do not need to
  deal with configuration spaces of points in $\mathbb{R}^{n}$.
\end{remark}

\begin{remark}
  $\Dn$-\iopds{} are a special case of Barwick's notion of
  \emph{\iopds{} over an operator category} as defined in
  \cite{BarwickOpCat}. Specifically, they are \iopds{} over the
  Cartesian product $\mathbb{O}^{\times n}$ where $\mathbb{O}$ is the
  operator category of finite ordered sets.
\end{remark}

\begin{remark}
  A $\Dn$-monoidal \icat{} as we defined it above is the same thing as
  a \Dniopd{} that is also a coCartesian fibration.
\end{remark}

To extend the definitions of iterated bimodules from
\S\ref{sec:cartenalg} to the non-Cartesian setting, we will need to
consider a more general notion than that of \dniopds{}. To introduce
this, recall that by iterating the definition of category object in
$\CatI$ we can define \emph{$\Dn$-uple \icats{}} (which model
$(n+1)$-uple \icats{}) as certain functors from $\Dnop$ to
$\CatI$. Rephrasing this in terms of coCartesian fibrations, we get
the following definition:
\begin{defn}
  A \emph{$\Dn$-uple \icat{}} is a coCartesian fibration $\mathcal{M}
  \to \Dnop$ such that for any $I \in \Dnop$, the functor
  \[\mathcal{M}_{I} \to \lim_{C \to I \in \CellnIop}
  \mathcal{M}_{C},\] induced by the coCartesian morphisms over the inert
  morphisms $C \to I$ in $\Dn$, is an equivalence.
\end{defn}
We can now weaken this definition in the same way as that which gave us the definition of \dniopds{} from that of $\Dn$-monoidal
\icats{}:
\begin{defn}\label{defn:gDniopd}
  A \emph{\gDniopd{}} is a functor of \icats{} $\pi \colon \mathcal{M}
  \to \Dnop$ such that:
  \begin{enumerate}[(i)]
  \item For every inert morphism  $\phi \colon I \to J$ in $\Dnop$ and
    every $X \in \mathcal{O}_{I}$, there exists
    a $\pi$-coCartesian edge $X \to \phi_{!}X$ over $\phi$.
  \item For every $I$ in $\Dnop$, the functor
    \[ \mathcal{M}_{I} \to \lim_{C \to I} \mathcal{M}_{C}, \]
    induced by the coCartesian arrows over the inert maps $C \to I$ in
    $\CellnIop$, is an equivalence.
  \item Given $Y$ in $\mathcal{O}_{J}$, choose a coCartesian lift
    $\overline{\eta} \colon (\CellnIop)^{\triangleleft} \to
    \mathcal{O}$ of the diagram of inert morphisms $J \to C$ with
    $\overline{\eta}(-\infty) \simeq Y$. Then for any map $\phi \colon
    I \to J$ in $\Dnop$ and $X \in \mathcal{O}_{I}$, the diagram
    $\overline{\eta}$ induces an equivalence
    \[ \Map_{\mathcal{O}}^{\phi}(X, Y) \simeq \lim_{i \colon C \to I
      \in \CellnIop} \Map_{\mathcal{O}}^{i \circ \phi}(X,
    \overline{\eta}(i)).\]
    (Equivalently, any coCartesian lift of the diagram
    $(\CellnIop)^{\triangleleft} \to \Dnop$ is a $\pi$-limit diagram
    in $\mathcal{O}$.)
  \end{enumerate}
\end{defn}

\begin{remark}
  A \defterm{$\Dn$-uple \icat{}} as we defined it above is the same
  thing as a generalized \Dniopd{} that is also a
  coCartesian fibration.
\end{remark}

\begin{defn}
  Let $\pi \colon \mathcal{M} \to \simp^{\op}$ be a
  (generalized) \Dniopd{}. We say that
  a morphism $f$ in $\mathcal{M}$ is \defterm{inert} if it
  is coCartesian and $\pi(f)$ is an inert morphism in
  $\simp^{\op}$. We say that $f$ is \defterm{active} if
  $\pi(f)$ is an active morphism in $\simp^{\op}$.
\end{defn}

\begin{lemma}
  The active and inert morphisms form a factorization system on any
  \gDniopd{}.
\end{lemma}
\begin{proof}
  This is a special case of \cite[Proposition 2.1.2.5]{HA}.
\end{proof}

\begin{defn}
  A morphism of (generalized) \Dniopds{} is a commutative diagram
  \opctriangle{\mathcal{M}}{\mathcal{N}}{\Dnop,}{\phi}{}{} where
  $\mathcal{M}$ and $\mathcal{N}$ are (generalized) \Dniopds{}, such
  that $\phi$ carries inert morphisms in $\mathcal{M}$ to inert
  morphisms in $\mathcal{N}$. We will also refer to a morphism of
  (generalized) \Dniopds{} $\mathcal{M} \to \mathcal{N}$ as an
  \emph{$\mathcal{M}$-algebra} in $\mathcal{N}$; we write
  $\Alg^{n}_{\mathcal{M}}(\mathcal{N})$ for the \icat{} of
  $\mathcal{M}$-algebras in $\mathcal{N}$, defined as a full
  subcategory of the \icat{} of functors $\mathcal{M} \to \mathcal{N}$
  over $\Dnop$.
\end{defn}

\begin{defn}
  If $\mathcal{M}$ and $\mathcal{N}$ are $\Dn$-uple \icats{}, a
  \emph{$\Dn$-uple functor} from $\mathcal{M}$ to $\mathcal{N}$ is a
  commutative diagram
  \opctriangle{\mathcal{M}}{\mathcal{N}}{\Dnop,}{\phi}{}{} where
  $\phi$ preserves all coCartesian morphisms; if $\mathcal{M}$ and
  $\mathcal{N}$ are in fact $\Dn$-monoidal \icats{} we will also refer
  to $\Dn$-uple functors as \emph{$\Dn$-monoidal functors}. We write
  $\Fun^{\otimes,n}(\mathcal{M}, \mathcal{N})$ for the \icat{} of
  $\Dn$-uple functors, defined as a full subcategory of the \icat{} of
  functors $\mathcal{M} \to \mathcal{N}$ over $\Dnop$.
\end{defn}

\subsection{Iterated Bimodules for $\mathbb{E}_{n}$-Algebras and their
  Tensor Products}\label{subsec:Enbimod}
In \S\ref{sec:cartenalg} we considered iterated bimodules for
$\En$-algebras as monoids for the overcategories $\DnIop$. Using
\gdniopds{} we now have a natural way to extend this definition to the
non-Cartesian setting, because of the following observation:
\begin{lemma}
  Let $I$ be any object of $\Dn$. Then the forgetful functor $\DnIop \to \Dnop$ is a $\Dn$-uple \icat{}.
\end{lemma}
\begin{proof}
  The forgetful functor $\DnIop \to \Dnop$ is the coCartesian
  fibration associated to the functor \[\Hom_{\Dn}(\blank, I) \colon
  \Dnop \to \Set.\] This fibration is a $\Dn$-uple \icat{} \IFF{} the
  associated functor satisfies the Segal condition, which it 
  does (for instance since, if $I = ([i_{1}],\ldots,[i_{n}])$, it is
  the product of the functors $\Hom_{\simp}(\blank, [i_{k}])$ which
  satisfy the Segal condition for $\Dop$).
\end{proof}

By Corollary~\ref{cor:unDnisEn} $\Dnop$-algebras in a $\Dn$-monoidal
\icat{} $\mathcal{C}$ are equivalent to $\En$-algebras. To define the
$n$-fold category object $\fALG_{n}(\mathcal{C})$ in $\CatI$ of
$\En$-algebras, a natural choice for the \icat{} of objects is thus
$\Alg^{n}_{\Dnop}(\mathcal{C})$. Similarly, the $n$ different \icats{} of
1-morphisms are given by \[\fALG_{n}(\mathcal{C})_{(1,0,\ldots,0)} :=
\Alg^{n}_{\Dop_{/[1]} \times \simp^{(n-1),\op}}(\mathcal{C}),\]
\[\fALG_{n}(\mathcal{C})_{(0,1,0,\ldots,0)}:= \Alg^{n}_{\Dop \times
  \Dop_{/[1]} \times \simp^{(n-2),\op}}(\mathcal{C}),\] \[\ldots,\]
\[\fALG_{n}(\mathcal{C})_{(0,\ldots,0,1) :=
}\Alg^{n}_{\simp^{(n-1),\op} \times \Dop_{/[1]}}(\mathcal{C}),\] and
more generally the \icats{} of commutative $k$-cubes are given by
\[\fALG_{n}(\mathcal{C})_{I} := \Alg^{n}_{\DnIop}(\mathcal{C})\] where
$I = ([i_{1}],\ldots,[i_{n}])$ with each $i_{j}$ either $0$ are $1$
and exactly $k$ $1$'s. To define the remaining \icats{}
$\fALG_{n}(\mathcal{C})_{I}$ we must define an appropriate notion of
\emph{composite} $\DnIop$-algebras; luckily, there is a natural
generalization of our definition in the case $n = 1$:
\begin{defn}
  We say a morphism $(\phi_{1},\ldots,\phi_{n})$ in $\Dn$ is
  \emph{cellular} if $\phi_{i}$ is cellular for all $i$. For $I \in
  \Dn$, we write $\LnI$ for the full subcategory of $\DnI$ spanned by
  the cellular maps.
\end{defn}

\begin{lemma}
  The projection $\LnIop \to \Dnop$ is a \gDniopd{}, and the inclusion
  $\tau_{I} \colon \LnIop \hookrightarrow \DnIop$ is a morphism of
  \gDniopds{}.
\end{lemma}
\begin{proof}
  As Lemma~\ref{lem:Lbrnopgnsiopd}, using the $\Dn$-analogue of
  Lemma~\ref{lem:fullsubgnsiopd}.
\end{proof}

\begin{propn}
  For every $I \in \Dn$, the inclusion $\tau_{I}\colon \LnIop \to
  \DnIop$ is extendable.
\end{propn}
\begin{proof}
  We must show that for any $I \in \Dn$ and any map $\xi \colon J \to
  I$ in $\Dn$, the map
  \[ (\LnIop)^{\txt{act}}_{/\xi} \to \prod_{\phi \colon C_{n}\to J}
  (\LnIop)^{\txt{act}}_{/\xi \phi} \]
  is cofinal, or equivalently that the map
  \[ (\LnI)^{\txt{act}}_{\xi/} \to \prod_{\phi \colon C_{n}\to J}
  (\LnI)^{\txt{act}}_{\xi \phi /} \]
  is coinitial. This map decomposes as a product, hence since a
  product of coinitial maps is coinitial this follows from the proof of
  Proposition~\ref{propn:tau1ext}.
\end{proof}

\begin{defn}
  We say a $\Dn$-monoidal \icat{} \emph{has good relative tensor
    products} if it is $\tau_{I}$-compatible for all $I \in
  \Dnop$. Similarly, we say a $\Dn$-monoidal functor is
  \emph{compatible with relative tensor products} if it is
  $\tau_{I}$-compatible for all $I$.
\end{defn}

Applying Corollary~\ref{cor:opdkanext}, we get:
\begin{propn}
  Supppose $\mathcal{C}$ is a $\Dn$-monoidal \icat{} with good relative
  tensor products. Then the restriction $\tau_{I}^{*} \colon
  \Algn_{\DnIop}(\mathcal{C}) \to
  \Algn_{\LnIop}(\mathcal{C})$ has a fully faithful left adjoint
  $\tau_{I,!}$.\qed
\end{propn}

Next, we observe that the notion of having good relative tensor
products has a simple equivalent reformulation:
\begin{lemma}\label{lem:Dnmongrtp}
  Let $\mathcal{C}$ be a $\Dn$-monoidal \icat{}. Then the following
  are equivalent:
  \begin{enumerate}[(1)]
  \item $\mathcal{C}$ has good relative tensor products.
  \item Any one of underlying monoidal \icats{} of $\mathcal{C}$
    (obtained by pulling back along the inclusions $\{[1]\} \times
    \cdots \times \Dop \times \cdots \times \{[1]\} \hookrightarrow
    \Dnop$) has good relative tensor products in the sense of
    Definition~\ref{defn:mongrtp}.
  \item Any one of the underlying monoidal \icats{} of $\mathcal{C}$
    satisfies the criterion of Lemma~\ref{lem:mongrtpcond}.
  \end{enumerate}
  Moreover, a $\Dn$-monoidal functor is compatible with relative
  tensor products \IFF{} any one of its underlying monoidal
  functors is compatible with relative tensor products.
\end{lemma}
\begin{proof}
  By definition, we must show that for any $\LnIop$-algebra $A$ in
  $\mathcal{C}$ and any $X \in \DnIop$, the colimit of the induced
  diagram $(\LnIop)^{\txt{act}}_{/X} \to \mathcal{C}$ exists and is
  preserved tensoring with any object of $\mathcal{C}$ using each of
  the $n$ tensor products. But since the $\Dn$-monoidal \icat{}
  $\mathcal{C}$ arises from an $\En$-monoidal \icat{} by
  Corollary~\ref{cor:DnmonisEnmon}, these $n$ tensor product functors
  are all equivalent. It therefore suffices to show that if one of the
  underlying monoidal \icats{} of $\mathcal{C}$ has good relative
  tensor products, then the colimits above exist in $\mathcal{C}$ and
  are preserved by tensoring (on either side) with any object of
  $\mathcal{C}$.

  But the category $(\LnIop)^{\txt{act}}_{/X}$ decomposes as a product
  $\prod_{k=1}^{n} (\bbLambda_{/[i_{k}]}^{\op})^{\txt{act}}_{/X_{k}}$ (where $I
  = ([i_{1}],\ldots,[i_{n}])$ and $X = (X_{1},\ldots,X_{n})$), and for
  each $Y_{k} \in (\bbLambda_{/[i_{k}]}^{\op})^{\txt{act}}_{/X_{k}}$
  with $k \neq j$, the restriction of the diagram to
  \[ \{Y_{1}\} \times \cdots \times
  (\bbLambda_{/[i_{j}]}^{\op})^{\txt{act}}_{/X_{j}} \times \cdots
  \times \{Y_{n}\} \to \mathcal{C} \]
  is obtained by tensoring a number of diagrams associated to
  $\bbLambda_{/[i_{j}]}^{\op}$-algebras in $\mathcal{C}$ with some
  fixed objects. By siftedness, the colimits of these diagrams
  therefore exist in $\mathcal{C}$, and our desired colimit can be
  obtained by an iterated colimit of such diagrams. It follows that 
  the colimit over $(\LnIop)^{\txt{act}}_{/X}$ does indeed exist, and
  is preserved under tensoring, as required. Similarly, a
  $\Dn$-monoidal functor is compatible with relative tensor products
  \IFF{} one of its underlying monoidal functors is.
\end{proof}

We can now define the \icats{} $\fALG_{n}(\mathcal{C})_{I}$ for all
$I$:
\begin{defn}
  Let $\mathcal{C}$ be a $\Dn$-monoidal \icat{} with good relative
  tensor products. We say that a $\DnIop$-algebra $M$ in $\mathcal{C}$
  is \emph{composite} if the counit map $\tau_{I,!}\tau_{I}^{*}M \to
  M$ is an equivalence, or equivalently if $M$ is in the essential
  image of the functor $\tau_{I,!}$. We write
  $\fALG_{n}(\mathcal{C})_{I}$ for the full subcategory of
  $\Algn_{\DnIop}(\mathcal{C})$ spanned by the composite
  $\DnIop$-algebras.
\end{defn}

\subsection{The $(n+1)$-Fold $\infty$-Category of $\En$-Algebras}\label{subsec:np1foldcat}
Our goal in this subsection is to extend the results of
\S\ref{subsec:alg1segcond} and \S\ref{subsec:functcomp} to the case of
$\En$-algebras, i.e. to prove that the \icats{}
$\fALGn(\mathcal{C})_{I}$ satisfy the Segal condition and are
functorial in $I$. Luckily, it turns out that these results both
follow from those in the case $n = 1$ by simple inductions.

We first prove that $\fALG_{n}(\mathcal{C})_{I}$ satisfies the Segal
condition.  Let $(\DnI)^{\amalg,\op}$ denote the ordinary
colimit \[\colim_{I \to C \in (\txt{Cell}^{n}_{I/})^{\op}} \DnCop\] in
(marked) simplicial sets (over $\Dnop$). From the structure of
$\txt{Cell}^{n}$ it is easy to see that this colimit can be written as
an iterated pushout along injective maps of simplicial sets, so this
is a homotopy colimit in the generalized $\Dn$-\iopd{} model structure
of \S\ref{subsec:dniopdcat}. We wish to prove that the inclusion
$\DnIAop \hookrightarrow \LnIop$ is a trivial cofibration in this
model structure:
\begin{lemma}\label{lem:DnIAisproduct}
  Suppose $I = ([i_{1}],\ldots,[i_{n}])$ is an object of $\Dn$. Then
  the natural map \[\DnIAop \to \prod_{p = 1}^{n}
  (\simp_{/[i_{p}]})^{\amalg,\op}\] is an isomorphism.
\end{lemma}
\begin{proof}
  The category $(\txt{Cell}^{n}_{I/})^{\op}$ is isomorphic to the
  product $\prod_{k} (\txt{Cell}^{1}_{[i_{k}]/})^{\op}$, and the
  functor $(I \to C) \mapsto \DnCop$ is isomorphic to the product of
  the functors $([i_{k}] \to [j]) \mapsto \simp^{\op}_{/[j]}$ (where
  $j = 0$ or $1$). Since the Cartesian product of (marked) simplicial
  sets preserves colimits in each variable, the result follows.
\end{proof}

\begin{propn}\label{propn:cellmodel}
  Let $I$ be an object of $\Dn$. The inclusion $\DnIAop
  \to \LnIop$ is a trivial cofibration in the model category
  $(\sSet^{+})_{\mathfrak{O}_{n}}$.
\end{propn}

\begin{proof}
  Suppose $I = ([i_{1}],\ldots,[i_{n}])$. By
  Lemma~\ref{lem:DnIAisproduct} we may identify the inclusion $\DnIAop
  \to \LnIop$ with the product over $p =1,\ldots,n$ of the inclusions
  $(\simp_{/[i_{p}]})^{\amalg,\op} \hookrightarrow
  \bbLambda_{/[i_{p}]}^{\op}$. By Proposition~\ref{propn:catpattprod} and
  Corollary~\ref{cor:Ontimes}, the Cartesian product is a left
  Quillen bifunctor $(\sSet^{+})_{\mathfrak{O}_{1}} \times
  (\sSet^{+})_{\mathfrak{O}_{n-1}} \to
  (\sSet^{+})_{\mathfrak{O}_{n}}$, so by induction it suffices to
  prove the result in the case $n = 1$, which is
  Proposition~\ref{propn:cellmodel1}.
\end{proof}

\begin{cor}\label{cor:algLIseg}
  Let $\mathcal{M}$ be a \gDniopd{}. The restriction map
  \[\Alg^{n}_{\LnIop}(\mathcal{M}) \to
  \lim_{I \to C \in \CellnIop}\Alg^{n}_{\DnCop}(\mathcal{M})\] is an equivalence of \icats{}.
\end{cor}
\begin{proof}
  Since the model category
  $(\sSet^{+})_{\mathfrak{O}_{n}^{\txt{gen}}}$ is enriched in marked
  simplicial sets and $\DnIAop \hookrightarrow \LnIop$ is a trivial
  cofibration by Proposition~\ref{propn:cellmodel}, for any \gDniopd{}
  $\mathcal{M}$ the restriction map $\Alg^{n}_{\LnIop}(\mathcal{M})
  \to \Alg^{n}_{\DnIAop}(\mathcal{M})$ is a trivial Kan
  fibration. Moreover, we have an equivalence of
  \icats{} \[\Alg^{n}_{\DnIAop}(\mathcal{M}) \simeq \lim_{I \to C \in
    \CellnIop}\Alg^{n}_{\DnCop}(\mathcal{M})\] since the colimit
  $\DnIAop$ is a homotopy colimit.
\end{proof}

\begin{cor}\label{cor:ALGnseg}
  Let $\mathcal{C}$ be a $\Dn$-monoidal \icat{} with good relative
  tensor products. Then the natural restriction map
  \[ \fALG_{n}(\mathcal{C})_{I} \to \lim_{I \to C \in \CellnIop} \fALG_{n}(\mathcal{C})_{C}
  \]
  is an equivalence.
\end{cor}
\begin{proof}
  This map factors as a composite of the maps
  \[ \fALG_{n}(\mathcal{C})_{I} \to \Alg^{n}_{\LnIop}(\mathcal{C}) \to
  \lim_{I \to C \in \CellnIop} \fALG_{n}(\mathcal{C})_{C},\] where the
  first is an equivalence by definition and the second by
  Corollary~\ref{cor:algLIseg}.
\end{proof}

Next we prove that the \icats{} $\fALGn(\mathcal{C})_{I}$ for $I \in
\Dnop$ give a multisimplicial object.
\begin{defn}
  Suppose $\mathcal{C}$ is a $\Dn$-monoidal \icat{}. Let
  $\ofALG_{n}(\mathcal{C}) \to \Dnop$ denote a coCartesian fibration
  associated to the functor $\Dnop \to \CatI$ that sends $I$ to 
  $\Alg^{n}_{\DnIop}(\mathcal{C})$. We write
  $\fALGn(\mathcal{C})$ for the full subcategory of
  $\ofALGn(\mathcal{C})$ spanned by the objects of
  $\fALG_{n}(\mathcal{C})_{I}$ for all $I$, i.e. by the composite
  $\DnIop$-algebras for all $I \in \Dnop$.
\end{defn}

We wish to show that the projection $\fALG_{n}(\mathcal{C}) \to \Dnop$
is a coCartesian fibration. To prove this, we extend the definitions
of \S\ref{subsec:functcomp} in the obvious way:
\begin{defn}
  Suppose $\Phi = (\phi_{1},\ldots,\phi_{n}) \colon I \to J$ is an
  injective morphism in $\Dn$. We say that a morphism
  $(\alpha_{1},\ldots,\alpha_{n}) \colon K \to J$ in $\Dn$ is
  \emph{$\Phi$-cellular} if $\alpha_{i}$ is $\phi_{i}$-cellular for
  all $i = 1,\ldots,n$.
\end{defn}

\begin{defn}
  For $I \in \Dn$ and $\Phi \colon J \to I$ an injective morphism in
  $\Dn$, we write $\LnIP$ for the full subcategory of $\DnI$ spanned
  by the $\Phi$-cellular maps to $I$.
\end{defn}

\begin{propn}\label{propn:cellcoinitial}\ 
  \begin{enumerate}
  \item If $\Phi \colon J \to I$ is an injective morphism in $\Dn$,
    then for any $\Gamma \colon K \to J$ the map \[ \Phi_{*} \colon
    (\Ln_{/J})^{\txt{act}}_{\Gamma/} \to
    (\LnIP)^{\txt{act}}_{\Phi\Gamma/} \] given by composition with
    $\Phi$ is coinitial.
  \item If $\Phi \colon J \to I$ is a surjective morphism in $\Dn$,
    then for any $\Gamma \colon K \to J$ the map
    \[ (\Ln_{/J})^{\txt{act}}_{\Gamma/} \to
    (\Ln_{/I})^{\txt{act}}_{\Phi\Gamma} \]
    given by composition with $\Phi$ is coinitial.
  \end{enumerate}
\end{propn}

\begin{proof}
  Since products of coinitial functors are coinitial, this is
  immediate from Proposition~\ref{propn:cell1coinitial}.
\end{proof}

\begin{cor}\label{cor:ALGncoCart}
  Suppose $\mathcal{C}$ is a $\Dn$-monoidal \icat{} compatible with
  small colimits. Then the projection $\fALGn(\mathcal{C}) \to \Dnop$
  is a coCartesian fibration.
\end{cor}
\begin{proof}
  Since $\ofALGn(\mathcal{C}) \to \Dnop$ is a coCartesian fibration,
  it suffices to show that if $X$ is an object of
  $\fALGn(\mathcal{C})$ over $I \in \Dnop$, and $X \to \bar{X}$ is a
  coCartesian morphism in $\ofALGn(\mathcal{C})$ over $\Phi
  \colon J \to I$ in $\Dn$, then $\bar{X}$ is also in
  $\fALGn(\mathcal{C})$. This follows from the same argument as in the
  proof of Corollary~\ref{cor:ALG1coCart}, using
  Proposition~\ref{propn:cellcoinitial}. 
\end{proof}

Combining  Corollary~\ref{cor:ALGncoCart} with
Corollary~\ref{cor:ALGnseg}, we have proved:
\begin{thm}\label{thm:ALGnuple}
  Let $\mathcal{C}$ be a $\Dn$-monoidal \icat{} with good relative
  tensor products. Then the projection $\fALG_{n}(\mathcal{C})
  \to \Dnop$ is a $\Dn$-uple \icat{}.\qed
\end{thm}

\begin{remark}\label{rmk:ALGnfunct}
  Suppose $\mathcal{C}^{\otimes}$ and $\mathcal{D}^{\otimes}$ are
  $\Dn$-monoidal \icats{} with good relative tensor products, and
  $f^{\otimes}\colon \mathcal{C}^{\otimes} \to \mathcal{D}^{\otimes}$
  is a $\Dn$-monoidal functor compatible with relative tensor
  products. Composition with $f^{\otimes}$ induces a functor
  $f_{*} \colon \ofALG_{n}(\mathcal{C}) \to
  \ofALG_{n}(\mathcal{D})$. It follows from Lemma~\ref{lem:icompftr}
  that this functor takes the full subcategory
  $\fALG_{n}(\mathcal{C})$ into $\fALG_{n}(\mathcal{D})$, and so
  induces a map $f_{*} \colon \fALG_{n}(\mathcal{C}) \to
  \fALG_{n}(\mathcal{D})$ of $(n+1)$-fold \icats{}.
\end{remark}

\begin{defn}
  Let $\mathcal{C}$ be a $\Dn$-monoidal \icat{} with good relative
  tensor products. We write $\fAlg_{n}(\mathcal{C})$ for the
  completion $L_{n}U_{\Seg}i\fALG_{n}(\mathcal{C})$ of the underlying
  $(n+1)$-fold Segal space $U_{\Seg}i\fALG_{n}(\mathcal{C})$ of the
  image of $\fALG_{n}(\mathcal{C})$ under the forgetful functor $i
  \colon \txt{Upl}_{\infty}^{\simp^{n}} \simeq \Cat^{n}(\CatI) \to
  \Cat^{n+1}(\mathcal{S})$. Thus $\fAlg_{n}(\mathcal{C})$ is a
  complete $(n+1)$-fold Segal space, i.e. an $(\infty,n+1)$-category.
\end{defn}

\subsection{Functoriality and Monoidal Structures}\label{subsec:ALGnmon}
Our goal in this subsection is to show that the $(n+1)$-fold \icats{}
$\fALGn(\mathcal{C})$ we constructed above are functorial in
$\mathcal{C}$, and moreover that this functor is lax monoidal. From
this it will follow immediately that if $\mathcal{C}$ is an
$\mathbb{E}_{n+m}$-monoidal \icat{} with good relative tensor
products, then the $(\infty,n+1)$-category $\fAlg_{n}(\mathcal{C})$
inherits a canonical $\mathbb{E}_{m}$-monoidal structure. We begin by
introducing some notation for the source of our functor:
\begin{defn}
  Let $\LMonIDn$ denote the \icat{} of $\Dn$-monoidal \icats{} and
  $\Dn$-monoidal functors. We write $\LMonIDnGR$ for the subcategory
  of $\LMonIDn$ determined by the $\Dn$-monoidal \icats{} with good
  relative tensor products and the $\Dn$-monoidal functors compatible
  with these. If $n = 1$ we also denote this by $\LMonIGR$.
\end{defn}

\begin{defn}
  Let $\Algn \to (\OpdIDng)^{\op} \times \LOpdIDng$ be defined in the
  same way as the coCartesian fibration in \S\ref{subsec:algfun}, but
  allowing the target \gdniopds{} to be large. Then we define
  $\ofALGn$ by the pullback square
  \nolabelcsquare{\ofALGn}{\AlgDn}{\Dnop \times
    \LMonIDnGR}{(\OpdIDng)^{\op} \times \LOpdIDng,} where the bottom
  horizontal map is the product of $\Dnop_{/(\blank)} \colon \Dnop \to
  (\OpdIDng)^{\op}$ and the forgetful functor from large
  $\Dn$-monoidal \icats{} with good relative tensor products to large
  \gDniopds{}. Write $\fALGn$ for the full subcategory of $\ofALGn$
  spanned by the objects in $\fALGn(\mathcal{C})$ for all
  $\Dn$-monoidal \icats{} $\mathcal{C}$ in $\LMonIDnGR$.
\end{defn}

\begin{propn}\label{propn:ALGnfun}
  The restricted projection $\fALGn \to \Dnop \times \LMonIDnGR$ is a
  coCartesian fibration.
\end{propn}
\begin{proof}
  Suppose $X$ is an object of $\fALGn$ over $(I, \mathcal{C})$ and
  $(\Phi, F) \colon (I, \mathcal{C}) \to (J, \mathcal{D})$ is a
  morphism in $\Dnop \times \LMonIDnGR$. Then it suffices to prove
  that if $X \to (\Phi,F)_{!}X$ is a coCartesian morphism in $\ofALGn$,
  then $(\Phi,F)_{!}X$ lies in $\fALGn$.

  It is enough to consider the morphisms $(\Phi, \id_{\mathcal{C}})$
  and $(\id_{I}, F)$ separately. We know that
  $(\Phi,\id_{\mathcal{C}})_{!}X$ is in $\fALGn$ by
  Corollary~\ref{cor:ALGncoCart}, and the object $(\id,F)_{!}X$ lies
  in $\fALGn$ by Remark~\ref{rmk:ALGnfunct}.
\end{proof}

\begin{cor}
  There is a functor $\fALGn(\blank) \colon \LMonIDnGR \to
  \Cat^{n}(\CatI)$ that sends $\mathcal{C}$ to $\fALGn(\mathcal{C})$.
\end{cor}
\begin{proof}
  By Proposition~\ref{propn:ALGnfun} there is a functor
  $\LMonIDnGR \times \Dnop \to \CatI$, or equivalently \[\LMonIDnGR \to
  \Fun(\Dnop, \CatI),\] associated to the coCartesian fibration
  $\fALG_{n} \to \Dnop \times \LMonIDnGR$. By
  Corollary~\ref{cor:ALGnseg} this functor
  lands in the full subcategory $\Cat^{n}(\CatI)$ of $n$-uple
  category objects.
\end{proof}

\begin{lemma}\ 
  \begin{enumerate}[(i)]
  \item The \icat{} $\LMonIGR$ has products, and the
    forgetful functor $\LMonIGR \to \LMonI$ preserves these.
  \item The \icat{} $\LMonIDnGR$ is equivalent to
    $\Alg^{n-1}_{\simp^{n-1,\op}}(\LMonIGR)$.
  \item The \icat{} $\LMonIDnGR$ has products for all $n$, and the
    forgetful functor $\LMonIDnGR \to \LMonIDn$ preserves these.
  \end{enumerate}
\end{lemma}
\begin{proof}
  Suppose $\mathcal{C}^{\otimes}$ and $\mathcal{D}^{\otimes}$ are
  monoidal \icats{} with good relative tensor products. We will show
  that the product $(\mathcal{C} \times \mathcal{D})^{\otimes} :=
  \mathcal{C}^{\otimes} \times_{\Dop} \mathcal{D}^{\otimes}$ in
  $\LMonI$ is also a product in the subcategory $\LMonIGR$. Thus, we
  need to prove
  \begin{enumerate}[(1)]
  \item The product $(\mathcal{C} \times \mathcal{D})^{\otimes}$ has
    good relative tensor products.
  \item If $\mathcal{E} \in \LMonIGR$, then a monoidal functor $F
    \colon \mathcal{E} \to \mathcal{C} \times \mathcal{D}$ is
    compatible with relative tensor products \IFF{} the monoidal
    functors $F_{1} \colon \mathcal{E} \to \mathcal{C}$ and $F_{2}
    \colon \mathcal{E} \to \mathcal{D}$ obtained by composing with the
    projections are both compatible with relative tensor products.
  \end{enumerate}
  To prove (1), we use Lemma~\ref{lem:mongrtpcond}. A
  $\bbLambda_{/[2]}^{\op}$-algebra in $\mathcal{C} \times \mathcal{D}$
  is given by a $\bbLambda_{/[2]}^{\op}$-algebra $A$ in $\mathcal{C}$
  and an algebra $B$ in $\mathcal{D}$. Moreover, the induced diagram
  $(\bbLambda_{/[2]}^{\op})^{\txt{act}}_{/(0,2)} \to \mathcal{C}
  \times \mathcal{D}$ is the composite
  \[ (\bbLambda_{/[2]}^{\op})^{\txt{act}}_{/(0,2)} \to
  (\bbLambda_{/[2]}^{\op})^{\txt{act}}_{/(0,2)} \times
  (\bbLambda_{/[2]}^{\op})^{\txt{act}}_{/(0,2)} \to \mathcal{C} \times
  \mathcal{D}.\]
  Since $(\bbLambda_{/[2]}^{\op})^{\txt{act}}_{/(0,2)}$ is sifted by
  Lemma~\ref{lem:Lbrnsift}, this colimit is therefore given by the
  pair of colimits in $\mathcal{C}$ and $\mathcal{D}$. It follows that
  $\mathcal{C} \times \mathcal{D}$ has good relative tensor products.
  (2) follows from a similar argument, again using siftedness. This
  proves (i).

  To prove (ii), observe that $\LMonIDnGR$ and
  $\Alg^{n-1}_{\simp^{n-1,\op}}(\LMonIGR)$ can both be identified with
  subcategories of $\LMonIDn$, and it follows from
  Lemma~\ref{lem:Dnmongrtp} that they are the same subcategory. (iii)
  now follows by the same argument as for (i), or using the
  description of limits in \icats{} of algebras from
  \rcite{HA}{Corollary 3.2.2.4}.  
\end{proof}

\begin{defn}
  Let $\Alg^{n,\otimes} \to (\OpdIDng)^{\op} \times
  (\LOpdIDng)^{\times}$ be the obvious variant of the coCartesian
  fibration of \gsiopds{} defined in \S\ref{subsec:algfun}. Then we
  define $\ofALGn^{\otimes}$ by the pullback square
  \nolabelcsquare{\ofALGn^{\otimes}}{\Alg^{n,\otimes}}{\Dnop \times
    \LMonIDnGRtens}{(\OpdIDng)^{\op} \times (\LOpdIDng)^{\times},}
  where the bottom horizontal map is the product of
  $\Dnop_{/(\blank)}$ and the symmetric monoidal structure on the
  forgetful functor that arises since this preserves products. Then
  $\ofALGn^{\otimes} \to \Dnop \times \LMonIDnGRtens$ is a coCartesian
  fibration of \gsiopds{}. Write $\fALGn^{\otimes}$ for the full
  subcategory of $\ofALGn^{\otimes}$ spanned by the objects
  corresponding to lists of objects of $\fALGn$.
\end{defn}

\begin{propn}
  The restricted projection $\fALGn^{\otimes} \to \Dnop \times
  \LMonIDnGRtens$ is a coCartesian fibration of \gsiopds{}.
\end{propn}
\begin{proof}
  Since $\fALGn^{\otimes}$ is the full subcategory of
  $\ofALGn^{\otimes}$ determined by the full subcategory $\fALGn$ of
  $\ofALGn \simeq (\ofALGn^{\otimes})_{\angled{1}}$, it is a
  \gsiopd{}. Moreover, it is easy to see from
  Proposition~\ref{propn:ALGnfun} and Remark~\ref{rmk:opdkanextprod}
  that $\fALGn^{\otimes}$ inherits coCartesian morphisms from
  $\ofALGn^{\otimes}$.
\end{proof}

\begin{cor}
  $\fALG_{n}$ defines a lax symmetric monoidal functor $\LMonIDnGR \to
  \Cat^{n}(\CatI)$. In particular, if $\mathcal{C}$ is an
  $\mathbb{E}_{n+m}$-monoidal \icat{} then $\fALG_{n}(\mathcal{C})$
  inherits a canonical $\mathbb{E}_{m}$-monoidal structure.
\end{cor}
\begin{proof}
  Since $\fALG_{n}^{\otimes} \to \Dnop \times \LMonIDnGRtens$ is a
  coCartesian fibration of \gsiopds{}, the associated functor $\Dnop
  \times \LMonIDnGRtens \to \CatI$ is a monoid object. The
  corresponding functor $\LMonIDnGRtens \to \Fun(\simp^{\op}, \CatI)$
  is then also a monoid object, and lands in the full subcategory
  $\Cat^{n}(\CatI)$ of $n$-fold category objects. This therefore
  corresponds to a lax monoidal functor $\LMonIDnGR \to
  \Cat^{n}(\CatI)$ by \rcite{HA}{Proposition 2.4.2.5}.
\end{proof}

\begin{cor}
  $\fAlg_{n}$ defines a lax symmetric monoidal functor $\LMonIDnGR \to
  \Cat_{(\infty,n)}$. In particular, if $\mathcal{C}$ is an
  $\mathbb{E}_{n+m}$-monoidal \icat{}, then $\fAlg_{n}(\mathcal{C})$
  inherits a canonical $\mathbb{E}_{m}$-monoidal structure.
\end{cor}
\begin{proof}
  By definition, $\fAlg_{n}$ is the composite of the lax monoidal
  functor \[\fALG_{n} \colon \LMonIDnGR \to \Cat^{n}(\CatI)\] with the
  inclusion $i \colon \Cat^{n}(\CatI) \to \Cat^{n+1}(\mathcal{S})$,
  the functor $U_{\Seg} \colon \Cat^{n+1}(\mathcal{S}) \to
  \Seg_{n}(\mathcal{S})$ that takes an $n$-uple Segal space to its
  underlying $n$-fold Segal space, and $L_{n} \colon
  \Seg_{n}(\mathcal{S}) \to \txt{CSS}_{n}(\mathcal{S}) \simeq
  \Cat_{(\infty,n)}$, the completion functor. The functor
  $L_{n}U_{\Seg}i$ is symmetric monoidal by
  Remark~\ref{rmk:complsegnmon}, and so the composite $\LMonIDnGR \to
  \Cat_{(\infty,n)}$ is also lax symmetric monoidal.
\end{proof}

\subsection{The Mapping $(\infty,n)$-Categories of
  $\fAlg_{n}(\mathcal{C})$}\label{subsec:Algnmaps}

Our goal in this subsection is to prove that if $A$ and $B$ are
$\En$-algebras in an $\En$-monoidal \icat{} $\mathcal{C}$, then the
$(\infty,n)$-category $\fAlg_{n}(\mathcal{C})(A,B)$ of maps from $A$
to $B$ in $\fAlg_{n}(\mathcal{C})$ can be identified with the
$(\infty,n)$-category  $\fAlg_{n-1}(\Bimod_{A,B}(\mathcal{C}))$ of
$\mathbb{E}_{n-1}$-algebras in the \icat{} $\Bimod_{A,B}(\mathcal{C})$
of $A$-$B$-bimodules, equipped with a natural
$\mathbb{E}_{n-1}$-monoidal structure.

First we will show that in this situation $\Bimod_{A,B}(\mathcal{C})$
does in fact inherit an $\mathbb{E}_{n-1}$-monoidal structure:

\begin{defn}
  Let $\mathcal{C}$ be a $\simp^{n+1}$-monoidal \icat{}. We write
  $\Bimod^{\otimes}(\mathcal{C})$ for the internal hom
  $\ALG_{\Dop_{/[1]}}^{1,n+1}(\mathcal{C})$ and
  $\Ass^{\otimes}(\mathcal{C})$ for
  $\ALG_{\Dop}^{1,n+1}(\mathcal{C})$. By Lemma~\ref{lem:ALGFUNeq}
  these are both $\Dn$-monoidal \icats{}, and the natural map
  $\Bimod^{\otimes}(\mathcal{C}) \to \Ass^{\otimes}(\mathcal{C})
  \times_{\Dnop} \Ass^{\otimes}(\mathcal{C})$ induced by the map of
  \gnsiopds{} $\Dop \amalg \Dop \to \Dop_{/[1]}$ is a $\Dn$-monoidal
  functor.
\end{defn}

\begin{propn}
  Let $\mathcal{C}^{\otimes}$ be a $\simp^{n+1}$-monoidal \icat{} with
  good relative tensor products. Then the
  projection $\Pi \colon \Bimod^{\otimes}(\mathcal{C}) \to
  \Ass^{\otimes}(\mathcal{C}) \times_{\Dnop}
  \Ass^{\otimes}(\mathcal{C})$ is a coCartesian fibration of
  $\Dn$-monoidal \icats{}.
\end{propn}
\begin{proof}
  We know that the projections $\Bimod^{\otimes}(\mathcal{C}) \to
  \Dnop$ and $\Ass^{\otimes}(\mathcal{C}) \to \Dnop$ are coCartesian
  fibrations, and that the map $\Pi$ preserves coCartesian morphisms. By
  \rcite{freepres}{Proposition 8.3} it thus suffices to check that
  \begin{enumerate}[(a)]
  \item the map on fibres
    \[ \Bimod^{\otimes}(\mathcal{C})_{I} \to
    \Ass^{\otimes}(\mathcal{C})_{I}^{\times 2} \] is a coCartesian
    fibration for all $I \in \Dnop$,
  \item for every map $\phi \colon I \to J$ in $\Dnop$ the induced
    functor $\Bimod^{\otimes}(\mathcal{C})_{I} \to
    \Bimod^{\otimes}(\mathcal{C})_{J}$ takes $\Pi_{I}$-coCartesian
    morphisms to $\Pi_{J}$-coCartesian morphisms.
  \end{enumerate}
  But by Corollary~\ref{cor:inthomeq} we may identify the map
  $\Pi_{I}$ with the map $\Bimod(\mathcal{C}^{\otimes}_{(I,\bullet)})
  \to \Ass(\mathcal{C}^{\otimes}_{(I, \bullet)})^{\times 2}$, which is
  a coCartesian fibration by Proposition~\ref{propn:BimodFibcoCart};
  this proves (a). Moreover, the map
  \[\Bimod^{\otimes}(\mathcal{C})_{I} \to
  \Bimod^{\otimes}(\mathcal{C})_{J}\] induced by $\phi$ can be
  identified with the map $\Bimod(\mathcal{C}^{\otimes}_{I,\bullet})
  \to \Bimod(\mathcal{C}^{\otimes}_{J,\bullet})$ induced by
  composition with $\phi_{!} \colon \mathcal{C}^{\otimes}_{I,\bullet}
  \to \mathcal{C}^{\otimes}_{J,\bullet}$. This is a $\Dn$-monoidal
  functor, and it is compatible with relative tensor products 
  since $\mathcal{C}^{\otimes}$ has good relative tensor products. The
  description of the $\Pi_{I}$-coCartesian morphisms in
  Proposition~\ref{propn:BimodFibcoCart} therefore implies (b).
\end{proof}

\begin{cor}
  Let $\mathcal{C}$ be a $\simp^{n+1}$-monoidal \icat{} with good
  relative tensor products, and suppose
  $A$ and $B$ are $\simp^{n+1,\op}$-algebras in $\mathcal{C}$. Then we
  can regard $A$ and $B$ as $\Dnop$-algebras in
  $\Ass^{\otimes}(\mathcal{C})$. Define an \icat{}
  $\Bimod_{A,B}^{\otimes}(\mathcal{C})$ by the pullback square
  \csquare{\Bimod_{A,B}^{\otimes}(\mathcal{C})}{\Bimod^{\otimes}(\mathcal{C})}{\Dnop}{\Ass^{\otimes}(\mathcal{C})
    \times_{\Dnop} \Ass^{\otimes}(\mathcal{C}).}{}{}{}{(A,B)}
  Then the projection $\Bimod_{A,B}^{\otimes}(\mathcal{C}) \to \Dnop$
  is a $\Dn$-monoidal \icat{} with underlying \icat{}
  $\Bimod_{A,B}(\mathcal{C})$.\qed
\end{cor}

\begin{lemma}\label{lem:bimodgrtp}
  Let $\mathcal{C}$ be a $\simp^{n+1}$-monoidal \icat{} with good
  relative tensor products, and suppose $A$ and
  $B$ are $\simp^{n+1,\op}$-algebras in $\mathcal{C}$. Then the
  $\Dn$-monoidal \icat{} $\Bimod_{A,B}(\mathcal{C})$ has good relative
  tensor products.
\end{lemma}
\begin{proof}
  By Lemma~\ref{lem:Dnmongrtp} it suffices to consider the case $n =
  1$, in which case we use the criterion of
  Lemma~\ref{lem:mongrtpcond}. Suppose given an algebra $U \colon
  \bbLambda^{\op}_{/[2]} \to \Bimod^{\otimes}_{A,B}(\mathcal{C})$. The
  induced diagram $F \colon (\bbLambda^{\op}_{/[2]})^{\txt{act}}_{(0,2)} \to
  \Bimod_{A,B}(\mathcal{C})$ can be identified with the coCartesian
  pushforward to the fibre over $(A,B)$ of the corresponding diagram
  $F' \colon (\bbLambda^{\op}_{/[2]})^{\txt{act}}_{(0,2)} \to
  \Bimod(\mathcal{C})$ for the composite algebra $U' \colon \bbLambda^{\op}_{/[2]} \to
  \Bimod^{\otimes}(\mathcal{C})$. Projecting the latter diagram to
  $\Ass(\mathcal{C}) \times \Ass(\mathcal{C})$ gives the simplicial
  diagrams $A \otimes A^{\otimes \bullet} \otimes A$ and $B \otimes
  B^{\otimes \bullet} \otimes B$ with colimits $A \otimes_{A} A \simeq
  A$ and $B \otimes_{B}B\simeq B$. To see that $F$ has a colimit, it
  then suffices by \rcite{HTT}{Propositions 4.3.1.9 and 4.3.1.10} to
  show that $F'$ has a colimit. For this we use
  Corollary~\ref{cor:bimodcolim}, since
  $(\bbLambda^{\op}_{/[2]})^{\txt{act}}_{(0,2)}$ is sifted by
  Lemma~\ref{lem:Lbrnsift}. The projections to $\mathcal{C}$ of this
  diagram are all relative tensor product diagrams, and so have
  monoidal colimits since $\mathcal{C}$ has good relative tensor
  products, so the colimit of $F'$ does exist in
  $\Bimod(\mathcal{C})$. The colimit in $\Bimod_{A,B}(\mathcal{C})$ is
  moreover preserved under tensoring with objects of
  $\Bimod_{A,B}(\mathcal{C})$ by a similar argument, since the tensor
  product in $A$-$B$-bimodules projects to a relative tensor product
  in $\mathcal{C}$.
\end{proof}

\begin{propn}\label{propn:BimodComposite}
  Let $\mathcal{C}$ be a $\simp^{n+1}$-monoidal \icat{} with good
  relative tensor products, and suppose $A$ and $B$ are
  $\simp^{n+1,\op}$-algebras in $\mathcal{C}$ and $U \colon \DnIop \to
  \Bimod_{A,B}^{\otimes}(\mathcal{C})$ is a $\DnIop$-algebra in
  $\Bimod_{A,B}(\mathcal{C})$. Then $U$ is composite \IFF{} the
  algebra $U'$ in $\Bimod(\mathcal{C})$ obtained by composing with the
  inclusion is a composite $\DnIop$-algebra in
  $\Bimod^{\otimes}(\mathcal{C})$.
\end{propn}
\begin{proof}
  For $X \in \DnIop$ we can conclude, using \rcite{HTT}{Propositions
    4.3.1.9 and 4.3.1.10} as in the proof of
  Lemma~\ref{lem:bimodgrtp}, that the diagram $\overline{\xi}' \colon
  ((\LnIop)^{\txt{act}}_{/X})^{\triangleright} \to
  \Bimod(\mathcal{C})$ induced by $U'$ is a colimit diagram \IFF{} the
  corresponding diagram $\overline{\xi}$ for $U$, which is obtained as
  the coCartesian pushforward of $\overline{\xi}'$ to the fibre over $(A,B)$, is
  a colimit in $\Bimod_{A,B}(\mathcal{C})$ and this is preserved by
  the functors $(f,g)_{!} \colon \Bimod_{A,B}(\mathcal{C}) \to
  \Bimod_{A',B'}(\mathcal{C})$ induced by the coCartesian morphisms
  over any maps $f \colon A \to A'$, $g \colon B\to B'$ of associative
  algebras. On the other hand, we know that $\xi :=
  \overline{\xi}|_{(\LnIop)^{\txt{act}}_{/X}}$ does have a colimit in
  $\Bimod_{A,B}(\mathcal{C})$ whose underlying diagram in
  $\mathcal{C}$ is a monoidal colimit diagram. Using
  Corollary~\ref{cor:bimodcolim} this implies that this colimit is
  necessarily preserved by the functors $(f,g)_{!}$. The two
  conditions are therefore equivalent, as required.
\end{proof}

\begin{cor}\label{cor:ALGpullback}
  Let $\mathcal{C}$ be a $\simp^{n+1}$-monoidal \icat{} compatible
  with geometric realizations and initial objects, and suppose $A$ and
  $B$ are $\simp^{n+1,\op}$-algebras in $\mathcal{C}$. Then we have a
  pullback square
  \csquare{\fALG_{n}(\Bimod_{A,B}(\mathcal{C}))}{\fALG_{n+1}(\mathcal{C})_{[1]}}{*}{\fALG_{n+1}(\mathcal{C})_{[0]}^{\times
      2}}{}{}{}{(A,B)} of $n$-uple category objects in $\CatI$.
\end{cor}
\begin{proof}
  By the universal property of the internal hom
  $\ALG^{n+1,1}_{(\blank)}(\blank)$, we can identify the map
  \[\ofALG_{n+1}(\mathcal{C})_{([1],I)} \to
  \ofALG_{n+1}(\mathcal{C})_{([0],I)}^{\times 2}\] with 
  \[ \Alg^{n}_{\DnIop}(\Bimod^{\otimes}(\mathcal{C})) \to
  \Alg^{n}_{\DnIop}(\Ass^{\otimes}(\mathcal{C}) \times_{\Dnop}
  \Ass^{\otimes}(\mathcal{C})).\] Since $\Alg^{n}_{\DnIop}(\blank)$
  preserves limits, and $\Dnop$ is the final $\Dn$-monoidal \icat{},
  we get a pullback square
  \csquare{\ofALG_{n}(\Bimod_{A,B}(\mathcal{C}))_{I}}{\ofALG_{n+1}(\mathcal{C})_{([1],I)}}{*}{\ofALG_{n+1}(\mathcal{C})_{([0],I),}^{\times
      2}}{}{}{}{(A,B)} natural in
  $I$. Proposition~\ref{propn:BimodComposite} implies that this
  restricts to a pullback square
  \csquare{\fALG_{n}(\Bimod_{A,B}(\mathcal{C}))_{I}}{\fALG_{n+1}(\mathcal{C})_{([1],I)}}{*}{\fALG_{n+1}(\mathcal{C})_{([0],I),}^{\times
      2}}{}{}{}{(A,B)} and the naturality in $I$ then gives a pullback
  square of $n$-fold category objects in $\CatI$, since the inclusion
  of these into all functors $\Dnop \to \CatI$ is a right adjoint and
  so preserves limits.
\end{proof}

From this we can now prove the main result of this subsection:
\begin{thm}\label{thm:Algnmap}
  Let $\mathcal{C}$ be a $\simp^{n+1}$-monoidal \icat{} with good
  relative tensor products, and suppose
  $A$ and $B$ are $\simp^{n+1,\op}$-algebras in $\mathcal{C}$. Then
  the $(\infty,n)$-category $\fAlg_{n+1}(\mathcal{C})(A, B)$ is
  naturally equivalent to $\fAlg_{n}(\Bimod_{A,B}(\mathcal{C}))$.
\end{thm}

This will follow from Corollary~\ref{cor:ALGpullback} together with
the following observation:
\begin{lemma}\label{lem:mapcompl}
  Suppose $\mathcal{X}$ is an $(n+1)$-fold Segal space and $x$ and $y$
  are two objects of $\mathcal{X}$. Then the mapping space
  $(L_{n+1}\mathcal{X})(x, y)$ in the completion of $\mathcal{X}$ is
  the completion $L_{n}(\mathcal{X}(x,y))$ of the $n$-fold Segal space
  $\mathcal{X}(x,y)$ of maps from $x$ to $y$ in $\mathcal{X}$.
\end{lemma}
\begin{proof}
  The localization $L_{n+1} \colon \Seg_{n+1}(\mathcal{S}) \to
  \txt{CSS}_{n+1}(\mathcal{S})$ can be written as a composite
  \[ \Seg_{n+1}(\mathcal{S}) \xto{L_{n,*}}
  \Seg(\txt{CSS}_{n}(\mathcal{S})) \xto{\Lambda}
  \txt{CSS}_{n+1}(\mathcal{S}). \] By \rcite{LurieGoodwillie}{Theorem
    1.2.13}, the natural map $\mathcal{Y} \to \Lambda\mathcal{Y}$ is
  fully faithful and essentially surjective for all $\mathcal{Y} \in
  \Seg(\txt{CSS}_{n}(\mathcal{S}))$; in particular, we have a pullback
  square
  \nolabelcsquare{\mathcal{Y}_{1}}{\Lambda\mathcal{Y}_{1}}{\mathcal{Y}_{0}^{\times
      2}}{\Lambda\mathcal{Y}_{0}^{\times 2}.}  Applying this to
  $L_{n,*}\mathcal{X}$ we see that we have an equivalence
  \[(L_{n,*}\mathcal{X})(x,y) \isoto (\Lambda L_{n,*}\mathcal{X})(x,y)
  \simeq (L_{n+1}\mathcal{X})(x,y).\] The $n$-fold Segal space
  $(L_{n,*}\mathcal{X})(x,y)$ is defined by the pullback square
  \nolabelcsquare{(L_{n,*}\mathcal{X})(x,y)}{L_{n}\mathcal{X}_{1}}{*}{L_{n}\mathcal{X}_{0}^{\times
      2}.} But by \rcite{spans}{Lemma 2.21} the functor $L_{n}$
  preserves pullbacks over constant diagrams, so this fibre is
  equivalent to $L_{n}(\mathcal{X}(x,y))$, which completes the proof.
\end{proof}

\begin{proof}[Proof of Theorem~\ref{thm:Algnmap}]
  Let $U_{\Seg}^{n+1} \colon \Cat^{n+1}(\mathcal{S}) \to
  \Seg_{n+1}(\mathcal{S})$ denote the right adjoint to the inclusion,
  and let $i_{n} \colon \Cat^{n}(\CatI) \to \Cat^{n+1}(\mathcal{S})$
  denote the inclusion, which is also a right adjoint.  By
  Corollary~\ref{cor:ALGpullback} we then have a pullback square
  \nolabelcsquare{U^{n+1}_{\Seg}i_{n}\fALG_{n}(\Bimod_{A,B}(\mathcal{C}))}{U^{n+1}_{\Seg}i_{n}(\fALG_{n+1}(\mathcal{C})_{[1]})}{*}{U^{n+1}_{\Seg}i_{n}(\fALG_{n+1}(\mathcal{C})_{[0]})^{\times
      2}} of $(n+1)$-fold Segal spaces. This factors through the
  pullback square
  \nolabelcsquare{(U^{n+2}_{\Seg}i_{n+1}\fALG_{n+1}(\mathcal{C}))_{[1]}}{U^{n+1}_{\Seg}i_{n}(\fALG_{n+1}(\mathcal{C})_{[1]})}{(U^{n+2}_{\Seg}i_{n+1}\fALG_{n+1}(\mathcal{C}))_{[0]}^{\times
      2}}{U^{n+1}_{\Seg}i_{n}(\fALG_{n+1}(\mathcal{C})_{[0]})^{\times 2},} and so
  we may identify
  $U^{n+1}_{\Seg}i_{n}\fALG_{n}(\Bimod_{A,B}(\mathcal{C}))$ with the
  $(n+1)$-fold Segal space of maps from $A$ to $B$ in the $(n+2)$-fold
  Segal space $U^{n+2}_{\Seg}i_{n+1}\fALG_{n+1}(\mathcal{C})$. By
  Lemma~\ref{lem:mapcompl} it follows that the completion
  \[\fAlg_{n}(\Bimod_{A,B}(\mathcal{C})) \simeq
  L_{n+1}U^{n+1}_{\Seg}i_{n}\fALG_{n}(\Bimod_{A,B}(\mathcal{C}))\] is
  equivalent to the mapping $(\infty,n+1)$-category
  \[\fAlg_{n+1}(\mathcal{C})(A, B) \simeq
  (L_{n+1}U^{n+2}_{\Seg}i_{n+1}\fALG_{n+1}(\mathcal{C}))(A, B).\qedhere\]
\end{proof}

\begin{cor}\label{cor:AlgnCII}
  Let $\mathcal{C}$ be a $\simp^{n+1}$-monoidal \icat{} with good
  relative tensor products, and write $I$ for
  the unit of the monoidal structure, regarded as a (trivial)
  $\mathbb{E}_{n+1}$-algebra in $\mathcal{C}$. Then we have an
  equivalence
  \[ \fAlg_{n+1}(\mathcal{C})(I, I) \simeq \fAlg_{n}(\mathcal{C}).\]
\end{cor}
\begin{proof}
  By Theorem~\ref{thm:Algnmap} there is an equivalence
  $\fAlg_{n+1}(\mathcal{C})(I, I) \simeq
  \fAlg_{n}(\Bimod_{I,I}(\mathcal{C}))$. But it follows from
  Corollary~\ref{cor:BimodIIeqC} and the definition of
  $\Bimod_{I,I}^{\otimes}(\mathcal{C})$ that the natural map
  $\Bimod_{I,I}^{\otimes}(\mathcal{C}) \to \mathcal{C}^{\otimes}$ is a
  $\simp^{n+1}$-monoidal equivalence.
\end{proof}

\begin{remark}
  Applying Corollary~\ref{cor:AlgnCII} inductively, we see that if
  $\mathcal{C}$ is an $\mathbb{E}_{n+m}$-monoidal \icat{}, then
  $\fAlg_{n}(\mathcal{C})$ is the endomorphism $(\infty,n+1)$-category
  of the identity $m$-morphism of the unit $I$ in the
  $(\infty,n+m+1)$-category $\fAlg_{n+m}(\mathcal{C})$. Thus
  $\fAlg_{n}(\mathcal{C})$ inherits an $\mathbb{E}_{m}$-monoidal
  structure (cf. \rcite{spans}{\S 5} for more details). It is
  intuitively plausible that this is the same as the
  $\mathbb{E}_{m}$-monoidal structure we constructed in
  \S\ref{subsec:ALGnmon}, but at the moment we are unable to prove
  this.
\end{remark}

\appendix

\section{Higher Algebra over $\Dn$}\label{sec:dnop}
In this section we discuss the more technical results we need about
\Dniopds{}. Many of these are slight variants of results proved for
symmetric \iopds{} in \cite{HA}, with essentially the same proofs,
and when this is the case we have not included proofs here.  Much
of the material in this section is also a special case either of
results of \cite{BarwickOpCat} or of unpublished work of Barwick and
Schommer-Pries.

\subsection{The $\infty$-Category of $\Dn$-$\infty$-Operads}\label{subsec:dniopdcat}
It is clear from the definition of morphisms of (generalized)
\Dniopds{} that the \icat{} of these objects should be regarded as a
subcategory of the slice \icat{} $(\CatI)_{/\Dnop}$. In this
subsection we will define model categories that describe the \icats{}
of \Dniopds{} and \gDniopds{}, using Lurie's theory of
\emph{categorical patterns}, which is a machine for constructing nice
model structures for certain subcategories of such slice \icats{}. We
will use these model structures to give an explicit model for a key
\icatl{} colimit of \gDniopds{} in \S\ref{subsec:alg1segcond} and
\S\ref{subsec:Enbimod}. We begin by recalling the definition of a
categorical pattern and Lurie's main results concerning them:

\begin{defn}
  A categorical pattern $\mathfrak{P} = (\mathcal{C}, S,
  \{p_{\alpha}\})$ consists of 
  \begin{itemize}
  \item an \icat{} $\mathcal{C}$,
  \item a marking of $\mathcal{C}$, i.e. a collection $S$ of
    1-simplices in $\mathcal{C}$ that includes all the degenerate
    ones,
  \item a collection of diagrams of \icats{} $p_{\alpha} \colon
    K_{\alpha}^{\triangleleft} \to \mathcal{C}$ such that $p_{\alpha}$
    takes every edge in $K_{\alpha}^{\triangleleft}$ to a marked edge
    of $\mathcal{C}$.
  \end{itemize}
\end{defn}

\begin{remark}
  Lurie's definition of a categorical pattern in \rcite{HA}{\S B} is
  more general than this: in particular, he includes the data of a
  \emph{scaling} of the simplicial set $\mathcal{C}$, i.e. a
  collection $T$ of 2-simplices in $\mathcal{C}$ that includes all the
  degenerate ones. In all the examples we consider, however, the
  scaling consists of \emph{all} 2-simplices of the simplicial set
  $\mathcal{C}$. We restrict ourselves to this special case as it
  gives a clearer description of the $\mathfrak{P}$-fibrant objects,
  and also simplifies the notation.
\end{remark}

From a categorical pattern, Lurie constructs a model category that
encodes the \icat{} of $\mathfrak{P}$-fibrant objects, in the
following sense:
\begin{defn}
  Suppose $\mathfrak{P} = (\mathcal{C}, S, \{p_{\alpha}\})$ is a
  categorical pattern. A map of simplicial sets $X \to \mathcal{C}$ is
  \emph{$\mathfrak{P}$-fibrant} if the following criteria are satisfied:
  \begin{enumerate}[(1)]
  \item The underlying map $\pi \colon Y \to \mathcal{C}$ is an inner
    fibration. (In particular, $Y$ is an \icat{}.)
  \item $Y$ has all $\pi$-coCartesian edges over the morphisms in $S$.
  \item For every $\alpha$, the coCartesian fibration $\pi_{\alpha}
    \colon Y \times_{\mathcal{C}} K_{\alpha}^{\triangleleft} \to
    K_{\alpha}^{\triangleleft}$, obtained by pulling back $\pi$ along
    $p_{\alpha}$, is classified by a limit diagram
    $K_{\alpha}^{\triangleleft} \to \CatI$.
  \item For every $\alpha$, any coCartesian lift $s \colon
    K_{\alpha}^{\triangleleft} \to Y$ of $p_{\alpha}$ is a
    $\pi$-limit diagram.
  \end{enumerate}
\end{defn}

\begin{thm}[Lurie, \rcite{HA}{Theorem B.0.20}]\label{thm:catpatternmodstr}
  Let $\mathfrak{P} = (\mathcal{C}, S, \{p_{\alpha}\})$ be a
  categorical pattern, and let $\overline{\mathcal{C}}$ denote the
  marked simplicial set $(\mathcal{C}, S)$. There is a unique left proper
  combinatorial simplicial model structure on the category
  $(\sSet^{+})_{/\overline{\mathcal{C}}}$ such that:
  \begin{enumerate}[(1)]
  \item The cofibrations are the morphisms whose underlying maps of
    simplicial sets are monomorphisms. In particular, all objects are
    cofibrant.
  \item An object $(X, T) \to \overline{\mathcal{C}}$ is fibrant
    \IFF{} $X \to \mathcal{C}$ is $\mathfrak{P}$-fibrant and $T$ is
    precisely the collection of coCartesian morphisms over the
    morphisms in $S$.
  \end{enumerate}
  We denote the category $(\sSet^{+})_{/\overline{\mathcal{C}}}$
  equipped with this model structure by $(\sSet^{+})_{\mathfrak{P}}$.\qed
\end{thm}

\begin{defn}
  We will make use of the following categorical patterns:
  \begin{enumerate}[(i)]
    \item Let $\mathfrak{O}_{n}$ be the categorical
    pattern \[(\Dnop, I_{n}, \{p_{I} \colon
    K_{I}^{\triangleleft} \to \Dnop\}),\] where $I_{n}$
    is the set of inert morphisms in $\Dnop$ and, for $I \in \Dn$, we
    write $K_{I}$ for the set of inert
    morphisms $I \to C_{n}$ in $\Dnop$ and $p_{I}$ for the functor
    $K_{I}^{\triangleleft} \to \Dnop$ associated to the inclusion
    $K_{I} \hookrightarrow (\Dnop)_{I/}$. It is immediate from
    Definition~\ref{defn:Dniopd} that a map $Y \to \Dnop$ is
    $\mathfrak{O}_{n}$-fibrant precisely if it is a \Dniopd{}.
  \item Let $\mathfrak{M}_{n}$ denote the categorical
    pattern \[(\Dnop, A_{n}, \{p_{I}
    \colon K_{I}^{\triangleleft} \to \Dnop\}),\]
    where $A_{n}$ denotes the set of all morphisms in $\Dnop$. Then a map $Y
    \to \Dnop$ is $\mathfrak{M}_{n}$-fibrant precisely if $Y \to
    \Dnop$ is a $\Dn$-monoidal \icat{}.
  \item Let $\mathfrak{O}^{\txt{gen}}_{n}$ be the categorical
    pattern \[(\Dnop, I_{n}, \{
    (\CellnIop)^{\triangleleft} \to \Dnop\}).\]
    It is immediate from Definition~\ref{defn:gDniopd} that a map $Y
    \to \Dnop$ is $\mathfrak{O}^{\txt{gen}}_{n}$-fibrant \IFF{} $Y \to
    \Dnop$ is a \gDniopd{}.
  \item Let $\mathfrak{U}_{n}$ denote the categorical
    pattern \[(\Dnop, \mathrm{N}\Dnop_{1}, \{
    (\CellnIop)^{\triangleleft} \to \Dnop\}).\]
    Then a map $Y \to \Dnop$ is $\mathfrak{U}_{n}$-fibrant \IFF{} $Y
    \to \Dnop$ is a $\Dn$-uple \icat{}.
  \end{enumerate}
\end{defn}

\begin{defn}
  The \icat{} $\OpdIDn$ of \Dniopds{} is the \icat{} associated to the
  simplicial model category $(\sSet^{+})_{\mathfrak{O}_{n}}$, i.e. the
  coherent nerve of its simplicial subcategory of fibrant objects. Thus
  the objects of $\OpdIDn$ can be identified with
  \Dniopds{}. Moreover, since the maps between these in
  $(\sSet^{+})_{\mathfrak{O}_{n}}$ are precisely the maps that
  preserve inert morphisms, it is also easy to see that the space of
  maps from $\mathcal{O}$ to $\mathcal{P}$ in $\OpdIDn$ is equivalent
  to the subspace of $\Map_{\Dnop}(\mathcal{O}, \mathcal{P})$ given by
  the components corresponding to inert-morphism-preserving maps, as
  expected. This justifies calling $\OpdIDn$ the \emph{\icat{} of
    \Dniopds{}}.
\end{defn}

\begin{remark}
  This \icat{} of \Dniopds{} is a special case of the \icats{} of
  \iopds{} over an operator category constructed by Barwick in
  \rcite{BarwickOpCat}{Theorem 8.15}.
\end{remark}

\begin{defn}
  Similarly, applying Theorem~\ref{thm:catpatternmodstr} to the
  categorical patterns $\mathfrak{M}_{n}$,
  $\mathfrak{O}_{n}^{\txt{gen}}$, and $\mathfrak{U}_{n}$ gives
  simplicial model categories $(\sSet^{+})_{\mathfrak{M}_{n}}$,
  $(\sSet^{+})_{\mathfrak{O}^{\txt{gen}}_{n}}$, and
  $(\sSet^{+})_{\mathfrak{U}_{n}}$ whose fibrant objects are,
  respectively, $\Dn$-monoidal \icats{}, \gDniopds{}, and $\Dn$-uple
  \icats{}. We write $\MonI^{\Dn}$, $\OpdIDng$, and $\txt{Upl}^{\Dn}_{\infty}$ for the \icats{} associated to these
  simplicial model categories, and refer to them as the \emph{\icats{} of
    $\Dn$-monoidal \icats{}}, \emph{\gDniopds{}}, and \emph{$\Dn$-uple \icats{}}.
\end{defn}

\begin{defn}
  The morphisms in $\MonI^{\Dn}$ are the (strong) $\Dn$-monoidal
  functors between $\Dn$-monoidal \icats{}. We write
  $\MonI^{\Dn,\txt{lax}}$ for the \icat{} of $\Dn$-monoidal \icats{}
  and lax $\Dn$-monoidal functors, i.e. the full subcategory of
  $\OpdIDn$ spanned by the $\Dn$-monoidal \icats{}.
\end{defn}

We now show that taking Cartesian products gives left Quillen
bifunctors relating \Dniopds{} for varying $n$. This will allow us to
reduce the proofs of the technical results needed in \S\ref{sec:enalg} to the case
where $n = 1$. First we introduce some notation and recall a result of Lurie:
\begin{defn}
  Suppose $\mathfrak{P} = (\mathcal{C}, S,
  \{p_{\alpha} \colon K_{\alpha}^{\triangleleft} \to \mathcal{C}\})$
  and $\mathfrak{Q} = (\mathcal{D}, T,
  \{q_{\beta} \colon L_{\beta}^{\triangleleft} \to \mathcal{D}\})$ are categorical
  patterns. The \emph{product categorical pattern} $\mathfrak{P}
  \times \mathfrak{Q}$ is given by
  \[ (\mathcal{C} \times \mathcal{D}, S \times T, \{p_{\alpha} \times
  \{d\} : d \in \mathcal{D}\} \cup \{\{c\} \times q_{\beta} : c \in \mathcal{C}\}).\]
\end{defn}

\begin{propn}[Lurie, \rcite{HA}{Remark B.2.5}]\label{propn:catpattprod}
  Suppose $\mathfrak{P}$ and $\mathfrak{Q}$ are categorical
  patterns. The Cartesian product is a left Quillen bifunctor
  \[ (\sSet^{+})_{\mathfrak{P}} \times (\sSet^{+})_{\mathfrak{Q}} \to
  (\sSet^{+})_{\mathfrak{P} \times \mathfrak{Q}}.\]\qed
\end{propn}

\begin{defn}
  Let us say that a categorical pattern $\mathfrak{P} = (\mathcal{C},
  S, D)$ is \emph{objectwise} if the set of diagrams $D$ is of the
  form $\{p_{x} \colon K_{x}^{\triangleleft} \to \mathcal{C} : x \in
  \mathcal{C}\}$, where $p_{x}(-\infty) = x$. We say that
  $\mathfrak{P}$ is \emph{reduced} if moreover $K_{c}$ has an initial
  object for every $c$ in the image of $p_{x}|_{K_{x}}$ for any
  $x$. If $\mathfrak{P} = (\mathcal{C}, S, \{p_{x} \colon
  K_{x}^{\triangleleft} \to \mathcal{C})\})$ and $\mathfrak{Q} =
  (\mathcal{D}, T, \{q_{y}\colon L_{y}^{\triangleleft} \to
  \mathcal{C}\})$ are objectwise categorical patterns, we let
  $\mathfrak{P} \boxtimes \mathfrak{Q}$ be the objectwise categorical
  pattern
  \[ (\mathcal{C} \times \mathcal{D}, S \times T, \{(K_{x} \times
  L_{y})^{\triangleleft} \to K_{x}^{\triangleleft} \times
  L_{y}^{\triangleleft}  \xto{p_{x} \times q_{y}} \mathcal{C} \times \mathcal{D}
  : (x,y) \in \mathcal{C} \times \mathcal{D}\}). \]
\end{defn}

\begin{propn}\label{propn:objwisecatpatt}
  Suppose $\mathfrak{P}$ and $\mathfrak{Q}$ are objectwise reduced
  categorical patterns. Then the model category structures
  $(\sSet^{+})_{\mathfrak{P} \times \mathfrak{Q}}$ and
  $(\sSet^{+})_{\mathfrak{P} \boxtimes \mathfrak{Q}}$ on
  $(\sSet^{+})_{/(\mathcal{C} \times \mathcal{D}, S \times T)}$ are
  identical.
\end{propn}

For the proof we make use of the following obvious observation:
\begin{lemma}\label{lem:rkesquare}
  Suppose given a commutative square
  \csquare{\mathcal{C}_{0}}{\mathcal{C}_{1}}{\mathcal{C}_{2}}{\mathcal{C}}{i_{1}}{i_{2}}{j_{1}}{j_{2}}
  of \icats{} where all the maps are fully faithful, such that every
  object of $\mathcal{C}$ is contained in the essential image of either $\mathcal{C}_{1}$ or
  $\mathcal{C}_{2}$. If $\pi \colon \mathcal{X} \to \mathcal{Y}$ is an
  inner fibration of \icats{} and $\phi \colon \mathcal{C} \to
  \mathcal{X}$ is a functor, then $\phi$ is a $\pi$-right Kan
  extension of $\phi|_{\mathcal{C}_{0}}$ along $j_{1}i_{1} \simeq
  j_{2}i_{2}$ \IFF{} $\phi|_{\mathcal{C}_{1}}$ is a $\pi$-right Kan
  extension of $\phi|_{\mathcal{C}_{0}}$ along $j_{1}$ and
  $\phi|_{\mathcal{C}_{2}}$ is a $\pi$-right Kan extension of
  $\phi|_{\mathcal{C}_{0}}$ along $j_{2}$.\qed
\end{lemma}

\begin{proof}[Proof of Proposition~\ref{propn:objwisecatpatt}]
  By the uniqueness statement in Theorem~\ref{thm:catpatternmodstr} it
  is enough to check that the fibrant objects are the same in the two
  model structures. Supposing $Y \to \mathcal{C}$ is an inner
  fibration with all coCartesian morphisms over the morphisms in $S$,
  we are interested in the following conditions:
  \begin{enumerate}[(1)]
  \item For all $(x,y) \in \mathcal{C} \times \mathcal{D}$, the
    coCartesian fibration $(K_{x} \times L_{y})^{\triangleleft}
    \times_{\mathcal{C}} Y \to (K_{x} \times L_{y})^{\triangleleft}$
    is classified by a limit diagram.
  \item[(1')] For all $(x,y) \in \mathcal{C} \times \mathcal{D}$, the
    coCartesian fibrations $(K_{x}^{\triangleleft} \times \{y\})
    \times_{\mathcal{C}} Y \to K_{x}^{\triangleleft}$ and
    $(\{x\} \times L_{y}^{\triangleleft}) \times_{\mathcal{C}} Y \to
    L_{y}^{\triangleleft}$ are classified by limit diagrams.
  \item For all $(x,y) \in \mathcal{C} \times \mathcal{D}$, any
    coCartesian section $s \colon (K_{x} \times L_{y})^{\triangleleft}
    \to Y$ is a $\pi$-limit diagram.
  \item[(2')] For all $(x,y) \in \mathcal{C} \times \mathcal{D}$, any
    coCartesian sections $s \colon K_{x}^{\triangleleft} \times \{y\}
    \to Y$ and $t \colon \{x\} \times L_{y}^{\triangleleft} \to Y$ are
    $\pi$-limit diagrams.
  \end{enumerate}
  We must show that (1) and (1') are equivalent, and that (2) and (2')
  are equivalent.

  To see that (1) implies (1'), let $\phi \colon K_{x}^{\triangleleft}
  \times L_{y}^{\triangleleft} \to \CatI$ be a diagram classified by
  the coCartesian fibration $(K_{x}^{\triangleleft} \times
  L_{y}^{\triangleleft}) \times_{\mathcal{C} \times \mathcal{D}} Y \to
  K_{x}^{\triangleleft} \times L_{y}^{\triangleleft}$ for some $(x,y)
  \in \mathcal{C} \times \mathcal{D}$. We now wish to apply
  Lemma~\ref{lem:rkesquare} to the square \nolabelcsquare{K_{x} \times
    L_{y}}{(K_{x} \times
    L_{y})^{\triangleleft}}{(K_{x}^{\triangleleft} \times L_{y})
    \amalg_{K_{x} \times L_{y}} (K_{x} \times
    L_{y}^{\triangleleft})}{K_{x}^{\triangleleft} \times
    L_{y}^{\triangleleft}.}  By assumption $\phi|_{(K_{x} \times
    L_{y})^{\triangleleft}}$ is a right Kan extension of $\phi|_{K_{x}
    \times L_{y}}$, so it remains to prove that
  the restriction of $\phi$ to $(K_{x}^{\triangleleft} \times L_{y}) \amalg_{K_{x} \times
      L_{y}} (K_{x} \times L_{y}^{\triangleleft})$ is a right Kan
  extension of $\phi|_{K_{x} \times L_{y}}$. In other words, we must
  show that for any $z \in L_{y}$ the object $\phi(-\infty, z)$ is a
  limit of $\phi|_{(K_{x} \times L_{y})_{(-\infty,z)/}}$, and that for
  any $w \in K_{x}$ the object $\phi(w, -\infty)$ is a limit of
  $\phi|_{(K_{x} \times L_{y})_{(w,-\infty)/}}$. The inclusion $K_{x}
  \times \{z\} \to (K_{x} \times L_{y})_{(-\infty,z)/}$ is coinitial,
  so it suffices to prove that the restriction of $\phi$ to
  $K_{x}^{\triangleleft} \times \{z\}$ is a limit diagram. Since the
  categorical pattern is reduced, by assumption the \icat{} $L_{z}$
  has an initial object, and so there is a coinitial map $K_{x} \times
  \{z\} \to K_{x} \times L_{z}$. Moreover, the restriction of $\phi$
  to $K_{x}^{\triangleleft} \times \{z\}$ is also the restriction of
  the analogous functor $(K_{z} \times L_{z})^{\triangleleft} \to
  \CatI$, which is a limit diagram by assumption. Thus
  $\phi(-\infty,z)$ is indeed the limit of $\phi|_{K_{x} \times
    \{z\}}$, and similarly $\phi(w,-\infty)$ is the limit of
  $\phi|_{\{w\} \times L_{y}}$. It then follows from
  Lemma~\ref{lem:rkesquare} that $\phi$ is a right Kan extension of
  $\phi|_{K_{x} \times L_{y}}$.

  Now considering the factorization
  \[ K_{x} \times L_{y} \hookrightarrow K_{x}^{\triangleleft} \times
L_{y} \hookrightarrow (K_{x}^{\triangleleft} \times
L_{y})^{\triangleleft} \to K_{x}^{\triangleleft} \times
L_{y}^{\triangleleft}\] we see that $\phi|_{(K_{x}^{\triangleleft}
  \times L_{y})^{\triangleleft}}$ is a limit of
$\phi|_{K_{x}^{\triangleleft} \times L_{y}}$. Since the inclusion
$\{-\infty\} \times L_{y} \hookrightarrow K_{x}^{\triangleleft} \times
L_{y}$ is coinitial, it follows that $\phi|_{\{-\infty\}
  \times L_{y}^{\triangleleft}}$ is a limit diagram. Similarly,
$\phi|_{K_{x}^{\triangleleft} \times \{-\infty\}}$ is a limit diagram,
which proves (1').
  
  Conversely, to see that (1') implies (1) we consider the square
  \nolabelcsquare{K_{x} \times L_{y}}{(K_{x}^{\triangleleft} \times
    L_{y}) \amalg_{K_{x} \times L_{y}} (K_{x} \times
    L_{y}^{\triangleleft})}{(K_{x}^{\triangleleft} \times
    L_{y})^{\triangleleft}}{K_{x}^{\triangleleft} \times
    L_{y}^{\triangleleft}.}  Let $\phi$ be as above; then it follows
  from (1') that $\phi|_{(K_{x}^{\triangleleft} \times L_{y})
    \amalg_{K_{x} \times L_{y}} (K_{x} \times L_{y}^{\triangleleft})}$
  is a right Kan extension of $\phi|_{K_{x} \times L_{y}}$ and
  $\phi|_{K_{x}^{\triangleleft} \times L_{y}}$ is a right Kan
  extension of $\phi|_{K_{x} \times L_{y}}$. Since $\{-\infty\} \times
  L_{y} \hookrightarrow K_{x}^{\triangleleft} \times L_{y}$ is
  coinitial, (1') also implies that $\phi|_{(K_{x}^{\triangleleft}
    \times L_{y})^{\triangleleft}}$ is a right Kan extension of
  $\phi|_{K_{x}^{\triangleleft} \times L_{y}}$, and so by
  Lemma~\ref{lem:rkesquare} it follows that $\phi$ is a right Kan
  extension of $\phi|_{K_{x} \times L_{y}}$. But then $\phi|_{(K_{x}
    \times L_{y})^{\triangleleft}}$ is also a right Kan extension of
  $\phi|_{K_{x} \times L_{y}}$, which proves (1).

  It follows by the same argument, applied to a coCartesian section
  $\phi \colon K_{x}^{\triangleleft} \times L_{y}^{\triangleleft} \to
  Y$, that (2) is equivalent to (2').
\end{proof}

Applying this to the categorical patterns we're interested in, we get:
\begin{cor}\label{cor:Ontimes}\ 
  \begin{enumerate}[(i)]
  \item The model categories $(\sSet^{+})_{\mathfrak{O}_{n} \times
      \mathfrak{O}_{m}}$ and $(\sSet^{+})_{\mathfrak{O}_{n+m}}$ are identical.
  \item The model categories $(\sSet^{+})_{\mathfrak{M}_{n} \times
      \mathfrak{M}_{m}}$ and $(\sSet^{+})_{\mathfrak{M}_{n+m}}$ are
    identical.
  \item The model categories $(\sSet^{+})_{\mathfrak{O}^{\txt{gen}}_{n} \times
      \mathfrak{O}^{\txt{gen}}_{m}}$ and
    $(\sSet^{+})_{\mathfrak{O}^{\txt{gen}}_{n+m}}$ are identical.   
  \item The model categories $(\sSet^{+})_{\mathfrak{U}_{n} \times
      \mathfrak{U}_{m}}$ and $(\sSet^{+})_{\mathfrak{U}_{n+m}}$ are
    identical.
  \end{enumerate}
\end{cor}
\begin{proof}
  The categorical patterns $\mathfrak{O}_{n}$, $\mathfrak{M}_{n}$,
  $\mathfrak{O}_{n}^{\txt{gen}}$ and $\mathfrak{U}_{n}$ are all
  objectwise reduced, and we have identifications
  $\mathfrak{O}_{n+m} = \mathfrak{O}_{n} \boxtimes \mathfrak{O}_{m}$,
  $\mathfrak{M}_{n+m} = \mathfrak{M}_{n} \boxtimes \mathfrak{M}_{m}$,
  $\mathfrak{O}^{\txt{gen}}_{n+m} = \mathfrak{O}^{\txt{gen}}_{n}
  \boxtimes \mathfrak{O}^{\txt{gen}}_{m}$ and $\mathfrak{U}_{n+m} =
  \mathfrak{U}_{n} \boxtimes \mathfrak{U}_{m}$. The result is
  therefore immediate from Proposition~\ref{propn:objwisecatpatt}
\end{proof}

\begin{cor}\label{cor:CartProdQuillenIOpd}
  The Cartesian product defines left Quillen bifunctors
  \[ (\sSet^{+})_{\mathfrak{O}_{n}} \times
  (\sSet^{+})_{\mathfrak{O}_{m}} \to
  (\sSet^{+})_{\mathfrak{O}_{n+m}}, \]
\[ (\sSet^{+})_{\mathfrak{M}_{n}} \times
  (\sSet^{+})_{\mathfrak{M}_{m}} \to
  (\sSet^{+})_{\mathfrak{M}_{n+m}}, \]
\[ (\sSet^{+})_{\mathfrak{O}^{\txt{gen}}_{n}} \times
  (\sSet^{+})_{\mathfrak{O}^{\txt{gen}}_{m}} \to
  (\sSet^{+})_{\mathfrak{O}^{\txt{gen}}_{n+m}}, \]
\[ (\sSet^{+})_{\mathfrak{U}_{n}} \times
  (\sSet^{+})_{\mathfrak{U}_{m}} \to
  (\sSet^{+})_{\mathfrak{U}_{n+m}}. \]
\end{cor}
\begin{proof}
  Combine Corollary~\ref{cor:Ontimes} with
  Proposition~\ref{propn:catpattprod}.
\end{proof}

Finally, we recall a useful result on functoriality of categorical
pattern model structures:
\begin{defn}
  Suppose $\mathfrak{P} = (\mathcal{C}, S, \{p_{\alpha}\})$ and
  $\mathfrak{Q} = (\mathcal{D}, T, \{q_{\beta}\})$ are categorical
  patterns. A \emph{morphism of categorical patterns} $f \colon
  \mathfrak{P} \to \mathfrak{Q}$ is a functor $f \colon \mathcal{C}
  \to \mathcal{D}$ such that $f(S) \subseteq f(T)$ and $f \circ
  p_{\alpha}$ lies in $\{q_{\beta}\}$ for all $\alpha$.
\end{defn}

\begin{propn}[Lurie, \rcite{HA}{Proposition B.2.9}]\label{propn:catpattmor}
  Suppose $f \colon \mathfrak{P} \to \mathfrak{Q}$ is a morphism of
  categorical patterns. Then composition with $f$ gives a left Quillen
  functor
  \[ f_{!} \colon (\sSet^{+})_{\mathfrak{P}} \to
  (\sSet^{+})_{\mathfrak{Q}}.\qed\]
\end{propn}

\subsection{$\Dn$-$\infty$-Operads and Symmetric
  $\infty$-Operads} \label{subsec:symDn} 

In this subsection we will relate \Dniopds{} to the symmetric \iopds{}
studied in \cite{HA}. We first recall some definitions:
\begin{defn}
  For $n$ a non-negative integer, let $\angled{n}$ denote the set
  $\{0,1,\ldots,n\}$, regarded as a pointed set with base point $0$. A
  morphism $f \colon \angled{n} \to \angled{m}$ of finite pointed sets
  is \emph{inert} if $f^{-1}(i)$ has a single element for every $i
  \neq 0$, and \emph{active} if $f^{-1}(0) = \{0\}$. Recall that the
  inert and active morphisms form a factorization system on $\bbGamma^{\op}$.
\end{defn}

\begin{defn}
  A \emph{symmetric \iopd{}} is a functor of \icats{} $\pi \colon
  \mathcal{O} \to \bbGamma^{\op}$ such that:
  \begin{enumerate}[(i)]
  \item For every inert morphism $\phi \colon \angled{m} \to \angled{n}$  in
    $\bbGamma^{\op}$ and every $X \in \mathcal{O}_{\angled{n}}$ there exists
    a $\pi$-coCartesian morphism $X \to \phi_{!}X$ over $\phi$.
  \item Let $\rho_{i} \colon \angled{n} \to \angled{1}$, $i =
    1,\ldots,n$, denote the (inert) map that sends $i$ to $1$ and
    every other element of $\angled{n}$ to $0$. For every $\angled{n}
    \in \bbGamma^{\op}$ the functor
    \[\mathcal{O}_{\angled{n}} \to
    (\mathcal{O}_{\angled{1}})^{\times n}\] induced by the
    coCartesian arrows over the maps $\rho_{i}$ is an equivalence of \icats{}.
  \item For every morphism $\phi \colon \angled{n} \to \angled{m}$ in
    $\bbGamma^{\op}$ and $Y \in \mathcal{O}_{\angled{m}}$, composition
    with coCartesian morphisms $Y \to Y_{i}$ over the inert morphisms
    $\rho_{i}$ gives an equivalence
    \[ \Map_{\mathcal{O}}^{\phi}(X, Y) \isoto \prod_{i}
    \Map_{\mathcal{O}}^{\rho_{i} \circ \phi}(X, Y_{i}),\] where
    $\Map_{\mathcal{O}}^{\phi}(X,Y)$ denotes the subspace of
    $\Map_{\mathcal{O}}(X,Y)$ of morphisms that map to $\phi$ in
    $\simp^{\op}$.
  \end{enumerate}
\end{defn}

\begin{defn}
  Let $\mathfrak{O}_{\Sigma}$ denote the categorical pattern
  $(\bbGamma^{\op}, I_{\Sigma}, \{p_{\angled{n}} \colon
  P_{\angled{n}}^{\triangleleft} \to \bbGamma^{\op}\})$, where
  $\bbGamma^{\op}$ is the category of finite pointed sets,
  $I_{\Sigma}$ denotes the set of inert morphisms in $\bbGamma^{\op}$,
  and $P_{\angled{n}}$ is the set of inert morphisms $\angled{n} \to
  \angled{1}$ in $\bbGamma^{\op}$.
\end{defn}

\begin{defn}
  The $\mathfrak{O}_{\Sigma}$-fibrant objects are precisely the
  symmetric \iopds{}, and we write $\OpdIS$ for the \icat{} associated
  to the model category $(\sSet^{+})_{\mathfrak{O}_{\Sigma}}$.
\end{defn}

\begin{defn}
  Let $u^{1} \colon \simp^{\op} \to \bbGamma^{\op}$ be the functor defined
  as in \cite[Construction 4.1.2.5]{HA} (this is the same as the
  functor introduced by Segal in \cite{SegalCatCohlgy}). Recall that
  this sends $[n]$ to $\angled{n}$, and a map $\phi \colon [n] \to
  [m]$ in $\simp$ to the map $\angled{m} \to \angled{n}$ given by 
  \[ u^{1}(\phi)(i) =
  \begin{cases}
    j & \txt{if $\phi(j-1) < i \leq \phi(j)$,} \\
    0 & \txt{if no such $j$ exists.}
  \end{cases}
  \]
  This takes inert morphisms in $\simp^{\op}$ to inert morphisms in
  $\bbGamma^{\op}$, and moreover induces a morphism of categorical
  patterns from $\mathfrak{O}_{1}$ to $\mathfrak{O}_{\Sigma}$. Let
  $\mu \colon \bbGamma^{\op} \times \bbGamma^{\op} \to \bbGamma^{\op}$
  be the functor defined in \rcite{HA}{Notation 2.2.5.1}; this takes
  $(\angled{m}, \angled{n})$ to $\angled{mn}$ and takes a morphism $(f
  \colon \angled{m} \to \angled{m'}, g \colon \angled{n} \to
  \angled{n'})$ to the morphism $\mu(f, g)$ given by 
  \[ \mu(f,g)(an + b - n) =
  \begin{cases}
    0 & \txt{if $f(a) = 0$ or $g(b) = 0$}, \\
    f(a)n' + g(b) - n' & \txt{otherwise.}
  \end{cases}
  \]
  
  The functor $\mu$ induces a morphism of categorical patterns
  $\mathfrak{O}_{\Sigma} \times \mathfrak{O}_{\Sigma} \to
  \mathfrak{O}_{\Sigma}$. We then inductively define $u^{n}\colon
  \Dnop \to \bbGamma^{\op}$ to be the composite
  \[ \simp^{\op} \times \simp^{n-1,\op} \xto{u^{1} \times
    u^{n-1}} \bbGamma^{\op} \times \bbGamma^{\op} \xto{\mu}
  \bbGamma^{\op},\] so that $u^{n}$ is a morphism of categorical patterns
  $\mathfrak{O}_{n} \to \mathfrak{O}_{\Sigma}$ for all $n$. Thus
  $u^{n}$ induces adjoint functors \[ u^{n}_{!} : \OpdIDn
  \rightleftarrows \OpdIS : u^{n,*}.\] Moreover, since the induced
  Quillen functors are enriched in marked simplicial sets we get
  equivalences
  \[ \Algn_{\mathcal{O}}(u^{n,*}\mathcal{P}) \simeq
  \AlgS_{u^{n}_{!}\mathcal{O}} (\mathcal{P}),\] where
  $\mathcal{O}$ is a \Dniopd{} and $\mathcal{P}$ is a symmetric
  \iopd{}.
\end{defn}

\begin{remark}
  The Quillen adjunction $u^{n}_{!} \dashv u^{n,*}$ is a special case
  of the adjunctions arising from morphisms of operator categories
  that are discussed in \rcite{BarwickOpCat}{Proposition 8.18}.
\end{remark}

By Corollary~\ref{cor:CartProdQuillenIOpd} and Proposition
\ref{propn:catpattmor} we then have a commutative diagram of left
Quillen functors
\[
\ltikzcd{
  (\sSet^{+})_{\mathfrak{O}_{1}} \times (\sSet^{+})_{\mathfrak{O}_{n}}
  \arrow{r}{\times} \arrow{d}{u^{1}_{!} \times u^{n}_{!}} \pgfmatrixnextcell (\sSet^{+})_{\mathfrak{O}_{1} \times \mathfrak{O}_{n}} \arrow{r}{\cong}
  \arrow{d}{(u^{1} \times u^{n})_{!}} \pgfmatrixnextcell
  (\sSet^{+})_{\mathfrak{O}_{n+1}} \arrow{d}{u^{n+1}_{!}} \\
  (\sSet^{+})_{\mathfrak{O}_{\Sigma}} \times
  (\sSet^{+})_{\mathfrak{O}_{\Sigma}} \arrow{r}{\times} \pgfmatrixnextcell
  (\sSet^{+})_{\mathfrak{O}_{\Sigma} \times \mathfrak{O}_{\Sigma}} \arrow{r}{\mu_{!}}\pgfmatrixnextcell (\sSet^{+})_{\mathfrak{O}_{\Sigma}},
  }
\]
where the left horizontal functors are given by the Cartesian
products. The Boardman-Vogt tensor product of symmetric \iopds{}, as
defined in \rcite{HA}{\S 2.2.5}, is the functor of \icats{} induced by
the composite functor along the bottom of this diagram. On the level
of \icats{} we have therefore proved the following:
\begin{propn}\label{propn:unandBVtens}
  There is a commutative diagram
  \[
  \ltikzcd{
    \OpdI^{\simp} \times \OpdIDn \arrow{r}{\times} \arrow{d}{u_{!}^{1}
      \times u_{!}^{n}} \pgfmatrixnextcell
    \OpdI^{\simp^{n+1}} \arrow{d}{u^{n+1}_{!}} \\
    \OpdIS \times \OpdIS \arrow{r}{\otimes} \pgfmatrixnextcell \OpdIS
    }
\]
\end{propn}

Invoking the Dunn-Lurie Additivity Theorem, we get:
\begin{cor}\label{cor:unDnisEn}
  The symmetric \iopd{} $u^{n}_{!}(\Dnop)$ is equivalent to $\mathbb{E}_{n}$.
\end{cor}
\begin{proof}
  Applying Proposition~\ref{propn:unandBVtens} we have an equivalence
  \[u^{n}_{!}(\Dnop) \simeq u^{1}_{!}(\simp^{\op}) \otimes
  u^{n-1}_{!}(\simp^{n-1,\op}).\] By
  \rcite{HA}{Proposition 4.1.2.10} and \rcite{HA}{Example 5.1.0.7},
  the symmetric \iopd{} $u^{1}_{!}(\simp^{\op})$ is equivalent to
  $\mathbb{E}_{1}$, so by induction we have an equivalence
  $u^{n}_{!}(\Dnop) \simeq \mathbb{E}_{1}^{\otimes n}$. Now
  \rcite{HA}{Theorem 5.1.2.2} says that the symmetric \iopd{}
  $\mathbb{E}_{1}^{\otimes n}$ is equivalent to $\mathbb{E}_{n}$,
  which completes the proof.
\end{proof}

\begin{cor}\label{cor:EnDnalgeq}
  Let $\mathcal{O}$ be a symmetric \iopd{}. Then there is a natural
  equivalence
    \[ \AlgS_{\mathbb{E}_{n}}(\mathcal{O}) \simeq
    \Algn_{\Dnop}(u^{n,*}\mathcal{O}).\]
\end{cor}

\subsection{$\Dn$-Monoid Objects}\label{subsec:dnmonoid}
We will now observe that $\Dn$-algebras in a Cartesian monoidal
\icat{} are equivalent to the $\Dn$-monoids we discussed above in
\S\ref{sec:cartenalg}. More generally, we can define
$\mathcal{O}$-monoids for any \gDniopd{} $\mathcal{O}$ as an
equivalent way of describing $\mathcal{O}$-algebras in a Cartesian
monoidal \icat{}:
\begin{defn}
  Let $\mathcal{C}$ be an \icat{} with finite products and
  $\mathcal{O}$ a \gDniopd{}. An \emph{$\mathcal{O}$-monoid} in
  $\mathcal{C}$ is a functor $F \colon \mathcal{O} \to \mathcal{C}$
  such that for every $I \in \Dnop$ and $X \in \mathcal{O}_{I}$, the
  map $F(X) \to \prod_{i \in |I|}F(X_{i})$, induced by the coCartesian
  morphisms $X \to X_{i}$ over $i$, is an equivalence.  We write
  $\Mon^{n}_{\mathcal{O}}(\mathcal{C})$ for the full subcategory of
  $\Fun(\mathcal{O}, \mathcal{C})$ spanned by the
  $\mathcal{O}$-monoids.
\end{defn}

\begin{propn}\label{propn:Dnmonalgeq}
  Suppose $\mathcal{C}$ be an \icat{} with finite products, and let
  $\mathcal{C}^{\times}$ denote the Cartesian symmetric monoidal
  structure on $\mathcal{C}$ constructed in \rcite{HA}{\S 2.4.1}. If
  $\mathcal{O}$ a (generalized) \Dniopd{}, then there is a natural
  equivalence $\Mon^{n}_{\mathcal{O}}(\mathcal{C}) \simeq
  \Alg^{n}_{\mathcal{O}}(u^{n,*}\mathcal{C}^{\times})$.
\end{propn}
\begin{proof}
  As \rcite{HA}{Proposition 2.4.2.5}.
\end{proof}

\begin{cor}\label{cor:symDnmoneq}
  Let $\mathcal{C}$ be an \icat{} with finite products. Then there is
    a natural equivalence
    \[ \MonS_{\mathbb{E}_{n}}(\mathcal{C}) \simeq
    \Mon^{n}_{\Dnop}(\mathcal{C}).\]
\end{cor}
\begin{proof}
  Combine Corollary~\ref{cor:EnDnalgeq} with
  Proposition~\ref{propn:Dnmonalgeq} and \rcite{HA}{Proposition
    2.4.2.5} --- this gives a natural equivalence
  \[ \MonS_{\En}(\mathcal{C}) \simeq \AlgS_{\En}(\mathcal{C}^{\times})
  \simeq \Alg^{n}_{\Dnop}(u^{n,*}\mathcal{C}^{\times}) \simeq
  \Mon^{n}_{\Dnop}(\mathcal{C}).\qedhere\]
\end{proof}

\begin{cor}\label{cor:DnmonisEnmon}
  The \icat{} $\MonI^{\Dn}$ of $\Dn$-monoidal \icats{} is equivalent
  to the \icat{} $\MonI^{\Sigma,\mathbb{E}_{n}}$ of
  $\mathbb{E}_{n}$-monoidal \icats{}. 
\end{cor}
\begin{proof}
  This is just the special case of Corollary~\ref{cor:symDnmoneq}
  where $\mathcal{C} = \CatI$.
\end{proof}

\subsection{Weak Operadic Colimits}
Suppose $\mathcal{C}$ is a $\Dn$-monoidal \icat{}. Then by
Corollary~\ref{cor:DnmonisEnmon} $\mathcal{C}$ is equivalently an
$\En$-monoidal \icat{}. Moreover, if $\mathcal{O}$ is a \dniopd{},
then $\mathcal{O}$-algebras in $\mathcal{C}$, regarded as
$\Dn$-monoidal, are equivalent to $u^{n}_{!}\mathcal{O}$-algebras in
$\mathcal{C}$, regarded as $\En$-monoidal, by the results of
\S\ref{subsec:symDn}. If $f \colon \mathcal{O} \to \mathcal{P}$ is a
morphism of \dniopds{}, this means that we can apply the results of
\rcite{HA}{\S 3.1.3} to $u^{n}_{!}f$ to conclude that if $\mathcal{C}$
is well-behaved, then the functor $f^{*} \colon
\Algn_{\mathcal{P}}(\mathcal{C}) \to \Algn_{\mathcal{O}}(\mathcal{C})$
has a left adjoint $f_{!}$.

In \rcite{HA}{\S 3.1.3}, such left adjoint functors are constructed by
forming certain concrete colimit diagrams. However, as we do not have
any explicit understanding of the symmetric \iopd{}
$u^{n}_{!}\mathcal{O}$, the results of \cite{HA} do not allow us to
understand what the functor $f_{!}$ does for $f$ a morphism of
\dniopds{}. For the results of \S\ref{sec:algbimod} and
\S\ref{sec:enalg} this is insufficient --- in fact, we need an
explicit description of such a left adjoint for certain maps of
\emph{generalized} \dniopds{}, which introduces another inexplicit
construction, namely the localization functor from \gdniopds{} to
\dniopds{}, before we can apply the results from \cite{HA}. For this
reason, we will in the next couple of subsections discuss analogues of
many of the results in \rcite{HA}{\S 3.1.1--3.1.3} in the setting of
\gdniopds{}. Luckily, these results can generally be obtained by minor
variations of the arguments from \cite{HA}, and when this is the case
we have not included complete details.

In this section we consider the analogue, in the setting of
\dniopds{}, of the \emph{weak operadic colimits} introduced in
\rcite{HA}{\S 3.1.1}. However, unlike in \rcite{HA}{\S 3.1.1}, we will
not consider \emph{relative} weak operadic colimits, as these are not
needed in this paper.

\begin{remark}
  In \rcite{HA}{\S 3.1.1} weak operadic colimits are considered as a
  preliminary to a notion of \emph{operadic colimits}. These do not
  have a straightforward analogue in the $\Dn$-context. Instead, we
  will introduce a notion we call a \emph{monoidal colimit}, which is
  an adequate substitute such that the required arguments from \cite{HA} still
  go through.
\end{remark}

\begin{notation}
  Suppose $\mathcal{O}$ is a \dniopd{}; we denote the subcategory of
  $\mathcal{O}$ containing only the active morphisms by
  $\mathcal{O}^{\act}$. If $p \colon K \to \mathcal{O}^{\act}$ is a
  diagram, we write $\mathcal{O}^{\act}_{C_{n},p/}$ for the \icat{}
  $\mathcal{O}_{C_{n}} \times_{\mathcal{O}^{\act}}
  \mathcal{O}^{\act}_{p/}$ --- thus an object of
  $\mathcal{O}^{\act}_{C_{n},p/}$ consists of a cone
  $K^{\triangleright} \to \mathcal{O}^{\act}$ that restricts to $p$ on
  $K$ and with the image of the cone point in the fibre over $C_{n}$.
\end{notation}

\begin{defn}
  Suppose $\mathcal{O}$ is a \dniopd{}, let $\overline{p} \colon
  K^{\triangleright} \to \mathcal{O}^{\act}$ be a diagram, and set $p
  := \overline{p}|_{K}$. We say that $\overline{p}$ is a \emph{weak
    operadic colimit diagram} if the evident forgetful map
  \[ \mathcal{O}^{\act}_{C_{n},\overline{p}/} \to
  \mathcal{O}^{\act}_{C_{n},p/}\]
  is an equivalence of \icats{}.
\end{defn}

\begin{remark}
  If the image of the cone point of $K^{\triangleright}$ under
  $\overline{p}$ lies over $C_{n}$, then $\overline{p}$ is itself an
  object of $\mathcal{O}^{\act}_{C_{n},p/}$, and so $\overline{p}$ is
  a weak operadic colimit diagram \IFF{} it is a final object of
  $\mathcal{O}^{\act}_{C_{n},p/}$.
\end{remark}

\begin{remark}
  It follows from \rcite{HTT}{Proposition 2.1.2.1} that the map
  $\mathcal{O}^{\act}_{C_{n},\overline{p}/} \to
  \mathcal{O}^{\act}_{C_{n},p/}$ is always a left
  fibration. By \rcite{HTT}{Proposition 2.4.4.6} it is therefore an
  equivalence of \icats{} \IFF{} it is a trivial Kan fibration.
\end{remark}

\begin{remark}
  Suppose $K^{\triangleright} \to
  \mathcal{O}^{\act}$ is a weak operadic colimit diagram and $L \to K$
  is a cofinal map. Then the composite $L^{\triangleright} \to
  K^{\triangleright} \to \mathcal{O}^{\act}$ is also a weak operadic
  colimit diagram.
\end{remark}

\begin{propn}
  Let $\mathcal{O}$ be a \dniopd{}. A diagram $\overline{p} \colon
  K^{\triangleright} \to \mathcal{O}^{\act}$ is a weak operadic
  colimit \IFF{} for every $n > 0$ and every diagram \liftcsquare{K
    \star \partial \Delta^{n}}{\mathcal{O}^{\act}}{K \star
    \Delta^{n}}{(\Dnop)^{\act}}{\overline{f}_{0}}{}{}{f}{\overline{f}}
  there exists an extension $\overline{f}$ of $\overline{f}_{0}$.
\end{propn}
\begin{proof}
  As \rcite{HA}{Proposition 3.1.1.7}.
\end{proof}

\begin{propn}
  Let $\mathcal{O}$ be a \dniopd{}, let $\overline{h} \colon
  \Delta^{1} \times K^{\triangleright} \to \mathcal{O}^{\act}$ be a
  natural transformation from $\overline{h}_{0} :=
  \overline{h}|_{\{0\} \times K^{\triangleright}}$ to $\overline{h}_{1} :=
  \overline{h}|_{\{1\} \times K^{\triangleright}}$. Suppose that:
  \begin{enumerate}[(1)]
  \item For every vertex $x \in K^{\triangleright}$, the restriction
    $\overline{h}|_{\Delta^{1} \times \{x\}}$ is a coCartesian edge of
    $\mathcal{O}$.
  \item The composite $\Delta^{1} \times \{\infty\} \to \mathcal{O}
    \to \Dnop$ is an identity morphism. (Equivalently, the restriction
    $\overline{h}|_{\Delta^{1} \times \{\infty\}}$ is an equivalence in
    $\mathcal{O}$.)
  \end{enumerate}
  Then $\overline{h}_{0}$ is a weak operadic colimit diagram \IFF{}
  $\overline{h}_{1}$ is a weak operadic colimit diagram.
\end{propn}
\begin{proof}
  As \rcite{HA}{Proposition 3.1.1.15}.
\end{proof}

Applied to a $\Dn$-monoidal \icat{} $\mathcal{C}^{\otimes}$, this lets
us reduce the question of whether a diagram in
$(\mathcal{C}^{\otimes})^{\act}$ is a weak operadic colimit diagram to
whether a diagram in a fibre $\mathcal{C}^{\otimes}_{I}$ is a weak
operadic colimit diagram:
\begin{cor}\label{cor:wocpushforward}
  Let $\mathcal{C}^{\otimes}$ be $\Dn$-monoidal \icat{}, and suppose
  $\overline{p} \colon K^{\triangleright} \to
  (\mathcal{C}^{\otimes})^{\act}$ is a diagram lying over
  $\overline{q} \colon K^{\triangleright} \to \Dnop$. Take
  $\overline{p}'$ to be the coCartesian pushforward to the fibre over
  $\overline{q}(\infty)$. Then $\overline{p}$ is a weak operadic
  colimit diagram \IFF{} $\overline{p}'$ is a weak operadic colimit
  diagram.
\end{cor}

\begin{propn}\label{propn:wocmonfib}
  Let $\mathcal{C}^{\otimes}$ be a $\Dn$-monoidal \icat{}, and let
  $\overline{p} \colon K^{\triangleright} \to
  \mathcal{C}^{\otimes}_{I}$ be a diagram in the fibre over some $I
  \in \Dnop$. Then $\overline{p}$ is a weak operadic colimit diagram
  \IFF{} for $m \colon I \to C_{n}$ the unique active map in $\Dnop$,
  the composite
  \[ K^{\triangleright} \xto{\overline{p}} \mathcal{C}^{\otimes}_{I}
  \xto{m_{!}} \mathcal{C}\]
  is a colimit diagram in $\mathcal{C}$.
\end{propn}
\begin{proof}
  As \rcite{HA}{Proposition 3.1.1.16}.
\end{proof}

\begin{defn}
  Let $\mu^{j}$ denote the map $(\id,\ldots,d_{1},\ldots,\id) \colon
  ([1], \ldots, [2],\ldots,[1]) \to ([1],\ldots,[1])$ in $\Dnop$ (with
  $d_{1}$ in the $j$th place), and let $\mathcal{C}^{\otimes}$ be a
  $\Dn$-monoidal \icat{}. We say a diagram $\overline{p} \colon
  K^{\triangleright} \to \mathcal{C}$ is a \emph{monoidal colimit
    diagram} if for every $x \in \mathcal{C}$ and every $j =
  1,\ldots,n$, the composite
  \[ K^{\triangleright} \times \{x\} \to \mathcal{C} \times
  \mathcal{C} \simeq
  \mathcal{C}^{\otimes}_{([1],\ldots,[2],\ldots,[1])}
  \xto{\mu^{j}_{!}} \mathcal{C}\] is a colimit diagram. More
  generally, if $\overline{p} \colon K^{\triangleright} \to
  (\mathcal{C}^{\otimes})^{\act}$ is a diagram with
  $\overline{p}(\infty)$ in $\mathcal{C}^{\otimes}_{C_{n}}$, then we
  say that $\overline{p}$ is a \emph{monoidal colimit diagram} if the
  coCartesian pushforward to a diagram $\overline{p}' \colon
  K^{\triangleright} \to \mathcal{C}^{\otimes}_{C_{n}}$ is a monoidal
  colimit diagram in the first sense.
\end{defn}

\begin{propn}
  Let $\mathcal{O}$ be a \dniopd{}. Suppose given, for some $I \in
  \Dnop$, a finite collection of simplicial sets $K_{i}$, $i \in |I|$,
  and diagrams $\overline{p}_{i} \colon K_{i}^{\triangleright} \to
  \mathcal{O}_{C_{n}}$. Suppose the product diagram
  \[ \prod_{i\in|I|} K_{i}^{\triangleright} \to \prod_{i\in|I|}
  \mathcal{O}_{C_{n}} \simeq \mathcal{O}_{I}\]
  is such that for every $i$ and every choice of $k_{j} \in
  K_{j}^{\triangleright}$ for all $j \neq i$ the diagram
  \[ K_{i}^{\triangleright} \isoto \{k_{1}\} \times \cdots \times
  K_{i}^{\triangleright} \times \cdots \times \{k_{n}\} \to
  \mathcal{O}_{I} \hookrightarrow \mathcal{O}^{\act}\]
  is a weak operadic colimit diagram. Then the composite
  \[ (\prod_{i=1}^{n} K_{i})^{\triangleright} \to \prod_{i=1}^{n}
  K_{i}^{\triangleright} \to \mathcal{O}_{I} \hookrightarrow \mathcal{O}^{\act}\]
  is also a weak operadic colimit diagram.
\end{propn}
\begin{proof}
  As \rcite{HA}{Proposition 3.1.1.8}.
\end{proof}

\begin{cor}\label{cor:moncolimprod}
  Let $\mathcal{C}^{\otimes}$ be a $\Dn$-monoidal \icat{}. Suppose
  given, for some $I \in \Dnop$, a finite collection of simplicial
  sets $K_{i}$, $i \in |I|$, and monoidal colimit diagrams
  $\overline{p}_{i} \colon K_{i}^{\triangleright} \to
  \mathcal{C}$. Then the composite diagram
  \[ (\prod_{i=1}^{n} K_{i})^{\triangleright} \to \prod_{i=1}^{n}
  K_{i}^{\triangleright} \to \prod_{i \in
    |I|}\mathcal{C} \simeq \mathcal{C}^{\otimes}_{I} \hookrightarrow (\mathcal{C}^{\otimes})^{\act}\]
  is a weak operadic colimit diagram. Moreover, the coCartesian
  pushforward of this diagram to $\mathcal{C}^{\otimes}_{C_{n}}$ is a
  monoidal colimit diagram.\qed
\end{cor}

\subsection{Operadic Left Kan Extensions}
In this section we introduce the notion of \emph{operadic left Kan
  extensions} in the $\Dn$-setting. We then use the results of the
previous section to give two key results: first, we will see that
operadic left Kan extensions have a lifting property that will allow
us to conclude, in the next section, that they can be used to
construct adjoints, and second we consider an existence result for
operadic left Kan extensions.

\begin{defn}
  If $\mathcal{C}$ is an \icat{}, a \emph{$\mathcal{C}$-family} of
  \gdniopds{} is a morphism $\mathcal{M} \to \Dnop \times \mathcal{C}$
  of \gdniopds{}.
\end{defn}

\begin{defn}
  Suppose $\mathcal{M} \to \Dnop \times \Delta^{1}$ is a
  $\Delta^{1}$-family of \gdniopds{} between $\mathcal{A} :=
  \mathcal{M}_{0}$ and $\mathcal{B} := \mathcal{M}_{1}$. If $\mathcal{O}$ is a \dniopd{},
  an algebra $\mathcal{M} \to \mathcal{O}$ is an \emph{operadic left Kan
    extension} if for every $X \in \mathcal{B}$, the diagram
  \[ (\mathcal{A}_{/X}^{\act})^{\triangleright} \to \mathcal{M} \to \mathcal{O} \]
  is a weak operadic colimit diagram.
\end{defn}

\begin{propn}
  For $n >1$, let $\mathcal{M} \to \Dnop \times \Delta^{n}$ be a
  $\Delta^{n}$-family of \gdniopds{}, and let $\mathcal{O}$ be a
  \Dniopd{}. Suppose we are given a commutative diagram of \gdniopds{}
  \liftcsquare{\mathcal{M} \times_{\Delta^{n}}
    \Lambda^{n}_{0}}{\mathcal{O}}{\mathcal{M}}{\Dnop}{\overline{f}_0}{}{}{}{\overline{f}}
  such that the restriction of $\overline{f}_{0}$ to $\mathcal{M}
  \times_{\Delta^{n}} \Delta^{\{0,1\}}$ is an operadic left Kan
  extension. Then there exists an extension $\overline{f}$ of
  $\overline{f}_{0}$.
\end{propn}
\begin{proof}
  As \rcite{HA}{Theorem 3.1.2.3(B)}. (Note that when
  \rcite{HA}{Proposition 3.1.1.7} is invoked in step (1) in the
  proof it is sufficient for the diagram to be a weak operadic colimit.)
\end{proof}

\begin{cor}\label{cor:lokeprop}
  Suppose $\mathcal{M} \to \Dnop \times \Delta^{1}$ is a
  $\Delta^{1}$-family of \gdniopds{} between $\mathcal{A} :=
  \mathcal{M}_{0}$ and $\mathcal{B} := \mathcal{M}_{1}$, and let
  $\mathcal{O}$ be a \dniopd{}. Suppose that $n > 0$ and that we are
  given a diagram \liftcsquare{(\mathcal{A} \times \Delta^{n})
    \amalg_{(\mathcal{A} \times \partial \Delta^{n})}(\mathcal{M}
    \times \partial \Delta^{n})}{\mathcal{O}}{\mathcal{M} \times
    \Delta^{n}}{\Dnop}{f_{0}}{}{}{}{f} of \gdniopds{}. If the
  restriction of $f_{0}$ to $\mathcal{M} \times \{0\}$ is an operadic
  left Kan extension, then there exists an extension $f$ of $f_{0}$
  that is a map of \gdniopds{}.
\end{cor}
\begin{proof}
  As \rcite{HA}{Lemma 3.1.3.16}.
\end{proof}

\begin{defn}\label{defn:extendable}
  We say a $\Delta^{1}$-family of \gdniopds{} $\mathcal{M} \to \Dnop
  \times \Delta^{1}$ is \emph{extendable} if for every object $B \in
  \mathcal{M}_{1}$, lying over $I \in \Dnop$, with inert projections
  $B \to B_{i}$ over $i \in |I|$, the map $\mathcal{M}_{0,/B}^{\act}
  \to \prod_{i \in |I|} \mathcal{M}_{0,/B_{i}}^{\act}$ is cofinal.
\end{defn}

\begin{propn}\label{propn:lokeexist}
  Let $\mathcal{M} \to \Dnop \times \Delta^{1}$ be an extendable
  $\Delta^{1}$-family of \gdniopds{} and let $\mathcal{C}^{\otimes}$
  be a $\Dn$-monoidal \icat{}. Suppose given a diagram
  \liftcsquare{\mathcal{M}_{0}}{\mathcal{C}^{\otimes}}{\mathcal{M}}{\Dnop}{f_{0}}{}{}{}{f}
  such that for every $x \in \mathcal{M}_{C_{n},1}$, the diagram 
  \[ \mathcal{M}^{\act}_{0,/x} \to \mathcal{M}_{0} \xto{f_{0}}
  \mathcal{C}^{\otimes} \] can be extended to a \emph{monoidal} colimit
  diagram lifting the map
  $(\mathcal{M}^{\act}_{0,/x})^{\triangleright} \to \mathcal{M} \to
  \Dnop$. Then there exists an extension $f \colon \mathcal{M} \to
  \mathcal{C}^{\otimes}$ of $f_{0}$ that is an operadic left Kan
  extension.
\end{propn}
\begin{proof}
  Essentially as \rcite{HA}{Theorem 3.1.2.3(A)}, with a slight
  difference in step (1): To extend the functor to the 0-simplices of
  $\mathcal{M}_{1}$ we use the monoidal colimits that exist by
  assumption. Then for the construction of the higher-dimensional
  simplices we need to show that the maps $\delta \colon
  (\mathcal{M}_{0,/B}^{\act})^{\triangleright} \to
  \mathcal{C}^{\otimes}$ are weak operadic colimits. If $B$ lies over
  $I \in \Dnop$, let $\delta' \colon
  (\mathcal{M}_{0,/B}^{\act})^{\triangleright} \to
  \mathcal{C}^{\otimes}_{I}$ denote the coCartesian pushforward along
  the active maps to I; by Corollary~\ref{cor:wocpushforward} it
  suffices to show that $\delta'$ is a weak operadic colimit. Choose
  coCartesian morphisms $B \to B_{i}$ over the inert maps $i \colon I
  \to C_{n}$. Then $\delta'$ factors as
  \[ (\mathcal{M}_{0,/B}^{\act})^{\triangleright} \to (\prod_{i}
  \mathcal{M}_{0,/B_{i}}^{\act})^{\triangleright} \to \prod_{i}
  (\mathcal{M}_{0,/B_{i}}^{\act})^{\triangleright} \xto{\prod \delta_{i}} \prod_{i}
  \mathcal{C} \simeq \mathcal{C}^{\otimes}_{I}.\]
  The map $\mathcal{M}_{0,/B}^{\act} \to \prod_{i}
  \mathcal{M}_{0,/B_{i}}^{\act}$ is cofinal since $\mathcal{M}$ is
  extendable, hence it suffices to show that the map from $(\prod_{i}
  \mathcal{M}_{0,/B_{i}}^{\act})^{\triangleright}$ is a weak operadic
  colimit diagram. This follows from Corollary~\ref{cor:moncolimprod},
  since the maps $\delta_{i}$ are monoidal colimit diagrams.
\end{proof}

\begin{defn}
  Suppose $\mathcal{C}^{\otimes}$ is a $\Dn$-monoidal
  \icat{}. If $\mathcal{K}$ is some class of
  simplicial sets we say that $\mathcal{C}^{\otimes}$ is \emph{compatible with
    $\mathcal{K}$-indexed colimits} if
  \begin{enumerate}[(1)]
  \item the underlying \icat{} $\mathcal{C}$ has $\mathcal{K}$-indexed
    colimits,
  \item for $j = 1,\ldots,n$, the functor $\mu^{j}_{!} \colon
    \mathcal{C} \times \mathcal{C} \to \mathcal{C}$ preserves
    $\mathcal{K}$-indexed colimits in each variable.
  \end{enumerate}
\end{defn}

\begin{cor}\label{cor:lokecompcolim}
  Let $\mathcal{M} \to \Dnop \times \Delta^{1}$ be an extendable
  $\Delta^{1}$-family of \gdniopds{} and let $\mathcal{C}^{\otimes}$
  be a $\Dn$-monoidal \icat{} that is compatible with
  $\mathcal{M}_{0,/x}^{\act}$-indexed colimits for all $x \in
  \mathcal{M}_{C_{n},1}$. Suppose given a diagram
  \liftcsquare{\mathcal{M}_{0}}{\mathcal{C}^{\otimes}}{\mathcal{M}}{\Dnop,}{f_{0}}{}{}{}{f}
  then there exists an extension $f \colon \mathcal{M} \to
  \mathcal{C}^{\otimes}$ of $f_{0}$ that is an operadic left Kan
  extension.\qed
\end{cor}

We end this subsection with the following observation, which will be
useful for recognizing operadic left Kan extension:
\begin{lemma}\label{lem:lokerecognize}
  Let $i \colon \mathcal{A} \to \mathcal{B}$ be a morphism of
  \gdniopds{}, let $\mathcal{C}^{\otimes}$ be a $\Dn$-monoidal
  \icat{}, and suppose given a $\mathcal{B}$-algebra $B$ in
  $\mathcal{C}$ and a morphism $A \to i^{*}B$ of
  $\mathcal{A}$-algebras. Choose a factorization of the induced map
  $\phi \colon \mathcal{A} \times \Delta^{1} \amalg_{\mathcal{A} \times \{1\}}
  \mathcal{B} \to \mathcal{C}^{\otimes}$ through a $\Delta^{1}$-family
  of \gdniopds{} $\mathcal{M}$. Then the following are equivalent:
  \begin{enumerate}[(1)]
  \item The map $\mathcal{M} \to \mathcal{C}^{\otimes}$ is an operadic
    left Kan extension.
  \item Choose a coCartesian pushforward
    \[\phi' \colon \mathcal{A}^{\txt{act}} \times \Delta^{1}
    \amalg_{\mathcal{A}^{\txt{act}} \times \{1\}}
    \mathcal{B}^{\txt{act}} \to \mathcal{C} \simeq
    \mathcal{C}^{\otimes}_{C_{n}}\] of the restriction of $\phi$ to
    the subcategories of active maps, along the unique active maps to
    $C_{n}$. Then $\phi'$ is a left Kan extension in the sense of
    \rcite{HTT}{Definition 4.3.3.2}.
  \end{enumerate}
\end{lemma}
\begin{proof}
  Immediate from the description of weak operadic colimits in
  $\Dn$-monoidal \icats{} from Corollary~\ref{cor:wocpushforward} and
  Proposition~\ref{propn:wocmonfib}.
\end{proof}

\subsection{Free Algebras}
\begin{defn}
  Suppose $i \colon \mathcal{O} \to \mathcal{P}$ is a morphism of
  \gdniopds{} and $\mathcal{C}^{\otimes}$ is a $\Dn$-monoidal
  \icat{}. If $\overline{A}$ is a $\mathcal{P}$-algebra in
  $\mathcal{C}$ and $\phi \colon A \to i^{*}\overline{A}$ is a
  morphism of $\mathcal{O}$-algebras, then we say that $\phi$ exhibits
  $\overline{A}$ as the \emph{free $\mathcal{P}$-algebra} generated by
  $A$ along $i$ if for every $\mathcal{P}$-algebra $B$ the composite
  \[ \Map_{\Algn_{\mathcal{P}}(\mathcal{C})}(\overline{A}, B) \to
  \Map_{\Algn_{\mathcal{O}}(\mathcal{C})}(i^{*}\overline{A}, i^{*}B) \to
  \Map_{\Algn_{\mathcal{O}}(\mathcal{C})}(A, i^{*}B)\] is an
equivalence.
\end{defn}

\begin{lemma}\label{lem:freeadj}
  Let $i \colon \mathcal{O} \to \mathcal{P}$ be a morphism of
  \gdniopds{} and let $\mathcal{C}^{\otimes}$ be a $\Dn$-monoidal
  \icat{}. If for every $\mathcal{O}$-algebra $A$ in $\mathcal{C}$
  there exists a $\mathcal{P}$-algebra $\overline{A}$ and a morphism
  $A \to i^{*}\overline{A}$ that exhibits $\overline{A}$ as the free
  $\mathcal{P}$-algebra generated by $A$ along $i$, then the functor
  $i^{*} \colon \Alg_{\mathcal{P}}(\mathcal{C}) \to
  \Alg_{\mathcal{O}}(\mathcal{C})$ induced by composition with $i$
  admits a left adjoint $i_{!}$, such that the unit morphism $A \to
  i^{*}i_{!}A$ exhibits $i_{!}A$ as the free $\mathcal{P}$-algebra
  generated by $A$ along $i$ for all $A \in
  \Alg_{\mathcal{O}}(\mathcal{C})$.
\end{lemma}
\begin{proof}
  Apply \rcite{HTT}{Lemma 5.2.2.10} to the coCartesian fibration
  associated to $i^{*}$.
\end{proof}

\begin{defn}
  Suppose $i \colon \mathcal{O} \to \mathcal{P}$ is a morphism of
  \gdniopds{} and $\mathcal{C}^{\otimes}$ is a $\Dn$-monoidal
  \icat{}. If $\overline{A}$ is a $\mathcal{P}$-algebra in
  $\mathcal{C}$ and $\phi \colon A \to i^{*}\overline{A}$ is a
  morphism of $\mathcal{O}$-algebras, we have an induced diagram
  \[ (\mathcal{O} \times \Delta^{1}) \amalg_{\mathcal{O} \times \{1\}}
  \mathcal{P}  \to \mathcal{C}^{\otimes}.\]
  Choose a factorization of this as 
  \[ (\mathcal{O} \times \Delta^{1}) \amalg_{\mathcal{O} \times \{1\}}
  \mathcal{P}  \hookrightarrow \mathcal{M} \to \mathcal{C}^{\otimes}\]
  such that the first map is inner anodyne and $\mathcal{M}$ is a
  $\Delta^{1}$-family of \gdniopds{}. 
  We say that $\phi$ exhibits
  $\overline{A}$ as an \emph{operadic left Kan extension} of $A$ along
  $i$ if the map $\mathcal{M} \to \mathcal{C}^{\otimes}$ is an
  operadic left Kan extension.
\end{defn}

\begin{propn}\label{propn:lokeisfree}
  Suppose $i \colon \mathcal{O} \to \mathcal{P}$ is a morphism of
  \gdniopds{} and $\mathcal{C}^{\otimes}$ is a $\Dn$-monoidal
  \icat{}. If $\overline{A}$ is a $\mathcal{P}$-algebra in
  $\mathcal{C}$ and $\phi \colon A \to i^{*}\overline{A}$ is a
  morphism of $\mathcal{O}$-algebras that exhibits $\overline{A}$ as
  an operadic left Kan extension of $A$ along $i$, then $\phi$
  exhibits $\overline{A}$ as the free $\mathcal{P}$-algebra generated
  by $A$ along $i$.
\end{propn}
\begin{proof}
  As \rcite{HA}{Proposition 3.1.3.2}, using Corollary~\ref{cor:lokeprop}.
\end{proof}

\begin{cor}\label{cor:lokegivesadj}
  Let $i \colon \mathcal{O} \to \mathcal{P}$ be a morphism of
  \gdniopds{} and let $\mathcal{C}^{\otimes}$ be a $\Dn$-monoidal
  \icat{}. If for every $\mathcal{O}$-algebra $A$ in $\mathcal{C}$
  there exists a $\mathcal{P}$-algebra $\overline{A}$ and a morphism
  $A \to i^{*}\overline{A}$ that exhibits $\overline{A}$ as the
  operadic left Kan extension of $A$ along $i$, then the functor
  $i^{*} \colon \Alg_{\mathcal{P}}(\mathcal{C}) \to
  \Alg_{\mathcal{O}}(\mathcal{C})$ induced by composition with $i$
  admits a left adjoint $i_{!}$, such that the unit morphism $A \to
  i^{*}i_{!}A$ exhibits $i_{!}A$ as the operadic left Kan extension of
  $A$ along $i$ for all $A \in
  \Alg_{\mathcal{O}}(\mathcal{C})$. Moreover, if $i$ is fully
  faithful, then so is $i_{!}$.
\end{cor}
\begin{proof}
  Combine Corollary~\ref{propn:lokeisfree} with
  Lemma~\ref{lem:freeadj}. The full faithfulness follows from the
  description of operadic left Kan extensions in terms of colimits: it
  is immediate from this that if $i$ is fully faithful then the unit
  morphism $A \to i^{*}i_{!}A$ is an equivalence.
\end{proof}

\begin{defn}\label{defn:icompatible}
  Let $i \colon \mathcal{O} \to \mathcal{P}$ be an extendable morphism
  of \gdniopds{}. We say that a $\Dn$-monoidal \icat{}
  $\mathcal{C}^{\otimes}$ is \emph{$i$-compatible} if for every
  $\mathcal{O}$-algebra $A$ in $\mathcal{C}$ and every $x \in
  \mathcal{P}_{C_{n}}$, the diagram
  \[ \mathcal{O}^{\act}_{/x} \to \mathcal{O} \xto{A}
  \mathcal{C}^{\otimes}\]
  can be extended to a monoidal colimit diagram.
\end{defn}

\begin{cor}\label{cor:opdkanext}
  Let $i \colon \mathcal{O} \to \mathcal{P}$ be an extendable morphism
  of \gdniopds{}. If $\mathcal{C}^{\otimes}$ is a $\Dn$-monoidal
  \icat{} that is $i$-compatible, then the functor $i^{*} \colon
  \Alg_{\mathcal{P}}(\mathcal{C}) \to \Alg_{\mathcal{O}}(\mathcal{C})$
  admits a left adjoint $i_{!}$, such that the unit morphism $A \to
  i^{*}i_{!}A$ exhibits $i_{!}A$ as the operadic left Kan extension of
  $A$ along $i$ for all $A \in
  \Alg_{\mathcal{O}}(\mathcal{C})$.
\end{cor}
\begin{proof}
  Combine Corollary~\ref{cor:lokegivesadj} with
  Proposition~\ref{propn:lokeexist}.
\end{proof}

\begin{cor}
  Let $i \colon \mathcal{O} \to \mathcal{P}$ be an extendable morphism
  of \gdniopds{}. If $\mathcal{C}^{\otimes}$ is a $\Dn$-monoidal
  \icat{} that is compatible with $\mathcal{O}^{\act}_{/p}$-indexed
  colimits for all $p \in \mathcal{P}_{C_{n}}$, then the functor
  $i^{*} \colon \Alg_{\mathcal{P}}(\mathcal{C}) \to
  \Alg_{\mathcal{O}}(\mathcal{C})$ admits a left adjoint $i_{!}$, such
  that the unit morphism $A \to i^{*}i_{!}A$ exhibits $i_{!}A$ as the
  operadic left Kan extension of $A$ along $i$ for all $A \in
  \Alg_{\mathcal{O}}(\mathcal{C})$.
\end{cor}
\begin{proof}
  Combine Corollary~\ref{cor:lokegivesadj} with
  Corollary~\ref{cor:lokecompcolim}.
\end{proof}

We will also need an observation on the functoriality of free
algebras, requiring some terminology:
\begin{defn}\label{defn:ftricomp}
  Let $i \colon \mathcal{O} \to \mathcal{P}$ be an extendable morphism
  of \gdniopds{}. If $\mathcal{C}^{\otimes}$ and
  $\mathcal{D}^{\otimes}$ are $i$-compatible $\Dn$-monoidal \icats{},
  we say that a $\Dn$-monoidal functor $F^{\otimes} \colon
  \mathcal{C}^{\otimes} \to \mathcal{D}^{\otimes}$ is
  \emph{$i$-compatible} if for every $\mathcal{O}$-algebra $A$ in
  $\mathcal{C}$ and every $x \in \mathcal{P}_{C_{n}}$, the underlying
  functor $F \colon \mathcal{C} \to \mathcal{D}$ preserves the
  (monoidal) colimit of the diagram $\mathcal{O}^{\txt{act}}_{/x} \to
  \mathcal{C}$.
\end{defn}

\begin{lemma}\label{lem:icompftr}
  Suppose $i \colon \mathcal{O} \to \mathcal{P}$ is an extendable
  morphism of \gdniopds{}, $\mathcal{C}^{\otimes}$ and
  $\mathcal{D}^{\otimes}$ are $i$-compatible $\Dn$-monoidal \icats{},
  and $F^{\otimes} \colon \mathcal{C}^{\otimes} \to
  \mathcal{D}^{\otimes}$ is an $i$-compatible $\Dn$-monoidal
  functor. Then we have a commutative diagram
  \csquare{\Algn_{\mathcal{O}}(\mathcal{C})}{\Algn_{\mathcal{O}}(\mathcal{D})}{\Algn_{\mathcal{P}}(\mathcal{C})}{\Algn_{\mathcal{P}}(\mathcal{D}).}{F_{*}}{i_{!}}{i_{!}}{F_{*}}
\end{lemma}
\begin{proof}
  We must show that for every $\mathcal{O}$-algebra $A$ in
  $\mathcal{C}$, the map $F_{*}A \to F_{*}i^{*}i_{!}A \simeq
  i^{*}F_{*}i_{!}A$ exhibits $F_{*}i_{!}A$ as the free algebra
  generated by $F_{*}A$ along $i$. This follows from
  Proposition~\ref{propn:lokeisfree} and the assumption that $F$ is
  $i$-compatible, since this implies that $F_{*}i_{!}A$ is a
  left operadic Kan extension of $F_{*}A$.
\end{proof}

\subsection{Monoidal Properties of the Algebra
  Functor}\label{subsec:algfun}
In this subsection we observe that the Cartesian product of
\gDniopds{} leads to natural monoidal structures on \icats{} of
algebras.

\begin{defn}
  For any categorical pattern $\mathfrak{P}$, the model category
  $(\sSet^{+})_{\mathfrak{P}}$ is enriched in marked simplicial sets
  by Proposition~\ref{propn:catpattprod}. The enriched Yoneda functor
  therefore gives a right Quillen bifunctor
  \[ \mathfrak{H}_{\mathfrak{P}} \colon (\sSet^{+})_{\mathfrak{P}}^{\op} \times
  (\sSet^{+})_{\mathfrak{P}} \to \sSet^{+}.\]
  Applied to $\mathfrak{P} = \mathfrak{O}_{n}^{\txt{gen}}$, this
  induces at the level of \icats{} a functor
  \[ \Algn_{(\blank)}(\blank) \colon (\OpdIDng)^{\op} \times
  \OpdIDng \to \CatI.\]
  We write $\Alg^{n} \to (\OpdIDng)^{\op} \times
  \OpdIDng$ for an associated coCartesian fibration.
\end{defn}

\begin{defn}
  Since $(\sSet^{+})_{\mathfrak{P}}$ is a (marked simplicially
  enriched) symmetric monoidal model category with respect to the
  Cartesian product, the functor $\mathfrak{H}_{\mathfrak{P}}$ is lax
  symmetric monoidal with respect to the Cartesian product. Thus, for
  $\mathfrak{P} = \mathfrak{O}_{n}^{\txt{gen}}$ it induces on the
  level of \icats{} a lax symmetric monoidal
  functor \[((\OpdIDng)^{\op})^{\amalg} \times_{\Gop}
  (\OpdIDng)^{\times} \to \CatI^{\times},\] where we write
  $((\OpdIDng)^{\op})^{\amalg}$ for the symmetric monoidal structure
  on $(\OpdIDng)^{\op}$ given by the Cartesian product in $\OpdIDng$,
  since this is the coCartesian monoidal structure on the opposite
  \icat{}.  Using \rcite{HA}{Proposition 2.4.2.5} this corresponds to
  a functor
  \[\phi \colon ((\OpdIDng)^{\op})^{\amalg} \times_{\Gop} (\OpdIDng)^{\times} \to
  \CatI\]
  that is a $((\OpdIDng)^{\op})^{\amalg} \times_{\Gop}
  (\OpdIDng)^{\times}$-monoid in $\CatI$ (i.e. it satisfies the
  relevant Segal conditions). Let
  \[ \Alg^{n,\boxtimes} \to ((\OpdIDng)^{\op})^{\amalg} \times_{\Gop}
  (\OpdIDng)^{\times}.\] be the coCartesian fibration associated to
  $\phi$; since the functor $\phi$ satisfies the Segal conditions,
  this is a coCartesian fibration of generalized symmetric \iopds{}.
\end{defn}
This construction describes the ``external product'' that combines two
algebras $A \colon \mathcal{O} \to \mathcal{O}'$ and $B \colon
\mathcal{P} \to \mathcal{P}'$ to $A \boxtimes B := A \times_{\Dnop} B
\colon \mathcal{O}\times_{\Dnop}\mathcal{O}' \to \mathcal{P}
\times_{\Dnop} \mathcal{P}'$. Since we are considering the coCartesian
symmetric monoidal structure on $(\OpdIDng)^{\op}$, by
\rcite{HA}{Example 2.4.3.5} there is a morphism of generalized
symmetric \iopds{} $\alpha \colon \Gop \times (\OpdIDng)^{\op} \to
((\OpdIDng)^{\op})^{\amalg}$. (Informally, this takes $(\angled{n},
\mathcal{O})$ to the list $(\mathcal{O},\ldots,\mathcal{O})$ with $n$
copies of $\mathcal{O}$.) We define $\Alg^{n,\otimes}$ by the pullback
square \csquare{\Alg^{n,\otimes}}{\Alg^{n,\boxtimes}}{(\OpdIDng)^{\op}
  \times (\OpdIDng)^\times}{((\OpdIDng)^{\op})^{\amalg} \times_{\Gop}
  (\OpdIDng)^{\times}.}{}{}{}{\alpha \times_{\Gop} \id} Then the
projection $\pi \colon \Alg^{n,\otimes} \to (\OpdIDng)^{\op} \times
(\OpdIDng)^\times$ is again a coCartesian fibration of generalized symmetric
\iopds{}. Over $\mathcal{O} \in (\OpdIDng)^{\op}$ this describes the
``half-internalized'' tensor product of $\mathcal{O}$-algebras given
by, for $A \colon \mathcal{O} \to \mathcal{P}$ and $B \colon
\mathcal{O} \to \mathcal{Q}$,
\[ A \otimes B \colon \mathcal{O} \xto{\Delta} \mathcal{O}
\times_{\Dnop} \mathcal{O} \xto{A \boxtimes B} \mathcal{P} \times_{\Dnop}\mathcal{Q}.\]

The functor associated to the coCartesian fibration $\pi$ is a $(\OpdIDng)^{\op} \times
(\OpdIDng)^\times$-monoid in $\CatI$, or equivalently a lax symmetric monoidal
functor $\OpdIDng \to \Fun((\OpdIDng)^{\op}, \CatI)$. Similarly, pulling back
$\pi$ along an arbitrary functor in the first variable, we get:
\begin{propn}\label{propn:alglaxmon}
  Let $F \colon \mathcal{C} \to \OpdIDng$ be any functor of
  \icats{}. Then the functor \[\AlgDn_{F(\blank)}(\blank) \colon
  \mathcal{C}^{\op} \times \OpdIDng \to \CatI\] induces a lax symmetric
  monoidal functor $\OpdIDng \to \Fun(\mathcal{C}^{\op}, \CatI)$.\qed
\end{propn}

\begin{cor}\label{cor:AlgEnmon}
  Suppose $\mathcal{O}$ is a \gDniopd{} and $\mathcal{C}$ is an
  $\mathbb{E}_{n+m}$-monoidal \icat{}. Then
  $\AlgDn_{\mathcal{O}}(\mathcal{C})$ is an $\mathbb{E}_{m}$-monoidal
  \icat{}.
\end{cor}
\begin{proof}
  By Proposition~\ref{propn:alglaxmon}, applied to the functor
  $\{\mathcal{O}\} \to \OpdIDng$, there is a lax symmetric monoidal
  functor $\OpdIDng \to \CatI$, which sends $\mathcal{P}$ to
  $\AlgDn_{\mathcal{O}}(\mathcal{P})$. The forgetful functor
  $\MonI^{\Dn} \to \OpdIDng$ preserves products, so we get a lax
  symmetric monoidal functor $\MonI^{\Dn} \to \CatI$, and hence a
  functor
  \[ \MonI^{\Sigma,\mathbb{E}_{n+m}} \simeq
  \AlgS_{\mathbb{E}_{n+m}}(\CatI) \simeq \AlgS_{\mathbb{E}_{m}}(\MonI^{\Dn})
  \to \AlgS_{\mathbb{E}_{m}}(\CatI) \simeq
  \MonI^{\Sigma,\mathbb{E}_{m}},\] which sends an
  $\mathbb{E}_{n+m}$-monoidal \icat{} $\mathcal{C}$ to a natural
  $\mathbb{E}_{m}$-monoidal structure on
  $\AlgDn_{\mathcal{O}}(\mathcal{C})$.
\end{proof}

\begin{remark}\label{rmk:opdkanextprod}
  Let $i \colon \mathcal{O} \to \mathcal{P}$ be an extendable morphism
  of \gdniopds{}, and let $\mathcal{C}^{\otimes}$ and
  $\mathcal{D}^{\otimes}$ be $i$-compatible $\Dn$-monoidal
  \icats{}. If the \icats{} $\mathcal{O}^{\txt{act}}_{/P}$ are all
  sifted, then the description of free algebras in
  terms of weak operadic colimits implies that there is a commutative diagram
  \[\ltikzcd{
    \Alg^{n}_{\mathcal{O}}(\mathcal{C}) \times
    \Alg^{n}_{\mathcal{O}}(\mathcal{D}) \arrow{r} \arrow{d}{i_{!}
      \times i_{!}} \pgfmatrixnextcell \Alg^{n}_{\mathcal{O} \times_{\Dnop}
    \mathcal{O}}(\mathcal{C} \times \mathcal{D}) \arrow{r}
  \arrow{d}{(i \times_{\Dnop} i)_{!}} \pgfmatrixnextcell
  \Alg^{n}_{\mathcal{O}}(\mathcal{C} \times \mathcal{D})\arrow{d}{i_{!}} \\
    \Alg^{n}_{\mathcal{P}}(\mathcal{C}) \times
    \Alg^{n}_{\mathcal{P}}(\mathcal{D}) \arrow{r}  \pgfmatrixnextcell \Alg^{n}_{\mathcal{P} \times_{\Dnop}
    \mathcal{P}}(\mathcal{C} \times \mathcal{D}) \arrow{r} \pgfmatrixnextcell
  \Alg^{n}_{\mathcal{P}}(\mathcal{C} \times \mathcal{D}).
}\] 
In other words, $i_{!}(A \otimes B) \simeq i_{!}A
\otimes i_{!}B$ where $\otimes$ denotes the ``half-internalized''
tensor product of algebras. If $\mathcal{C}$ is a
$\simp^{n+1}$-monoidal \icat{} such that its tensor product, regarded
as a $\simp^{n}$-monoidal functor $\mathcal{C}^{\otimes} \times_{\Dnop}
\mathcal{C}^{\otimes} \to \mathcal{C}^{\otimes}$, is $i$-compatible,
then by Lemma~\ref{lem:icompftr} we get a commutative square
\csquare{\Alg^{n}_{\mathcal{O}}(\mathcal{C}) \times
    \Alg^{n}_{\mathcal{O}}(\mathcal{C})}{\Alg^{n}_{\mathcal{O}}(\mathcal{C})}{\Alg^{n}_{\mathcal{P}}(\mathcal{C}) \times
    \Alg^{n}_{\mathcal{P}}(\mathcal{C})}{\Alg^{n}_{\mathcal{P}}(\mathcal{C}).}{\otimes}{i_{!}
    \times i_{!}}{i_{!}}{\otimes}
\end{remark}

\subsection{$\Dn$-uple Envelopes}\label{subsec:env}
It is immediate from the definition of the model categories
$(\sSet^{+})_{\mathfrak{O}^{\txt{gen}}_{n}}$ and
$(\sSet^{+})_{\mathfrak{U}_{n}}$ that the identity is a left Quillen
functor $(\sSet^{+})_{\mathfrak{O}^{\txt{gen}}_{n}} \to
(\sSet^{+})_{\mathfrak{U}_{n}}$. On the level of \icats{}, this means
that the inclusion $\txt{Upl}_{\infty}^{n} \to \OpdIDng$ has a left
adjoint. In this subsection we observe that the arguments of
\rcite{HA}{\S 2.2.4} give an explicit description of this left
adjoint.

\begin{defn}
  Let $\txt{Act}(\Dnop)$ be the full subcategory of
  $\Fun(\Delta^{1}, \Dnop)$ spanned by the active morphisms. If
  $\mathcal{M}$ is a \gDniopd{}, we define $\txt{Env}_{n}(\mathcal{M})$ to
  be the fibre product
  \[ \mathcal{M} \times_{\Fun(\{0\}, \Dnop)}
  \txt{Act}(\Dnop).\]
\end{defn}

We will refer to $\txt{Env}_{n}(\mathcal{M})$ as the \emph{$\Dn$-uple
  envelope} of $\mathcal{M}$ --- this terminology is justified by the
next results:
\begin{propn}\label{propn:EnvDbl}
  The map $\txt{Env}_{n}(\mathcal{M}) \to \Dnop$ induced
  by evaluation at $1$ in $\Delta^{1}$ is a $\Dn$-uple \icat{}.
\end{propn}
\begin{proof}
  As \cite[Proposition 2.2.4.4]{HA}.
\end{proof}

\begin{propn}\label{propn:EnvAdj}
  Suppose $\mathcal{N}$ is a $\Dn$-uple \icat{} and $\mathcal{M}$ a
  \gDniopd{}. The inclusion $\mathcal{M} \to
  \txt{Env}_{n}(\mathcal{M})$ induces an equivalence
  \[ \Fun^{\otimes,n}(\txt{Env}_{n}(\mathcal{M}), \mathcal{N}) \to
  \Alg^{n}_{\mathcal{M}}(\mathcal{N}).\]
\end{propn}
\begin{proof}
  As \cite[Proposition 2.2.4.9]{HA}.
\end{proof}

\begin{lemma}\label{lem:envprod}
  Suppose $\mathcal{O}$ is a \gDniopd{} and $\mathcal{P}$ is a
  generalized $\simp^{m}$-\iopd{}. There is a natural equivalence
  \[\txt{Env}_{n}(\mathcal{O}) \times \txt{Env}_{m}(\mathcal{P}) \simeq
  \txt{Env}_{n+m}(\mathcal{O} \times \mathcal{P}).\]
\end{lemma}
\begin{proof}
  This is immediate from the definition.
\end{proof}

\subsection{The Internal Hom}
In this subsection we observe, following \rcite{BarwickOpCat}{\S 9},
that if $\mathcal{O}$ is a \gDniopd{} and $\mathcal{P}$ is a
generalized $\simp^{m+n}$-\iopd{} then the \icat{}
$\Algn_{\mathcal{O}}(\mathcal{P})$ has a natural generalized $\simp^{m}$-\iopd{}
structure. When $\mathcal{P}$ is a $\simp^{m+n}$-monoidal \icat{} we
will prove that this makes $\Algn_{\mathcal{O}}(\mathcal{P})$ a
$\simp^{m}$-monoidal \icat{}, and that this structure agrees with that
we described in \S\ref{subsec:algfun}.
\begin{defn}
  By Corollary~\ref{cor:CartProdQuillenIOpd}, the Cartesian product gives a left
  Quillen bifunctor
  \[ (\sSet^{+})_{\mathfrak{O}_{n}^{\txt{gen}}} \times
  (\sSet^{+})_{\mathfrak{O}_{m}^{\txt{gen}}} \to
  (\sSet^{+})_{\mathfrak{O}_{n+m}^{\txt{gen}}}.\]
  It therefore induces a right Quillen bifunctor
  \[ \ALG_{(\blank)}^{n,m}(\blank) \colon
  (\sSet^{+})_{\mathfrak{O}_{n}^{\txt{gen}}}^{\op} \times
  (\sSet^{+})_{\mathfrak{O}_{m+n}^{\txt{gen}}} \to
  (\sSet^{+})_{\mathfrak{O}_{m}^{\txt{gen}}}.\]
  Similarly, there is a right Quillen bifunctor
  \[ \txt{FUN}^{\otimes,n,m}_{(\blank)}(\blank) \colon 
  (\sSet^{+})_{\mathfrak{U}_{n}}^{\op} \times
  (\sSet^{+})_{\mathfrak{U}_{m+n}} \to
  (\sSet^{+})_{\mathfrak{U}_{m}},\]
  right adjoint to the Cartesian product.
\end{defn}

On the level of \icats{}, these right Quillen bifunctors induce
functors
  \[\ALG_{(\blank)}^{n,m}(\blank) \colon \OpdIDng \times
  \OpdI^{\simp^{n+m},\txt{gen}} \to \OpdI^{\simp^{m},\txt{gen}},\]
  \[\txt{FUN}_{(\blank)}^{\otimes,n,m}(\blank) \colon \txt{Upl}_{\infty}^{\Dn} \times
  \txt{Upl}_{\infty}^{\simp^{n+m}}  \to
  \txt{Upl}_{\infty}^{\simp^{m}},\]
with the universal property that there are natural equivalences of \icats{}
  \[  \Alg_{\mathcal{O}}^{m}(\ALG_{\mathcal{P}}^{n,m}(\mathcal{Q}))
  \simeq \Alg_{\mathcal{O} \times \mathcal{P}}^{n+m}(\mathcal{Q}),\]
where $\mathcal{O}$ is a generalized $\simp^{m}$-\iopd{},
$\mathcal{P}$ is a \gDniopd{}, and $\mathcal{Q}$ is a generalized
$\simp^{m+n}$-\iopd{}, and
\[ \Fun^{\otimes,m}(\mathcal{L}, \txt{FUN}^{\otimes,n,m}(\mathcal{M},
\mathcal{N})) \simeq \Fun^{\otimes,m+n}(\mathcal{L} \times
\mathcal{M}, \mathcal{N}),\] where $\mathcal{L}$ is a $\simp^{m}$-uple
\icat{}, $\mathcal{M}$ is a $\Dn$-uple \icat{}, and $\mathcal{N}$ is a
$\simp^{m+n}$-uple \icat{}.

\begin{lemma}\label{lem:ALGopdFUNmon}\ 
  \begin{enumerate}[(i)]
  \item If $\mathcal{O}$ is a $\simp^{n+m}$-\iopd{}, then
    $\ALG^{n,m}_{\mathcal{M}}(\mathcal{O})$ is a \dniopd{} for any
    \gdniopd{} $\mathcal{M}$.
  \item If $\mathcal{C}^{\otimes}$ is a $\simp^{n+m}$-monoidal
    \icat{}, then $\FUN^{\otimes,n,m}(\mathcal{M},
    \mathcal{C}^{\otimes})$ is a $\simp^{m}$-monoidal \icat{} for any
    $\Dn$-uple \icat{} $\mathcal{M}$.
  \end{enumerate}
\end{lemma}
\begin{proof}
  We will prove (i); the proof of (ii) is similar. Suppose $C_{S} \neq
  C_{n}$ is
  a cell of $\simp^{m,\op}$. Then we have
  \[ \ALG^{n,m}_{\mathcal{M}}(\mathcal{O})_{C_{S}} \simeq
  \Alg^{m}_{\{C_{S}\}}(\ALG^{n,m}_{\mathcal{M}}(\mathcal{O})) \simeq
  \Alg^{n+m}_{\{C_{S}\} \times \mathcal{M}}(\mathcal{O}),\] which is
  contractible if $\mathcal{O}$ is a $\simp^{m+n}$-\iopd{}.
\end{proof}

\begin{lemma}\label{lem:ALGFUNeq}
  Suppose $\mathcal{M}$ is a $\simp^{n+m}$-uple \icat{}. Then there is
  a natural equivalence
  \[ \ALG^{n,m}_{\mathcal{O}}(\mathcal{M}) \simeq
  \txt{FUN}^{\otimes,n,m}(\txt{Env}_{n}(\mathcal{O}), \mathcal{M}).\]
  for all \gdniopds{} $\mathcal{O}$. In particular,
  $\ALG^{n,m}_{\mathcal{O}}(\mathcal{M})$ is a $\simp^{m}$-uple
  \icat{}.
\end{lemma}
\begin{proof}
  Using Lemma~\ref{lem:envprod}, we have natural equivalences
  \[
  \begin{split}
    \Map_{\OpdI^{\simp^{m},\txt{gen}}}(\mathcal{P}, \ALG^{n,m}_{\mathcal{O}}(\mathcal{M})) &  \simeq
  \Map_{\OpdI^{\simp^{m+n},\txt{gen}}}(\mathcal{P} \times \mathcal{O},
  \mathcal{M}) \\ & \simeq
  \Map_{\txt{Upl}_{\infty}^{\simp^{n+m}}}(\txt{Env}_{n+m}(\mathcal{P}
  \times \mathcal{O}), \mathcal{M}) \\
  & \simeq
  \Map_{\txt{Upl}_{\infty}^{\simp^{n+m}}}(\txt{Env}_{m}(\mathcal{P}) \times \txt{Env}_{n}(\mathcal{O}),
  \mathcal{M})\\ & \simeq
  \Map_{\txt{Upl}_{\infty}^{\simp^{m}}}(\txt{Env}_{m}(\mathcal{P}),
  \txt{FUN}^{\otimes,n,m}(\txt{Env}_{n}(\mathcal{O}), \mathcal{M})) \\
  & \simeq
  \Map_{\OpdI^{\simp^{m},\txt{gen}}}(\mathcal{P}, \txt{FUN}^{\otimes,n,m}(\txt{Env}_{n}(\mathcal{O}),
  \mathcal{M})).\qedhere
  \end{split}
 \]
\end{proof}

If $\mathcal{C}^{\otimes}$ is a $\simp^{n+m}$-monoidal \icat{},
combining Lemmas~\ref{lem:ALGopdFUNmon} and \ref{lem:ALGFUNeq} we see
that $\ALG^{n,m}_{\mathcal{O}}(\mathcal{C})$ is a $\simp^{m}$-monoidal
\icat{} for any \gDniopd{} $\mathcal{O}$; the underlying \icat{} of
this is $\Alg^{n}_{\mathcal{O}}(\mathcal{C})$. On the other
hand, we saw in Corollary~\ref{cor:AlgEnmon} that
$\Alg^{n}_{\mathcal{O}}(\mathcal{C})$ inherits an
$\mathbb{E}_{m}$-monoidal structure from the lax monoidal
functoriality of $\Alg^{n}_{\mathcal{O}}(\blank)$; let us denote the
resulting $\simp^{m}$-monoidal \icat{} by
$\Alg^{n,\otimes}_{\mathcal{O}}(\mathcal{C})$. We will now show that
these two $\mathbb{E}_{m}$-monoidal structures agree:
\begin{propn}
  Let $\mathcal{C}^{\otimes}$ be a $\simp^{n+m}$-monoidal \icat{},
  $\mathcal{O}$ a \gDniopd{} and $\mathcal{M}$ a $\simp^{m}$-uple
  \icat{}. Then we have a natural equivalence
  \[ \Map_{\txt{Upl}_{\infty}^{\simp^{m}}}(\mathcal{M},
  \Alg^{n,\otimes}_{\mathcal{O}}(\mathcal{C})) \simeq
  \Map_{\txt{Upl}_{\infty}^{\simp^{m+n}}}(\mathcal{M} \times \txt{Env}_{n}(\mathcal{O}), \mathcal{C}^{\otimes}).\]
\end{propn}
\begin{proof}
  We may identify $\txt{Upl}_{\infty}^{\simp^{m}}$ with a full
  subcategory of the \icat{} of coCartesian fibrations over
  $\simp^{m,\op}$, which is equivalent to $\Fun(\simp^{m,\op},
  \CatI)$; under this equivalence $\mathcal{M}$ corresponds to a
  functor $\mu \colon \simp^{m,\op} \to \CatI$. If $\gamma \colon
  \simp^{m,\op} \to \MonI^{\simp^{n}}$ is the $\simp^{m}$-monoid
  corresponding to $\mathcal{C}^{\otimes}$, then we have a
  natural equivalence
  \[
  \begin{split}
    \Map_{\txt{Upl}_{\infty}^{\simp^{m}}}(\mathcal{M},
  \Alg^{n,\otimes}_{\mathcal{O}}(\mathcal{C})) &  \simeq
  \Map_{\Fun(\simp^{m,\op}, \CatI)}(\mu,
  \Alg^{n}_{\mathcal{O}}(\gamma)) \\
  & \simeq \Map_{\Fun(\simp^{m,\op}, \CatI)}(\mu,
  \Fun^{\otimes,n}(\txt{Env}_{n}(\mathcal{O}), \gamma))\\ & \simeq
  \Map_{\Fun(\simp^{m,\op}, \Cat_{\infty/\Dnop}^{\txt{cocart}})}(\mu
  \times \txt{Env}_{n}(\mathcal{O}), \gamma)\\ & \simeq
  \Map_{\txt{Upl}_{\infty}^{\simp^{n+m}}}(\mathcal{M} \times
  \txt{Env}_{n}(\mathcal{O}), \mathcal{C}^{\otimes}). \qedhere
  \end{split}\]
\end{proof}

Combining this with Lemma~\ref{lem:ALGFUNeq}, we get:
\begin{cor}\label{cor:inthomeq}
  Let $\mathcal{C}^{\otimes}$ be a $\simp^{n+m}$-monoidal \icat{} and
  $\mathcal{O}$ a \gDniopd{}. Then the $\mathbb{E}_{m}$-monoidal
  \icats{} $\ALG^{n,m}_{\mathcal{O}}(\mathcal{C})$ and
  $\Alg^{n,\otimes}_{\mathcal{O}}(\mathcal{C})$ are naturally equivalent.\qed
\end{cor}

\bibliographystyle{gtart}

\begin{thebibliography}{}
\providecommand\bibmarginpar{\leavevmode\marginpar}
\def\urlstyle#1{{\tt #1}}

\bibitem{AndradeThesis}
\textbf{R Andrade}, \emph{From manifolds to invariants of $E_n$-algebras}
  (2012) \xox{arXiv}{1210.7909}

\bibitem{AntieauGepnerBrauer}
\textbf{B Antieau}, \textbf{D Gepner},
  \href{http://dx.doi.org/10.2140/gt.2014.18.1149} {\emph{Brauer groups and
  \'etale cohomology in derived algebraic geometry}}, Geom. Topol. 18 (2014)
  1149--1244 \xox{arXiv}{1210.0290}

\bibitem{AtiyahTQFT}
\textbf{M Atiyah}, \href{http://www.numdam.org/item?id=PMIHES_1988__68__175_0}
  {\emph{Topological quantum field theories}}, Inst. Hautes \'Etudes Sci. Publ.
  Math.  (1988) 175--186 (1989)

\bibitem{AyalaFrancisRozenblyumFactHlgy}
\textbf{D Ayala}, \textbf{J Francis}, \textbf{N Rozenblyum},
  \emph{Factorization homology from higher categories} (2015)
  \xox{arXiv}{1504.04007}

\bibitem{AyalaFrancisTanakaFactHlgy}
\textbf{D Ayala}, \textbf{J Francis}, \textbf{H\,L Tanaka}, \emph{Factorization
  homology of stratified spaces} (2014) \xox{arXiv}{1409.0848}

\bibitem{AyalaHepworthThetan}
\textbf{D Ayala}, \textbf{R Hepworth},
  \href{http://dx.doi.org/10.1090/S0002-9939-2014-11946-0} {\emph{Configuration
  spaces and {$\Theta_n$}}}, Proc. Amer. Math. Soc. 142 (2014) 2243--2254

\bibitem{BaezDolanTQFT}
\textbf{J\,C Baez}, \textbf{J Dolan}, \href{http://dx.doi.org/10.1063/1.531236}
  {\emph{Higher-dimensional algebra and topological quantum field theory}}, J.
  Math. Phys. 36 (1995) 6073--6105

\bibitem{BakerRichterSzymik}
\textbf{A Baker}, \textbf{B Richter}, \textbf{M Szymik},
  \href{http://dx.doi.org/10.1016/j.jpaa.2012.03.001} {\emph{Brauer groups for
  commutative {$S$}-algebras}}, J. Pure Appl. Algebra 216 (2012) 2361--2376

\bibitem{BarwickThesis}
\textbf{C Barwick}, \emph{$(\infty, n)$-{C}at as a closed model category}, PhD
  thesis, University of Pennsylvania (2005)

\bibitem{BarwickOpCat}
\textbf{C Barwick}, \emph{From operator categories to topological operads}
  (2013) \xox{arXiv}{1302.5756}

\bibitem{CalaqueScheimbauerCob}
\textbf{D Calaque}, \textbf{C Scheimbauer}, \emph{A note on the
  $(\infty,n)$-category of cobordisms} (2015) \xox{arXiv}{1509.08906}

\bibitem{CruttwellShulman}
\textbf{G\,S\,H Cruttwell}, \textbf{M\,A Shulman}, \emph{A unified framework
  for generalized multicategories}, Theory Appl. Categ. 24 (2010) No. 21,
  580--655

\bibitem{FrancisTgtCx}
\textbf{J Francis}, \href{http://dx.doi.org/10.1112/S0010437X12000140}
  {\emph{The tangent complex and {H}ochschild cohomology of {$E_n$}-rings}},
  Compos. Math. 149 (2013) 430--480

\bibitem{FreedHigherAlg}
\textbf{D\,S Freed}, \href{http://projecteuclid.org/euclid.cmp/1104254603}
  {\emph{Higher algebraic structures and quantization}}, Comm. Math. Phys. 159
  (1994) 343--398

\bibitem{GarnerGurskiTricat}
\textbf{R Garner}, \textbf{N Gurski},
  \href{http://dx.doi.org/10.1017/S0305004108002132} {\emph{The low-dimensional
  structures formed by tricategories}}, Math. Proc. Cambridge Philos. Soc. 146
  (2009) 551--589

\bibitem{enr}
\textbf{D Gepner}, \textbf{R Haugseng},
  \href{http://dx.doi.org/10.1016/j.aim.2015.02.007} {\emph{Enriched
  {$\infty$}-categories via non-symmetric {$\infty$}-operads}}, Adv. Math. 279
  (2015) 575--716 \xox{arXiv}{1312.3178}

\bibitem{freepres}
\textbf{D Gepner}, \textbf{R Haugseng}, \textbf{T Nikolaus}, \emph{Lax colimits
  and free fibrations in $\infty$-categories} (2015) \xox{arXiv}{1501.02161}

\bibitem{GinotFactNotes}
\textbf{G Ginot}, \emph{Notes on factorization algebras, factorization homology
  and applications}, from: ``Mathematical aspects of quantum field theories'',
  Math. Phys. Stud., Springer, Cham (2015)  429--552 \xox{arXiv}{1307.5213}

\bibitem{SGA1}
\textbf{A Grothendieck}, \emph{Rev\^etements \'etales et groupe fondamental.
  {F}asc. {I}: {E}xpos\'es 1 \`a 5}, volume 1960/61 of \emph{S\'eminaire de
  G\'eom\'etrie Alg\'ebrique}, Institut des Hautes \'Etudes Scientifiques,
  Paris (1963)

\bibitem{spans}
\textbf{R Haugseng}, \emph{Iterated spans and ``classical'' topological field
  theories} (2014) \xox{arXiv}{1409.0837}

\bibitem{HopkinsMahowaldSadofsky}
\textbf{M\,J Hopkins}, \textbf{M Mahowald}, \textbf{H Sadofsky},
  \href{http://dx.doi.org/10.1090/conm/158/01454} {\emph{Constructions of
  elements in {P}icard groups}}, from: ``Topology and representation theory
  ({E}vanston, {IL}, 1992)'', Contemp. Math. 158, Amer. Math. Soc., Providence,
  RI (1994)  89--126

\bibitem{JohnsonFreydScheimbauerLax}
\textbf{T Johnson-Freyd}, \textbf{C Scheimbauer}, \emph{(Op)lax natural
  transformations, twisted quantum field theories, and ``even higher'' Morita
  categories} (2015) \xox{arXiv}{1502.06526}

\bibitem{JohnsonDerMor}
\textbf{N Johnson}, \emph{Morita Theory For Derived Categories: A Bicategorical
  Perspective} (2008) \xox{arXiv}{0805.3673}

\bibitem{JoyalQCNotes}
\textbf{A Joyal}, \emph{Notes on Quasicategories}
\ Available at \setbox0\hbox{\makeatletter\@url
{http://www.math.uchicago.edu/~may/IMA/Joyal.pdf}}
\href{http://www.math.uchicago.edu/~may/IMA/Joyal.pdf}
{\unhbox0}

\bibitem{JoyalTierney}
\textbf{A Joyal}, \textbf{M Tierney},
  \href{http://dx.doi.org/10.1090/conm/431/08278} {\emph{Quasi-categories vs
  {S}egal spaces}}, from: ``Categories in algebra, geometry and mathematical
  physics'', Contemp. Math. 431, Amer. Math. Soc., Providence, RI (2007)
  277--326

\bibitem{KapustinICM}
\textbf{A Kapustin}, \emph{Topological field theory, higher categories, and
  their applications}, from: ``Proceedings of the {I}nternational {C}ongress of
  {M}athematicians. {V}olume {III}'', Hindustan Book Agency, New Delhi (2010)
  2021--2043

\bibitem{LawrenceTQFT}
\textbf{R\,J Lawrence}, \href{http://dx.doi.org/10.1142/9789812796387_0011}
  {\emph{Triangulations, categories and extended topological field theories}},
  from: ``Quantum topology'', Ser. Knots Everything 3, World Sci. Publ., River
  Edge, NJ (1993)  191--208

\bibitem{LeinsterHigherOpds}
\textbf{T Leinster}, \href{http://dx.doi.org/10.1017/CBO9780511525896}
  {\emph{Higher operads, higher categories}}, volume 298 of \emph{London
  Mathematical Society Lecture Note Series}, Cambridge University Press,
  Cambridge (2004)

\bibitem{HTT}
\textbf{J Lurie}, \href{http://math.harvard.edu/~lurie} {\emph{Higher topos
  theory}}, volume 170 of \emph{Annals of Mathematics Studies}, Princeton
  University Press, Princeton, NJ (2009)

\bibitem{LurieGoodwillie}
\textbf{J Lurie}, \emph{($\infty $,2)-Categories and the {G}oodwillie Calculus
  {I}} (2009)
\ Available at \setbox0\hbox{\makeatletter\@url
{http://math.harvard.edu/~lurie/}}
\href{http://math.harvard.edu/~lurie/}
{\unhbox0}

\bibitem{LurieCob}
\textbf{J Lurie}, \emph{On the classification of topological field theories},
  from: ``Current developments in mathematics, 2008'', Int. Press, Somerville,
  MA (2009)  129--280

\bibitem{HA}
\textbf{J Lurie}, \emph{Higher Algebra} (2014)
\ Available at \setbox0\hbox{\makeatletter\@url
{http://math.harvard.edu/~lurie/}}
\href{http://math.harvard.edu/~lurie/}
{\unhbox0}

\bibitem{MathewStojanoska}
\textbf{A Mathew}, \textbf{V Stojanoska}, \emph{The Picard group of topological
  modular forms via descent theory} (2014) \xox{arXiv}{1409.7702}

\bibitem{MayPic}
\textbf{J\,P May}, \href{http://dx.doi.org/10.1006/aima.2001.1996}
  {\emph{Picard groups, {G}rothendieck rings, and {B}urnside rings of
  categories}}, Adv. Math. 163 (2001) 1--16

\bibitem{MorrisonWalkerBlob}
\textbf{S Morrison}, \textbf{K Walker},
  \href{http://dx.doi.org/10.2140/gt.2012.16.1481} {\emph{Blob homology}},
  Geom. Topol. 16 (2012) 1481--1607

\bibitem{RezkCSS}
\textbf{C Rezk}, \href{http://dx.doi.org/10.1090/S0002-9947-00-02653-2}
  {\emph{A model for the homotopy theory of homotopy theory}}, Trans. Amer.
  Math. Soc. 353 (2001) 973--1007 (electronic)

\bibitem{ScheimbauerThesis}
\textbf{C Scheimbauer}, \emph{Factorization homology as a fully extended
  topological field theory}, PhD thesis, Eidgenössische Technische Hochschule,
  Zürich (2014)
\ Available at \setbox0\hbox{\makeatletter\@url
{http://guests.mpim-bonn.mpg.de/scheimbauer/}}
\href{http://guests.mpim-bonn.mpg.de/scheimbauer/}
{\unhbox0}

\bibitem{SchommerPries2TQFT}
\textbf{C\,J Schommer-Pries}, \emph{The Classification of Two-Dimensional
  Extended Topological Field Theories} (2011) \xox{arXiv}{1112.1000}

\bibitem{SegalCatCohlgy}
\textbf{G Segal}, \emph{Categories and cohomology theories}, Topology 13 (1974)
  293--312

\bibitem{ShulmanSymMonBicat}
\textbf{M\,A Shulman}, \emph{Constructing symmetric monoidal bicategories}
  (2010) \xox{arXiv}{1004.0993}

\bibitem{SzymikBrauer}
\textbf{M Szymik}, \emph{Brauer spaces for commutative rings and structured
  ring spectra} (2011) \xox{arXiv}{1110.2956}

\bibitem{ToenAzumaya}
\textbf{B To{\"e}n}, \href{http://dx.doi.org/10.1007/s00222-011-0372-1}
  {\emph{Derived {A}zumaya algebras and generators for twisted derived
  categories}}, Invent. Math. 189 (2012) 581--652

\end{thebibliography}

\end{document}